\documentclass[leqno]{article}

\usepackage{amsmath,amsthm,verbatim}
\usepackage[psamsfonts]{amssymb}

\newtheorem{theorem}{Theorem}[section]
\newtheorem{proposition}[theorem]{Proposition}
\newtheorem{lemma}[theorem]{Lemma}
\newtheorem{corollary}[theorem]{Corollary}

\theoremstyle{remark}
\newtheorem{remark}[theorem]{Remark}
\newtheorem{example}[theorem]{Example}
\newtheorem{definition}[theorem]{Definition}

\date{}

\begin{document}
\author{\textsc{M. Castrill\'on L\'opez} \\
Instituto de Ciencias Matem\'aticas CSIC-UAM-UC3M-UCM \\
Departamento de Geometr\'{\i}a y Topolog\'{\i}a \\
Facultad de Matem\'aticas, UCM \\
Avda.\ Complutense s/n, 28040-Madrid, Spain \\
\emph{E-mail:\/} \texttt{mcastri@mat.ucm.es}
\and \textsc{V. Fern\'andez Mateos\thanks{
The second author is deceased (March 20, 2011).
A Memorial Seminar was delivered at Facultad
de Ciencias Matem\'aticas, Universidad Complutense
de Madrid, on May 13, 2011; also see,
http://www.mat.ucm.es/geomfis/HomenajeVictor.html.
We are proud of having been able to work with such
an enthusiastic and generous young researcher as
V\'{\i}ctor.}, J. Mu\~{n}oz Masqu\'e} \\
Instituto de Seguridad de la Informaci\'on, CSIC \\
C/ Serrano 144, 28006-Madrid, Spain \\
\emph{E-mail:\/} \texttt{victor.fernandez@iec.csic.es} \\
\emph{E-mail:\/} \texttt{jaime@iec.csic.es}}
\title{\textbf{The Equivalence Problem of Curves } \\
\textbf{in a Riemannian Manifold}}
\date{}
\maketitle
\tableofcontents

\begin{abstract}
The equivalence problem of curves with values
in a Riemannian manifold, is solved. The domain of validity
of Frenet's theorem is shown to be the spaces of constant
curvature. For a general Riemannian manifold new invariants
must thus be added.

There are two important generic classes of curves; namely,
Frenet curves and a new class, called curves
``in normal position''. They coincide in dimensions
$\leq 4$ only.

A sharp bound for asymptotic stability
of differential invariants is obtained, the complete systems
of invariants are characterized, and a procedure of generation
is presented. Different classes of examples (specially
in low dimensions) are analyzed in detail.
\end{abstract}


\medskip


\medskip

\noindent \emph{Acknowledgements:\/} Supported by Ministerio
de Educaci\'on y Ciencia of Spain under grants \#MTM2011--22528
and \#MTM2010--19111.

\section{Introduction}
A fundamental problem in Riemannian Geometry is
that of equivalence of objects in a determined class, namely,
to provide a criterion to know whether two given objects
in this class are congruent under isometries or not. Below,
this problem is solved in full generality for the simplest
case: That of curves with values in a Riemannian manifold.

For the Euclidean space $\mathbb{R}^m $ the equivalence problem
is solved by virtue of the Frenet the theorem: Two curves
parametrized by the arc-length are congruent if and only if
they have the same curvatures, $\kappa _1,\dotsc,\kappa _{m-1}$;
but the domain of validity of Frenet's theorem is too restrictive.
In fact, Theorem \ref{CTE} states that Frenet's theorem
classifies curves in a Riemannian manifold $(M,g)$ if and only if
it is of constant curvature. In consequence, in spaces
of non-constant curvature new invariants are required
(different from curvatures $\kappa _i$) to classify
curves and, although by means of curvatures a given curve
can be reconstructed (see Theorem \ref{ecsfrenet}),
the role of such invariants becomes weaker in spaces
of non-constant curvature, even of low dimension
(see Theorem \ref{3_isometries}).
A generic Riemanniannian metric in a compact manifold
admits no isometry other than the identity map (cf.\
\cite{Ebin}, \cite{Ebin2}). Therefore, the difficulty
of the equivalence problem is closely related to the size
of the isometry group.

Below, the equivalence problem is solved in general
by means of functions that are invariant under
the isometry group of the Riemannian manifold.
In this way, a totally general result
is stated for their solution in Theorem \ref{CGC}:
Essentially, it gives a set of invariants that,
together with the classical Frenet curvatures,
solves the congruence problem; but it has
the inconvenience of using a redundant number
of invariants (cf.\ Remark \ref{remark_2},
Theorem \ref{surfaces}). The remarks following
its proof (see section \ref{remarks_CGC}) show,
however, that this is the best general result
that could be expected. Furthermore, certain classes
of Riemannian manifolds can be characterized
by means of their invariants; e.g., symmetric spaces
(cf.\ Theorem \ref{locSIM}) or Lie groups
with invariant metrics
(cf.\ Proposition \ref{proposition_independent_vectors}).

The study of invariants is developed in section
\ref{Differential_invariants}, where the main questions
on such functions are solved for an arbitrary Riemannian
manifold: The theorem of asymptotic stability
(Theorem \ref{stability} and Corollary \ref{corollary_stability}),
the completeness theorem (Theorem \ref{completeness})
that allows us to solve the general problem of equivalence
by means of a complete system of invariants
and the theorem of generation of invariants
(Theorem \ref{Generating_Invariants}). An interesting consequence
of the generating theorem proves that the ring of invariants
can be generated by means of $m$ invariants (where $m=\dim M$)
by taking successive total derivatives with respect to $t$.

In \cite{Green} the number of differential invariants
with respect to the induced operation of the group $G$
on jet bundles $J^r(\mathbb{R},G/H)$ of the homogeneous
space $G/H$, is calculated without assuming
$G$ is the group of isometries of a metric. According
to \cite [IV, Example 1.3]{KN}, if the subgroup $H$
is compact, the quotient manifold $ G/H$ admits a Riemannian
metric left invariant. The converse statement also holds true,
as the isotropy subgroup of a point in a Riemannian manifold
is compact (cf.\ \cite [I, Corollary 4.8]{KN}).
Moreover, it should be noted that most part of the results
in \cite{Green} hold in general, i.e., without assuming
the manifold to be Riemannian homogeneous, as shown
in Theorem \ref{Generating_Invariants} and Remark
\ref{remark_Green} below.

In section \ref{existence_ths} two basic existence theorems
for the generic class $\mathcal{F}$ of Frenet curves are stated.
The first result (Theorem \ref{ecsfrenet}) is a generalization
to arbitrary Riemannian manifolds of the existence of curves
in Euclidean $3$-space with given curvature and torsion,
but the second one (Theorem \ref{ecsfrenet_bis}) is
completely new.

In addition to Frenet curves, another generic class
$\mathcal{N}$ of curves in a Riemannian manifold is introduced
in Definition \ref{normal_g_position}, which seems
to be the natural setting for the statement of the asymptotic
stability theorem (Theorem \ref{stability}). The classes $\mathcal{F}$
and $\mathcal{N}$ are compared in detail in section \ref{FandN}.
As it is proved in Theorem \ref{comparingFandN}, if either
$\dim M=m\leq 4$ or $g$ is flat at a neighbourhood of $x_0$, then
$\mathcal{F}_{t_0,x_0}^{m-1}(M)=\mathcal{N}_{t_0,x_0}^{m-1}(M)$,
$(t_0,x_0)\in \mathbb{R}\times M$. In general,
however, both generic sets of curves do not coincide, as shown in
the examples \ref{m=5} and \ref{m=6}.

Finally, we illustrate all these results by studying several important
examples in detail: Theorem \ref{Euclidean_invariants} summarizes
the structure of invariants on the Euclidean space and, in particular,
it shows that the curvatures constitute a basis of invariants
on the full set of Frenet curves; in section \ref{FewIsom}
the Riemannian manifolds $(M,g)$ with Killing algebra $\mathfrak{i}(M,g)$
such that $\dim \mathfrak{i}(M,g)\leq \dim M$, are studied;
in section \ref{surf} the case of an arbitrary Riemannian surface
is tackled. Theorem \ref{surfaces} proves that a finite basis of invariants
can efficiently be computed for a generic metric $g$;
section \ref{three_dim} presents the computation of a basis
of differential invariants (with their geometric meaning)
for Riemannian homogeneous complete $3$-dimensional manifolds,
according the dimension of their Killing algebra be of dimension $4$
or $3$, as the solution to the equivalent problem is completely
different in both cases. As a by-product of the results obtained
above a explicit geomeric desciption of invariant functions
in complete Riemannian manifolds of dimensions $2$ and $3$
is provided.
\section{General position}
\subsection{Definitions}
\begin{definition}
A smooth curve $\sigma \colon (a,b)\to M$ taking values
into a manifold $M$ endowed with a linear connection
$\nabla $ is said to be in
\emph{general position up to the order} $r$,
for $1\leq r\leq m=\dim M $, at $t_0\in (a,b)$
if the vector fields
$T^\sigma ,\nabla _{T^\sigma }T^\sigma ,\dotsc,
\nabla _{T^\sigma }^{r-1}T^\sigma $ along $\sigma $
are linearly independent at $t_0$, where $T^\sigma $
is the tangent field to $\sigma $. The curve $\sigma $
is in general position up to the order $r$ if it is
in general position up to this order for every
$t\in (a,b)$.
\end{definition}
Geometrically, a curve in general position is as twisted
as possible. For example, if $(M,g)$ is of constant curvature,
then $\sigma $ is in general position up to order $r$
at $\sigma (t_0)$ if and only if no neighbourhood
$\{ \sigma (t):|t-t_0|<\varepsilon \} $ is contained
into an auto-parallel submanifold (cf.\
\cite[VII, Section 8]{KN}) of $M$ of dimension $<r$.
The condition of being in general position up to first order
is none other than an immersion and hence, it is independent
of $\nabla $; but for $r\geq 2$ the condition of being
in general position up to order $r$ does depend on $\nabla $.
\begin{lemma}
Let $\sigma \colon (a,b)\to M$ be a smooth curve taking
values into a manifold $M$ endowed with a linear connection
$\nabla $. If $(x^1,\dotsc,x^m)$ is a normal coordinate
system with respect to $\nabla $\ centered at
$x_0=\sigma (t_0)$, $a<t_0<b$, then the tangent vectors
\begin{equation}
\label{U^sigma,k}
U_{t_0}^{\sigma ,k}
=\left.
\frac{d^k(x^i\circ \sigma )}{dt^k}(t_0)
\frac{\partial }{\partial x^i}\right\vert _{\sigma (t_0)}
\in T_{\sigma (t_0)}M,
\quad k\geq 1,k\in \mathbb{N},
\end{equation}
do not depend on the particular normal coordinates chosen.
\end{lemma}
\begin{proof}
If $x^{\prime i}=a_j^ix^j$, $A=(a_j^i)\in Gl(m,\mathbb{R})$,
is another normal coordinate system, then
\begin{equation*}
\frac{\partial }{\partial x^{\prime ^i}}
=b_i^h\frac{\partial }{\partial x^h},
\quad (b_i^h)=A^{-1},
\end{equation*}
and hence
\begin{eqnarray*}
\frac{d^k(x^{\prime i}\circ \sigma )}{dt^k}(t_0)
\left.
\frac{\partial }{\partial x^{\prime i}}
\right\vert _{\sigma (t_0)}
&=&
\left.
\frac{d^k(a_j^ix^j\circ \sigma )}{dt^k}(t_0)
b_i^h\frac{\partial }{\partial x^h}
\right\vert _{\sigma (t_0)}\\
&=& b_i^ha_j^i
\frac{d^k(x^j\circ \sigma )}{dt^k}(t_0)
\left.
\frac{\partial }{\partial x^h}
\right\vert _{\sigma (t_0)}
\\
&=&\frac{d^k(x^j\circ \sigma )}{dt^k}(t_0)
\left.
\frac{\partial }{\partial x^j}
\right\vert _{\sigma (t_0)}.
\end{eqnarray*}
\end{proof}
\begin{definition}
\label{normal_g_position}
A curve $\sigma $ is said to be in
\emph{normal general position up to the order}
$r$ at $t_0\in (a,b)$ if the tangent vectors
$U_{t_0}^{\sigma ,1},U_{t_0}^{\sigma ,2},\dotsc,
U_{t_0}^{\sigma ,r}$ are linearly independent.
The curve $\sigma $ is in normal general position
up to the order $r$ if it is in normal general
position up to this order for every $t\in (a,b)$.
\end{definition}
\subsection{Genericity results}
\begin{lemma}
\label{lem20}
Let $(U;x^1,\dotsc,x^m)$ be a coordinate open domain
in a smooth manifold $M$ endowed with a linear
connection $\nabla $. There exist smooth functions
\begin{equation*}
F^{k,i}\colon J^k\left( \mathbb{R},U\right)
\to \mathbb{R},
\qquad
k\in \mathbb{N},\;1\leq i\leq m,
\end{equation*}
such that,
\begin{equation}
\label{suc}
\left(
\nabla _{T^\sigma }^kT^\sigma
\right) _t
=\left(
\frac{d^{k+1}(x^i\circ \sigma )}
{dt^{k+1}}(t)+F^{k,i}
\left(
j_t^k\sigma
\right)
\right)
\left.
\frac{\partial }{\partial x^i}
\right| _{\sigma (t)},
\end{equation}
for every curve $\sigma \colon \mathbb{R}\to U$
and every $t\in \mathbb{R}$, which are determined
as follows:
\begin{eqnarray}
F^{0,i} &=& 0,
\label{suc0} \\
F^{1,i} &=&\sum _{j,h=1}^m\Gamma _{jh}^ix_1^jx_1^h,
\label{suc1}
\end{eqnarray}
where $\Gamma _{jh}^i$ are the local symbols
of $\nabla $ in $(U;x^1,\dotsc,x^m)$, and
\begin{equation}
\label{R}
F^{k,i}=D_t
\left(
F^{k-1,i}
\right)
+\sum_{h,j=1}^m\Gamma _{hj}^ix_1^j
\left(
x_k^h+F_h^{k-1}
\right) ,
\quad
\forall k\geq 2,
\end{equation}
$(x_l^h)_{0\leq l\leq k}^{1\leq h\leq m}$ being
the coordinates induced by $(x^i)_{i=1}^m$ in
the $k$-jet bundle, i.e.,
\begin{equation*}
x_l^h
\left(
j_t^k\sigma
\right)
=\frac{d^l(x^h\circ \sigma )}{dt^l}(t),
\quad
x_0^h=x^h,
\qquad
0\leq l\leq k,\; 1\leq h\leq m,
\end{equation*}
and $D_t$ denotes the ``total derivative''
with respect to $t$, namely,
\begin{equation*}
D_t=\frac{\partial }{\partial t}
+\sum _{r=0}^\infty x_{r+1}^i
\frac{\partial }{\partial x_r^i}.
\end{equation*}
\end{lemma}
\begin{proposition}
\label{generic}
Let $M$ be a smooth manifold of dimension $m$
endowed with a linear connection $\nabla $.
The set of curves in general position up to
the order $r\leq m-1$ is a dense open subset
in $C^\infty (\mathbb{R},M)$
with respect to the strong topology.
\end{proposition}
\begin{proof}
By using the formulas \eqref{suc} it follows
that the mapping
\begin{equation}
\begin{array}{l}
\Phi _\nabla ^r\colon J^r(\mathbb{R},M)
\to \mathbb{R}\times
\left(
\oplus ^rTM
\right) ,
\smallskip \\
\Phi _\nabla ^r
\left(
j_{t_0}^r\sigma
\right) =\left(
t_0;T_{t_0}^\sigma ,
\left(
\nabla _{T^\sigma }T^\sigma
\right) _{t_0},\dotsc,
\left(
\nabla _{T^\sigma }^{r-1}T^\sigma
\right) _{t_0}\right) ,
\end{array}
\label{Phi^r}
\end{equation}
is a diffeomorphism inducing the identity
on $J^0(\mathbb{R},M)=\mathbb{R}\times M$.
We set
\begin{equation*}
E=\left\{
\left(
t,X^1,\dotsc,X^r
\right)
\in \mathbb{R}\times
\left(
\oplus ^rTM
\right)
:X^1\wedge \ldots \wedge X^r=0
\right\} ,
\end{equation*}
for every $1\leq k\leq r-1$ and for every
strictly increasing system of indices
$1\leq i_1<\ldots <i_k\leq r$ we set
\begin{equation*}
E_{i_1,\dotsc,i_k}
=\left\{
\begin{array}{r}
\left(
t,X^1,\dotsc,X^r
\right)
\in \mathbb{R}\times
\left(
\oplus ^rTM
\right)
:X^{i_1}\wedge \cdots \wedge X^{i_k}\neq 0, \\
X^{j_1},\dotsc,X^{j_{r-k}}\in
\left\langle
X^{i_1},\dotsc,X^{i_k}
\right\rangle ,
\end{array}
\right\}
\end{equation*}
with $j_1<\ldots <j_{r-k}$ and
$\{ j_1,\dotsc,j_{r-k}\}
=\{ 1,2,\dotsc,r\} \setminus \{i_1,\dotsc,i_k\} $,
and finally we set $E_0=\mathbb{R}\times Z$,
$Z$ being the zero section in $\oplus ^rTM$. Hence
\begin{equation*}
E=E_0\cup \;
\bigcup\limits_{k=1}^{r-1}\;
\bigcup_{i_1<\ldots <i_k}E_{i_1,\dotsc,i_k}.
\end{equation*}
Moreover, if $U_k(M)\subset \oplus ^kTM$ denotes
the open subset of all linearly independent systems
of $k$ vectors, then the mapping
\begin{equation*}
\begin{array}{l}
A_{i_1,\dotsc,i_k}\colon \mathbb{R}^{1+k(r-k)}
\times U_k(M)\to \mathbb{R}\times
\left(
\oplus ^rTM
\right) , \\
A_{i_1,\dotsc,i_k}\!
\left(
t;\lambda _1^1,\dotsc,\lambda _k^1,\dotsc,
\lambda _1^{r-k},\dotsc,\lambda _k^{r-k};
X^1,\dotsc,X^k
\right)
\!=\!
\left(
t;\bar{X}^1,\dotsc,\bar{X}^r
\right) ,
\end{array}
\end{equation*}
\begin{equation*}
\begin{array}{ll}
\bar{X}^{i_h}=X^h, & 1\leq h\leq k, \\
\bar{X}^{j_h}=\sum _{i=1}^k\lambda _i^hX^i,
& 1\leq h\leq r-k,
\end{array}
\end{equation*}
is an injective immersion such that
$\operatorname{im}(A_{i_1,\dotsc,i_k})
=E_{i_1,\dotsc,i_k}$, and we have
\begin{eqnarray*}
\operatorname{codim}E_{i_1,\dotsc,i_k}
&=&\dim
\left(
\mathbb{R}\times
\left(
\oplus ^rTM
\right)
\right)
-\dim E_{i_1,\dotsc,i_k} \\
&=&\left(
1+m+rm
\right)
-\left(
1+k(r-k)+m+km
\right) \\
&=&(m-k)(r-k) \\
&\geq &m+1-r,
\end{eqnarray*}
as the product $(m-k)(r-k)$ takes its minimum value
when $k$ takes its maximum value, i.e., $k=r-1$.
Accordingly, $Y_{i_1,\dotsc,i_k}
=\left(
\Phi _\nabla ^r
\right) ^{-1}
\left(
E_{i_1,\dotsc,i_k}
\right) $
is a submanifold in $J^r(\mathbb{R},M)$ of codimension
$(m-k)(r-k)$. From Thom's transversality theorem (e.g.,
see \cite[VII, Th\'{e}or\`{e}me 4.2]{Tougeron})
the set of curves $\sigma \colon \mathbb{R}\to M$
the $r$-jet extension of which, $j^r\sigma $,
is transversal to $Y_{i_1,\dotsc,i_k}$ is a residual
subset (and hence dense) in $C^\infty (\mathbb{R},M)$
for the strong topology. For such curves,
$\left(
j^r\sigma
\right) ^{-1}
\left(
Y_{i_1,\dotsc,i_k}
\right) $ is a submanifold of the real line of codimension
$(m-k)(r-k)\geq m+1-r$. If $r\leq m-1$, then it is only
possible if such a submanifold is the empty set.
Consequently, for $r\leq m-1$, the following formula holds:
\begin{align*}
\digamma ^r&
=\left\{
\sigma \in C^\infty (\mathbb{R},M):j^r\sigma
\text{ is transversal to every }Y_{i_1,\dotsc,i_k}
\right\} \\
& =\left\{
\sigma \in C^\infty (\mathbb{R},M):
\left(
j^r\sigma
\right)
\left(
\mathbb{R}
\right)
\cap Y=\emptyset
\right\} ,
\end{align*}
where $Y=Y_0\cup
\;\bigcup\nolimits_{k=1}^{r-1}
\bigcup_{i_1<\ldots <i_k}Y_{i_1,\dotsc,i_k}$.
Therefore $\Phi _\nabla ^r
\left(
j^r\sigma (\mathbb{R})
\right)
\cap E=\emptyset $
if $\sigma \in \digamma ^r$; in other words,
$\sigma $ is a curve in general position
up to order $r$ with respect to $\nabla $.

Finally, we prove that $\digamma ^r$ is an open susbet.
If $d$ is a complete distance function defining
the topology in $J^r(\mathbb{R},M)$, then for every
$\sigma \in \digamma ^r$ the function
$\delta _\sigma \colon \mathbb{R}\to \mathbb{R}^+$,
$\delta _\sigma (t)=d
\left(
j_t^r\sigma ,Y
\right) >0$
makes sense as $Y$ is a closed subset and
\begin{equation*}
N(\sigma )
=\left\{
\gamma \in C^\infty (\mathbb{R},M):d
\left(
j_t^r\sigma ,j_t^r\gamma
\right)
<\delta _{\sigma }(t),\forall t\in \mathbb{R}
\right\}
\end{equation*}
is a neigbourhood of $\sigma $ in the strong topology
of order $r$ and hence, also in the strong topology
of order $\infty $. As $\gamma \in N(\sigma )$
implies $\gamma \in \digamma ^r$, we can conclude.
\end{proof}
\begin{remark}
The statement of Proposition \ref{generic} is the best
possible, as the curves in general position up to the order
$m=\dim M$ with respect to a linear connection $\nabla $
are not dense in $C^\infty (\mathbb{R},M)$ for the strong
topology, because inflection points are unavoidable.
In fact, with the similar notations as in the proof
of Proposition \ref{generic}, we set
\begin{eqnarray*}
\digamma ^m &=&
\left\{
\sigma \in C^\infty (\mathbb{R},M):
j^m\sigma (\mathbb{R})\cap Y
=\emptyset
\right\} , \\
\bar{\digamma }^m &=&
\left\{
\sigma \in C^\infty (\mathbb{R},M):
j^m\sigma
\text{ is transversal to every }Y_{i_1,\dotsc,i_k}
\right\} .
\end{eqnarray*}
The set $\bar{\digamma }^m$ is dense
in $C^\infty (\mathbb{R},M)$ as it is residual
and $\digamma ^m$ coincides with the set of curves
in general position up to order $m$. In order to prove
that $\digamma ^m$ is not dense, it suffices to obtain
an open subset contained in its complementary set.
We set
\begin{equation*}
Y^\prime =\bigcup\limits_{k=0}^{m-2}
\; \bigcup_{i_1<\ldots <i_k}Y_{i_1,\dotsc,i_k};
\qquad
Y_i
=\left(
\Phi _\nabla ^m
\right) ^{-1}
\left(
E_i
\right) ,
\quad 1\leq i\leq m,
\end{equation*}
where $E_i$ is the set of points
$(t,X^1,\dotsc,X^m)\in \mathbb{R}\times (\oplus ^mTM)$
such that,
\begin{enumerate}
\item[i)]
$X^1\wedge \ldots \wedge \widehat{X^i}
\wedge \ldots \wedge
X^m\neq 0$,
\item[ii)]
$X^i\in
\left\langle
X^1,\dotsc,\widehat{X^i},\dotsc,X^m
\right\rangle $.
\end{enumerate}
Then, $Y=Y^\prime \cup Y_1\cup \ldots \cup Y_m$
and $Y_0=Y_1\cap \ldots \cap Y_m$ is an open subset
in each $Y_i$; hence $Y_0$ is a submanifold
of codimension $1$ in $J^m(\mathbb{R},M)$. According
to a classical result (see \cite[Theorem 6.1]{Mather})
there exists a curve $\sigma \colon \mathbb{R}\to  M$
such that, 1) $j^m\sigma $ is transversal to $Y_0$,
and 2) $j^m\sigma (\mathbb{R})\cap Y_0\neq \emptyset $.
Therefore, $j^m\sigma (\mathbb{R})\cap Y\neq \emptyset $.
Moreover, according to \cite[Lemma 1, p.\ 45]{Levine},
given a neighbourhood $U$ of $t$, there exists
a neigbourhood $E_\sigma $ of $\sigma $\ in the weak
(and hence, in the strong) topology, such that
$\tau \in E_\sigma $ implies $j^m\tau $ cuts transversally
to $Y$ at some point $t^\prime \in U$. Hence,
$\tau \in E_\sigma $ implies $j^m\tau (\mathbb{R})
\cap Y\neq \emptyset $, i.e., $\tau \notin \digamma ^m$,
and $\sigma $ thus possesses a neighbourhood of curves
not belonging to $\digamma ^m$.
\end{remark}
\begin{proposition}
\label{generic_bis}
Let $M$ be a smooth manifold of dimension $m$ endowed
with a linear connection $\nabla $. The set of curves
in normal general postion up to the order $r\leq m-1$
is a dense open subset in $C^\infty (\mathbb{R},M)$
with respect to the strong topology.
\end{proposition}
\begin{proof}
It is similar to the proof of Proposition \ref{generic}
by using the fact that the mapping
\begin{equation*}
\begin{array}{l}
\Psi _\nabla ^r\colon J^r(\mathbb{R},M)
\to \mathbb{R}\times
\left(
\oplus ^rTM
\right) ,
\smallskip \\
\Psi _\nabla ^r
\left(
j_{t_0}^r\sigma
\right)
=\left(
t_0;U_{t_0}^{1,\sigma },U_{t_0}^{2,\sigma },
\dotsc,U_{t_0}^{r,\sigma }
\right) ,
\end{array}
\end{equation*}
is a diffeomorphism over $\mathbb{R}\times M$.
\end{proof}
\section{Frenet curves}
\subsection{A Frenet curve defined}
\begin{definition}
A curve $\sigma \colon (a,b)\to M$ with values
into a Riemannian manifold $(M,g)$ is said to be
a \emph{Frenet\ curve}\ if $\sigma $ is in general
position up to order $m-1$ with respect to the
Levi-Civita connection of the metric $g$.
\end{definition}
\begin{proposition}
[Frenet frame, \protect\cite{Berard},
\protect\cite{Gluck1}, \protect\cite{Gluck2},
\protect\cite{Griffiths}, \protect\cite{Jensen},
\protect\cite{Rochowski}]
\label{referenciafrenet}
If $(M,g)$ is an oriented connected Riemannian
manifold of dimension $m$ and $\sigma \colon (a,b)\to M$
is a Frenet curve, then there exist unique vector fields
$X_1^\sigma ,\dotsc,X_m^\sigma $ defined along $\sigma $
and smooth functions $\kappa _0^\sigma ,\dotsc,
\kappa _{m-1}^\sigma \colon (a,b)\to \mathbb{R}$
with $\kappa _j^\sigma >0$, $0\leq j\leq m-2$, such that,
\begin{enumerate}
\item[\emph{(i)}]
$(X_1^\sigma (t),\dotsc,X_m^\sigma (t))$ is a positively
oriented orthonormal linear frame, $\forall t\in (a,b)$.
\item[\emph{(ii)}]
The systems $(X_1^\sigma (t),\dotsc,X_i^\sigma (t))$,
$(T_t^\sigma ,(\nabla _{T^\sigma }T^\sigma )_t,\dotsc,
(\nabla _{T^\sigma }^{i-1}T^\sigma )_t)$
span the same vector subspace and they are equally
oriented for every $1\leq i\leq m-1$ and every
$t\in (a,b)$.
\item[\emph{(iii)}]
The following formulas\ hold:
\begin{enumerate}
\item[\emph{(a)}]
$T^\sigma =\kappa _0^\sigma X_1$,
\item[\emph{(b)}]
$\nabla _{X_1^\sigma }X_1^\sigma
=\kappa _1^\sigma X_2^\sigma $,
\item[\emph{(c)}]
$\nabla _{X_1^\sigma }X_i^\sigma
=-\kappa _{i-1}^\sigma X_{i-1}^\sigma
+\kappa _i^\sigma X_{i+1}^\sigma ,
\quad 2\leq i\leq m-1$,
\item[\emph{(d)}]
$\nabla _{X_1^\sigma }X_{m}^\sigma
=-\kappa _{m-1}^\sigma X_{m-1}^\sigma $.
\end{enumerate}
\end{enumerate}
\end{proposition}
\begin{definition}
The frame $(X_1^\sigma ,\dotsc,X_{m}^\sigma )$
along $\sigma $ determined by the conditions
(i)-(iii) above is called the \emph{Frenet frame}
of $\sigma $, and the functions
$\kappa _0^\sigma ,\dotsc,\kappa _{m-1}^\sigma $
are the \emph{curvatures} of $\sigma $.
\end{definition}
\subsection{Basic formulas}
According to the item (ii) of Proposition
\ref{referenciafrenet} there exist functions
$f_{ij}^\sigma \in C^\infty (a,b)$,
$1\leq i\leq j\leq m$, such that,
\begin{equation}
\label{F4}
\nabla _{T^\sigma }^{j-1}T^\sigma
=\sum _{i=1}^jf_{ij}^\sigma X_i^\sigma ,
\quad 1\leq j\leq m,
\end{equation}
and by using the equations (a)-(d) in the item (iii)
above the following recurrence formulas are obtained
for these functions:
\begin{equation}
\left\{
\begin{array}{l}
f_{11}^\sigma
=\kappa _0^\sigma ,
\smallskip \\
f_{12}^\sigma
=\dfrac{df_{11}^\sigma }{dt},
\smallskip \\
f_{22}^\sigma
=f_{11}^\sigma \kappa _0^\sigma \kappa _1^\sigma ,
\end{array}
\right.
\label{der0}
\end{equation}
\begin{equation}
3\leq j\leq m\quad
\left\{
\begin{array}{l}
f_{1j}^\sigma
=\dfrac{df_{1,j-1}^\sigma }{dt}
-f_{2,j-1}^\sigma \kappa _0^\sigma \kappa _1^\sigma ,
\smallskip \\
f_{ij}^\sigma
=\dfrac{df_{i,j-1}^\sigma }{dt}
-f_{i+1,j-1}^\sigma \kappa _0^\sigma \kappa _i^\sigma
+f_{i-1,j-1}^\sigma \kappa _0^\sigma \kappa _{i-1}^\sigma ,\\
2\leq i\leq j-2,
\smallskip \\
f_{j-1,j}^\sigma =
\dfrac{df_{j-1,j-1}^\sigma }{dt}
+f_{j-2,j-1}^\sigma \kappa _0^\sigma \kappa _{j-2}^\sigma ,
\smallskip \\
f_{jj}^\sigma
=f_{j-1,j-1}^\sigma \kappa _0^\sigma \kappa _{j-1}^\sigma .
\end{array}
\right.
\label{der}
\end{equation}
\begin{proposition}
\label{lemdelta}
If $\sigma \colon (a,b)\to  M$ is a Frenet curve
in an oriented connected Riemannian manifold $(M,g)$,
then
\begin{equation*}
\begin{array}{ll}
\kappa _0^\sigma
=\sqrt{\Delta _1^\sigma },
\smallskip &  \\
\kappa _1^\sigma
=\sqrt{\dfrac{\Delta _2^\sigma }{(\Delta _1^\sigma )^3}},
\smallskip &  \\
\kappa _i^\sigma
=\dfrac{
\varepsilon _i
\sqrt{\Delta _{i-1}^\sigma \Delta _{i+1}^\sigma }
}
{\sqrt{\Delta _1^\sigma }\Delta _i^\sigma },
& 2\leq i\leq m-1,
\end{array}
\end{equation*}
where
\begin{equation}
\left\{
\begin{array}{l}
\Delta _k^\sigma
=\det
\left(
g
\left(
\nabla _{T^\sigma }^{i-1}T^\sigma ,
\nabla _{T^\sigma }^{j-1}T^\sigma
\right)
\right)
_{i,j=1}^k,
\smallskip \\
\varepsilon _i=1
\text{ for }2\leq i\leq m-2,
\text{ and }
\varepsilon _{m-1}=\pm 1.
\end{array}
\right.  \label{Delta_k^sigma}
\end{equation}
\end{proposition}
\begin{proof}
The formulas in the statement follow from \eqref{F4},
\eqref{der0}, and \eqref{der} taking the identity
\begin{eqnarray*}
\Delta _k^\sigma
&=&
\left(
\det
\left(
f_{ij}^\sigma
\right)
_{i,j=1}^k
\right) ^2 \\
&=&\prod\nolimits_{i=1}^k
\left( f_{ii}^\sigma
\right) ^2
\end{eqnarray*}
into account.
\end{proof}
For a smooth curve $\sigma $, the property of being
a Frenet curve at $t$ depends on $j_t^{m-1}\sigma $ only;
hence for every $t\in \mathbb{R}$ we can speak
about the open subset
$\mathcal{F}_t^{m-1}(M)\subset J_t^{m-1}(\mathbb{R},M)$
of Frenet jets. Let
\begin{equation*}
f_{ij}\colon (\pi _{m-1}^m)^{-1}(\mathcal{F}^{m-1}(M))
\to \mathbb{R},
\quad 1\leq i\leq j\leq m,
\end{equation*}
be the mapping defined by
$f_{ij}(j_t^m\sigma )=f_{ij}^\sigma (t)$,
$\pi _l^k\colon J^k(\mathbb{R},M)\to J^l(\mathbb{R},M)$,
$k\geq l$, being the canonical projections. Similarly, let
\begin{equation}
\label{Delta_k}
\Delta _k\colon J^k(\mathbb{R},M)\to \mathbb{R},
\quad 1\leq k\leq m,
\end{equation}
be the mapping given by
$\Delta _k(j_t^k\sigma )=\Delta _k^\sigma (t)$,
which is well defined according to the formula
\eqref{Delta_k^sigma}.
\begin{proposition}
\label{lem2}
If $\sigma \colon (a,b)\to M$,
$\bar{\sigma }\colon (a,b)\to \bar{M}$ are two Frenet
curves with values in Riemannian manifolds $(M,g)$,
$(\bar{M},\bar{g})$, then
$\left|
\kappa _i^\sigma
\right|
=\left|
\kappa _i^{\bar{\sigma }}
\right| $,
$0\leq i\leq m-1$, if and only if,
\begin{equation}
\label{L1}
g\left(
\nabla _{T^\sigma }^{i-1}
T^\sigma ,\nabla _{T^\sigma }^{j-1}T^\sigma
\right)
=\bar{g}
\left(
\bar{\nabla }_{T^{\bar{\sigma }}}^{i-1}
T^{\bar{\sigma }},
\bar{\nabla }_{T^{\bar{\sigma }}}^{j-1}
T^{\bar{\sigma }}
\right) ,
\quad i,j=1,\dotsc,m,
\end{equation}
$\nabla ,\bar{\nabla}$ being the Levi-Civita
connections associated to $g$, $\bar{g}$.
\end{proposition}
\begin{proof}
If \eqref{L1} holds for every $i,j=1,\dotsc,m$,
then $\Delta _k^\sigma =\Delta _k^{\bar{\sigma }}$
for $1\leq k\leq m$.
From the formulas \eqref{Delta_k^sigma}
in Proposition \ref{lemdelta} we deduce
$\kappa _i^\sigma =\kappa _i^{\bar{\sigma}}$
for $0\leq i\leq m-2$ and $|\kappa _{m-1}^\sigma |
=|\kappa _{m-1}^{\bar{\sigma}}|$. Conversely, if
$\left\vert
\kappa _i^\sigma
\right\vert
=\left\vert
\kappa _i^{\bar{\sigma}}
\right\vert $,
$0\leq i\leq m-1$, then from the formulas \eqref{der0}
and \eqref{der} by recurrence on the subindex
of $\kappa _i$ we obtain
$\left\vert
f_{mm}^\sigma
\right\vert
=\left\vert
f_{mm}^{\bar{\sigma}}
\right\vert $
and $f_{ij}^\sigma =f_{ij}^{\bar{\sigma }}$ otherwise.
Hence, for every $i,j=1,\dotsc,m$, we have
\begin{eqnarray*}
g\left(
\nabla _{T^\sigma }^{i-1}T^\sigma ,
\nabla _{T^\sigma }^{j-1}T^\sigma T
\right)
&=& \sum _{k=1}^i\sum _{l=1}^j
f_{ki}^\sigma f_{lj}^\sigma \delta _{kl} \\
&=& \sum _{k=1}^i\sum _{l=1}^j
f_{ki}^{\bar{\sigma}}f_{lj}^{\bar{\sigma }}
\delta _{kl} \\
&=& \bar{g}
\left(
\bar{\nabla }_{T^{\bar{\sigma }}}^{i-1}
T^{\bar{\sigma }},\bar{\nabla }_{T^{\bar{\sigma }}}^{j-1}
T^{\bar{\sigma }}
\right) .
\end{eqnarray*}
\end{proof}
\subsection{Existence theorems}\label{existence_ths}
The fundamental theorem for curves in Euclidean $3$-space
(cf.\ \cite[pp.\ 29--31]{Struik}) states that if two smooth
functions $\kappa (s)>0$, $\tau (s)$ are given, then there
exists a unique curve for which $s$ is the arc length,
$\kappa $ the curvature, and $\tau $ the torsion,
the moving trihedron of which at $s=s_0$ coincides
with the coordinate axes. The full generalization
of this result is as follows:
\begin{theorem}
\label{ecsfrenet}
Let $(M,g)$ be an $m$-dimensional oriented Riemannian
manifold and let $(v_1,\dotsc,v_m)$ be a positively
oriented orthonormal basis for $T_{x_0}M$.
Given functions
$\kappa _j\in C^\infty (t_0-\delta ,t_0+\delta )$,
$0\leq j\leq m-1$,
with $\kappa _j>0$ for $0\leq j\leq m-2$, there
exists $0<\varepsilon \leq \delta $ and
a unique Frenet curve
$\sigma \colon (t_0-\varepsilon ,t_0+\varepsilon )\to M$
such that,
\begin{enumerate}
\item[\emph{(i)}]
$\sigma (t_0)=x_0$,
\item[\emph{(ii)}]
$X_j^\sigma (t_0)=v_j$ for $1\leq j\leq m$,
\item[\emph{(iii)}]
$\kappa _j^\sigma =\kappa _j$ for $0\leq j\leq m-1$.
\end{enumerate}
\end{theorem}
\begin{proof}
Let $(U;\,x^1,\dotsc,x^m)$ be the normal coordinate
system centered at $x_0$ associated to the orthonormal
linear frame $(v_1,\dotsc,v_m)$ given in the statement,
let $p_M^m\colon \oplus ^mTM\to  M$ be the bundle
projection, and let denote by $(x^i,y_k^j)$,
$i,j,k=1,\dotsc,m$, the induced coordinate system
on $(p_M^m)^{-1}(U)$, i.e.,
\begin{equation*}
u_j=y_j^i(u)
\left.
\frac{\partial }{\partial x^i}
\right| _x,
\qquad
\forall u=(u_1,\dotsc,u_m)\in \oplus ^mT_{x}M,\;
x\in U.
\end{equation*}
First of all, we prove that the Frenet formulas
are locally equivalent to a system of first-order
ordinary differential equations on the manifold
$\oplus ^mTM$. In fact, as a computation shows,
the formulas (a)-(d) in Proposition
\ref{referenciafrenet} can be written in local
coordinates as follows:
\begin{equation}
\label{x^j_sigma}
\dfrac{d(x^j\circ \sigma )}{dt}
=\kappa _0^\sigma (y_1^j\circ X^\sigma ),
\end{equation}
\begin{equation}
\label{y^j_1}
\dfrac{d(y_1^j\circ X^\sigma )}{dt}
=\kappa _0^\sigma \kappa _1^\sigma
(y_2^j\circ X^\sigma )
-\kappa _0^\sigma
(\Gamma _{hi}^j\circ \sigma )
(y_1^h\circ X^\sigma )
(y_1^i\circ X^\sigma ),
\end{equation}
\begin{equation}
\left\{
\begin{array}{ll}
\dfrac{d(y_i^{c}\circ X^\sigma )}{dt}
= & \!\! \kappa _0^\sigma
[\kappa _i^\sigma (y_{i+1}^c\circ X^\sigma )
-\kappa _{i-1}^\sigma (y_{i-1}^c\circ X^\sigma )] \\
& \!\! -\kappa _0^\sigma
\left(
\Gamma _{ab}^c\circ \sigma
\right)
(y_1^a\circ X^\sigma )(y_i^b\circ X^\sigma ),
\quad 2\leq i\leq m-1,
\end{array}
\right.
\label{y^c_i}
\end{equation}
\begin{equation}
\label{y^ c_m}
\dfrac{d(y_{m}^{c}\circ X^\sigma )}{dt}
=-\kappa _0^\sigma \kappa _{m-1}^\sigma
(y_{m-1}^c\circ X^\sigma )
-\kappa _0^\sigma
\left(
\Gamma _{ab}^{c}\circ \sigma
\right)
(y_1^a\circ X^\sigma )(y_m^b\circ X^\sigma ),
\end{equation}
where $\Gamma _{ab}^{c}$ are the components
of the Levi-Civita connection $\nabla $ of $g$
with respect to the coordinate system
$(x^h)_{h=1}^m$ and
$X^\sigma \colon (a,b) \to \oplus ^mTM$,
$a<t_0<b$, is the curve given by
$X^\sigma (t)
=t(X_1^\sigma (t),\dotsc,X_m^\sigma (t))$,
$\forall t\in (a,b)$.

Hence the functions
$x^h\circ \sigma ,y_j^i\circ X^\sigma
\colon (a,b)\to \mathbb{R}$, $h,i,j=1,\dotsc,m$,
are the only solutions to the system
\eqref{x^j_sigma}--\eqref{y^ c_m} satisfying
the initial conditions (i), (ii) in the statement;
i.e., $(x^h\circ \sigma )(t_0)=x^h(x_0)$,
$(y_j^i\circ X^\sigma )(t_0)
=y_j^i(v_1,\dotsc,v_m)=\delta _j^i$.

Conversely, if $X^\sigma $ and the curvatures
$\kappa _{j-1}^\sigma $, $1\leq j\leq m$,
are replaced by an arbitrary smooth curve
$X=(X_1,\dotsc,X_{m})
\colon (t_0-\delta ,t_0+\delta )
\to \oplus ^mTM$, and the given functions
$\kappa _{j-1}$, $1\leq j\leq m$,
respectively, with $\sigma =p_M^m\circ X$,
into the equations
\eqref{x^j_sigma}--\eqref{y^ c_m} above,
then the following system is obtained:
\begin{equation}
\label{x^j_sigma_BIS}
\dfrac{d(x^j\circ \sigma )}{dt}
=\kappa _0(y_1^j\circ X),
\end{equation}
\begin{equation}
\label{y^j_1_BIS}
\dfrac{d(y_1^j\circ X)}{dt}
=\kappa _0\kappa _1(y_2^j\circ X)
-\kappa _0(\Gamma _{hi}^j\circ \sigma )
(y_1^h\circ X)(y_1^i\circ X),
\end{equation}
\begin{equation}
\left\{
\begin{array}{ll}
\dfrac{d(y_i^{c}\circ X)}{dt}
= & \!\!\kappa _0
[\kappa _i(y_{i+1}^c\circ X)
-\kappa _{i-1}(y_{i-1}^c\circ X)] \\
& \!\! -\kappa _0
\left(
\Gamma _{ab}^c\circ \sigma
\right)
(y_1^a\circ X)(y_i^b\circ X),
\quad 2\leq i\leq m-1,
\end{array}
\right.
\label{y^c_i_BIS}
\end{equation}
\begin{equation}
\label{y^c_m_BIS}
\dfrac{d(y_m^c\circ X)}{dt}
=-\kappa _0\kappa _{m-1}
(y_{m-1}^c\circ X)-\kappa _0
\left(
\Gamma _{ab}^c\circ \sigma
\right)
(y_1^a\circ X)(y_m^b\circ X),
\end{equation}
We claim that the only solution
$x^h\circ \sigma ,y_j^i\circ X
\colon (t_0-\varepsilon ,t_0+\varepsilon )
\to \mathbb{R}$ to the system
\eqref{x^j_sigma_BIS}--\eqref{y^c_m_BIS}
satisfying the initial conditions
\begin{equation*}
\begin{array}{ll}
(x^h\circ \sigma )(t_0)=x^h(x_0),
& 1\leq h\leq m, \\
(y_j^i\circ X)(t_0)=\delta _j^i,
& i,j=1,\dotsc,m,
\end{array}
\end{equation*}
provides the desired Frenet curve.

First, we observe that from the very definition
of \eqref{x^j_sigma_BIS}--\eqref{y^c_m_BIS},
the linear frame $(X_1,\dotsc,X_{m})$---defined
along the curve $\sigma $ with components
$x^h\circ \sigma $---determined by $X$, i.e.,
$X_j=(y_j^i\circ X)\frac{\partial }{\partial x^i}$,
satisfies the following equations:
\begin{equation}
\left\{
\begin{array}{ll}
T^\sigma =\kappa _0X_1,
&  \\
\nabla _{X_1}X_1=\kappa _1X_2,
&  \\
\nabla _{X_1}X_i
=-\kappa _{i-1}X_{i-1}
+\kappa _iX_{i+1},
& 2\leq i\leq
m-1, \\
\nabla _{X_1}X_m
=-\kappa _{m-1}X_{m-1}.
&
\end{array}
\right.
\label{Frenet_bis}
\end{equation}
Next, the item (i) in Proposition
\ref{referenciafrenet} is proved
to hold for this linear frame. In fact,
the functions $\varphi _{ij}(t)=g(X_i(t),X_j(t))$,
$|t-t_0|<\varepsilon $ , $1\leq i\leq j\leq m$,
are the only solution to the system
\begin{equation*}
\begin{array}{l}
\dfrac{d\varphi _{11}}{dt}
=2\kappa _0\kappa _1\varphi _{12},
\medskip \\
\dfrac{d\varphi _{1j}}{dt}=\kappa _0
\left(
\kappa _1\varphi _{2j}
-\kappa _{j-1}\varphi _{1,j-1}
+\kappa _j\varphi _{1,j+1}
\right) , \\
2\leq j\leq m-1,
\medskip \\
\dfrac{d\varphi _{1m}}{dt}
=\kappa _0
\left(
\kappa _1\varphi _{2m}
-\kappa _{m-1}\varphi _{2,m-1}
\right) ,
\medskip \\
\dfrac{d\varphi _{ij}}{dt}
=\kappa _0
\left(
\kappa _i\varphi _{i+1,j}
+\kappa _j\varphi _{i,j+1}
-\kappa _{i-1}\varphi _{i-1,j}
-\kappa _{j-1}\varphi _{i,j-1}
\right) , \\
2\leq i\leq j\leq m-1,\medskip \\
\dfrac{d\varphi _{im}}{dt}
=\kappa _0
\left(
\kappa _i\varphi _{i+1,m}
-\kappa _{i-1}\varphi _{i-1,m}
-\kappa _{m-1}\varphi _{i,m-1}
\right) ,
\\
2\leq i\leq m-1,\medskip \\
\dfrac{d\varphi _{mm}}{dt}
=-2\kappa _0\kappa _{m-1}\varphi _{m-1,m},
\end{array}
\end{equation*}
such that $\varphi _{ij}(0)=\delta _{ij}$,
but Kronecker deltas are readily seen to be also
a solution to this system; hence
$g(X_i(t),X_j(t))=\delta _{ij}$. By virtue
of the assumption, one has
$\mathrm{vol}_g(X_1(0),\dotsc,X_m(0))
=\mathrm{vol}_g(v_1,\dotsc,v_m)=1$,
and accordingly,
$\mathrm{vol}_g(X_1(t),\dotsc,X_m(t))=1$
for every $t\in (t_0-\varepsilon ,t_0+\varepsilon )$.

Finally, as the curvatures $\kappa _j^\sigma $,
$0\leq j\leq m-1$, are completely determined
by the Frenet formulas, it suffices to prove
that the linear frame $(X_1,\dotsc,X_m)$
satisfies the property (ii) of Proposition
\ref{referenciafrenet}, which is equivalent
to prove the existence of functions
$h_{ij}\in C^\infty (t_0-\varepsilon ,t_0+\varepsilon )$,
$1\leq i\leq j\leq m$, such that
$\nabla _{T^\sigma }^{j-1}T^\sigma
=\sum_{i=1}^jh_{ij}X_i$ for $1\leq j\leq m $.
If $j=1$, then this formula follows from the first
formula in \eqref{Frenet_bis} with $h_{11}=\kappa _0$.
Hence we can proceed by recurrence on $j\geq 2$.
By applying the operator $\nabla _{T^\sigma }$
to both sides of the equation
$\nabla _{T^\sigma }^{j-2}T^\sigma
=\sum _{i=1}^{j-1}h_{i,j-1}X_i$, we have
$\nabla _{T^\sigma }^{j-1}T^\sigma
=\sum _{i=1}^{j-1}((dh_{i,j-1}/dt)X_i
+\nabla _{T^\sigma }X_i)$, and the result follows
by replacing the term $\nabla _{T^\sigma }X_i
=\kappa _0\nabla _{X_1}X_i$ by its expression deduced
from the formulas in \eqref{Frenet_bis} above.
\end{proof}
\begin{theorem}
\label{ecsfrenet_bis}
Let $(M,g)$ be an $m$-dimensional oriented
Riemannian manifold. Given a system of functions
$\kappa =(\kappa _0,\dotsc,\kappa _{m-1})$,
with $\kappa _j\in C^\infty (t_0-\delta ,t_0+\delta )$
for $0\leq j\leq m-1$ and $\kappa _j>0 $
for $0\leq j\leq m-2$, let
$f_{ij}^\kappa \colon (t_0-\delta ,t_0+\delta )
\to \mathbb{R}$, $1\leq i\leq j\leq m$, be the functions
defined by the following recurrence relations:
\begin{equation}
\left\{
\begin{array}{l}
f_{11}^\kappa =\kappa _0,
\smallskip \\
f_{12}^\kappa =\dfrac{df_{11}^\kappa }{dt},
\smallskip \\
f_{22}^\kappa =f_{11}^\kappa \kappa _0\kappa _1,
\end{array}
\right.  \label{der0_kappa}
\end{equation}
\begin{equation}
3\leq j\leq m
\quad
\left\{
\begin{array}{l}
f_{1j}^\kappa
=\dfrac{df_{1,j-1}^\kappa }{dt}
-f_{2,j-1}^\kappa
\kappa _0\kappa _1,
\smallskip \\
f_{ij}^\kappa
=\dfrac{df_{i,j-1}^\kappa }{dt}
-f_{i+1,j-1}^\kappa \kappa _0\kappa _i
+f_{i-1,j-1}^\kappa \kappa _0\kappa _{i-1}, \\
2\leq i\leq j-2,
\smallskip \\
f_{j-1,j}^\kappa
=\dfrac{df_{j-1,j-1}^\kappa }{dt}
+f_{j-2,j-1}^\kappa
\kappa _0\kappa _{j-2},
\smallskip \\
f_{jj}^\kappa =
f_{j-1,j-1}^\kappa \kappa _0\kappa _{j-1}.
\end{array}
\right.  \label{der_kappa}
\end{equation}
Let $w_j\in T_{x_0}M$, $1\leq j\leq m$,
be vectors such that the system
$(w_1,\dotsc,w_{m-1})$ is linearly
independent. The necessary and sufficient
conditions for a Frenet curve
$\sigma \colon (t_0-\varepsilon ,t_0+\varepsilon )\to M$,
$0<\varepsilon \leq \delta $, to exist such that,
\begin{enumerate}
\item[\emph{(i)}]
$\sigma (t_0)=x_0$,
\item[\emph{(ii)}]
$\left(
\nabla _{T^\sigma }^{j-1}T^\sigma
\right)
(t_0)=w_j$ for $1\leq j\leq m$,
\item[\emph{(iii)}]
$\kappa _{j-1}^\sigma =\kappa _{j-1}$
for $1\leq j\leq m$,
\end{enumerate}
\noindent are the following:
\begin{equation}
\label{scalar_products}
g\left(
w_i,w_j
\right)
=\sum _{h=1}^if_{hi}^\kappa (t_0)f_{hj}^\kappa (t_0),
\quad
1\leq i\leq j\leq m.
\end{equation}
\end{theorem}
\begin{proof}
If the curve $\sigma $ in the statement exists,
then $f_{ij}^\kappa =f_{ij}^\sigma $,
where the functions $f_{ij}^\sigma $ are given
in the formulas \eqref{der0} and \eqref{der},
and from the formulas \eqref{F4} we have
\begin{eqnarray*}
g\left(
w_i,w_j
\right)
&=& g\left(
\nabla _{T^\sigma }^{i-1}
T^\sigma ,
\nabla _{T^\sigma }^{j-1}T^\sigma
\right)
\left( t_0
\right) \\
&=& \sum _{h=1}^if_{hi}^\sigma (t_0)f_{hj}^\sigma (t_0).
\end{eqnarray*}
Hence, all the conditions \eqref{scalar_products}
are necessary for the curve $\sigma $ to exist.

Let $(v_1,\dotsc,v_{m-1})$ be the orthonormal system
in $T_{x_0}M$ obtained by applying the Gram-Schmidt
process to the system $(w_1,\dotsc,w_{m-1})$, and let
$v_{m}$ be the only unitary tangent vector orthogonal
to $v_1,\dotsc,v_{m-1}$ for which the basis
$(v_1,\dotsc,v_{m-1},v_m)$ of $T_{x_0}M$ is positively
oriented. According to Theorem \ref{ecsfrenet},
there exists a Frenet curve
$\sigma \colon (t_0-\varepsilon ,t_0+\varepsilon )\to M$
such that,

a) $\sigma (t_0)=x_0$;
b) $X_j^\sigma (t_0)=v_j$, $1\leq j\leq m$;
c) $\kappa _j^\sigma =\kappa _j$, $0\leq j\leq m-1$.

Hence $f_{ij}^\kappa =f_{ij}^\sigma $, as follows
from the formulas \eqref{der0}, \eqref{der},
\eqref{der0_kappa}, and \eqref{der_kappa}, and from
\eqref{F4} we obtain
$(\nabla _{T^\sigma }^{j-1}T^\sigma )(t_0)
=\sum_{i=1}^jf_{ij}^\sigma (t_0)v_i$, $1\leq j\leq m$.
Consequently, the Gram-Schmidt process applied
to $(T^\sigma (t_0),\dotsc,
(\nabla _{T^\sigma }^{m-2}T^\sigma )(t_0))$
also leads one to the orthonormal system
$(v_1,\dotsc,v_{m-1})$. By virtue of \eqref{scalar_products}
we thus have
$g(w_i,w_j)=g(\nabla _{T^\sigma }^{i-1}T^\sigma ,
\nabla _{T^\sigma }^{j-1}T^\sigma )(t_0)$
for $1\leq i\leq j\leq m$, and we can conclude by simply
recalling the following fact: If $(u_1,\dotsc,u_k)$,
$(u_1^\prime ,\dotsc,u_k^\prime )$ are two linearly
independent systems such that,
1st) the Gram-Schmidt process applied
to $(u_1,\dotsc,u_k)$, as well as
to $(u_1^\prime ,\dotsc,u_k^\prime )$, leads
to the same orthonormal system, and
2nd) $g(u_i,u_j)=g(u_i^\prime ,u_j^\prime )$
for $1\leq i\leq j\leq k$, then both systems coincide,
i.e., $u_i=u_i^\prime $ for $i=1,\dotsc,k$.

Finally, the Frenet frame at $t_0$ of any Frenet curve
$\sigma \colon (t_0-\varepsilon ,t_0+\varepsilon )\to  M$
satisfying (i)-(iii) in the statement coincides
with the system $(v_1,\dotsc,v_m)$ and we can conlude
its uniqueness from Theorem \ref{ecsfrenet}.
\end{proof}
\begin{remark}
The explicit formulas for the functions $f_{ij}^\kappa $
are rather involved, but their computational evaluation
is quite feasible; for example,
\begin{equation*}
\begin{array}{l}
\text{For }m=3\text{:} \\
\multicolumn{1}{r}{f_{33}^\kappa
=\left(
\frac{d^2\kappa _0}{dt^2}
\right) ^2
-2\frac{d^2\kappa _0}{dt^2}
(\kappa _0)^3(\kappa _1)^2+(\kappa _0)^6(\kappa _1)^4
+9\left(
\frac{d\kappa _0}{dt}
\right) ^2
(\kappa _0)^2(\kappa _1)^2} \\
\multicolumn{1}{r}{+6\frac{d\kappa _0}{dt}
\frac{d\kappa _1}{dt}(\kappa _0)^3\kappa _1
+\left(
\frac{d\kappa _1}{dt}
\right) ^4(\kappa _0)^4
+(\kappa _0)^6(\kappa _1)^2(\kappa _2)^2.}
\end{array}
\end{equation*}
\begin{equation*}
\begin{array}{l}
\text{For }m=4\text{:} \\
\multicolumn{1}{r}{f_{24}^\kappa
=\frac{d\kappa _0}{dt}\frac{d^3\kappa _0}{dt^3}
-3\left( \frac{d\kappa _0}{dt}
\right) ^2
(\kappa _0)^2(\kappa _1)^2
+2\frac{d\kappa _0}{dt}\frac{d\kappa _1}{dt}
(\kappa _0)^3\kappa _1
+4\frac{d^2\kappa _0}{dt^2}
(\kappa _0)^3(\kappa _1)^2} \\
\multicolumn{1}{r}
{+\frac{d^2\kappa _1}{dt^2}(\kappa _0)^4
\kappa _1-(\kappa _0)^6(\kappa _1)^2(\kappa _2)^2
-(\kappa _0)^{6}(\kappa _1)^4.}
\end{array}
\end{equation*}
\end{remark}
\begin{remark}
Although the formulas \eqref{Delta_k^sigma}
in the statement of Proposition \ref{lemdelta}
are necessary for the curve $\sigma $ to exist,
they are not sufficient. More formally, let
$\Xi _\kappa
=(\pi ^{\prime m},\Xi _\kappa ^1,\dotsc,\Xi _\kappa ^m)
\colon J^m(\mathbb{R},M)\to M\times \mathbb{R}^m$
be the mapping of fibred manifolds over $M$ given by,
\begin{equation*}
\begin{array}{l}
\Xi _{\kappa }^1
=\Delta _1\circ \pi _1^m-(\kappa _0\circ \pi ^m)^2,
\smallskip \\
\Xi _{\kappa }^2
=\Delta _2\circ \pi _2^m-(\kappa _1\circ \pi ^m)^2
\left(
\Delta _1\circ \pi _1^m
\right) ^3,
\smallskip \\
\Xi _{\kappa }^i
=\left(
\Delta _{i-2}\circ \pi _{i-2}^m
\right)
\left(
\Delta _i\circ \pi _i^m
\right)
-(\kappa _{i-1}\circ \pi ^m)^2
\left(
\Delta _1\circ \pi _1^m
\right)
\left(
\Delta _{i-1}\circ \pi _{i-1}^m
\right) ^2, \\
3\leq i\leq m-1, \\
\Xi _\kappa ^m
\left(
j_{t_0}^m\sigma
\right)
=\mathrm{vol}_g
\left(
T^\sigma ,
\nabla _{T^\sigma }T^\sigma ,
\dotsc,
\nabla _{T^\sigma }^{m-1}T^\sigma
\right) (t_0)
\sqrt{\Delta _{m-2}^\sigma (t_0)} \\
\qquad
\qquad
\qquad
\qquad
\qquad
\qquad
\qquad
\qquad
\qquad
-\kappa _{m-1}(t_0)
\sqrt{\Delta _1^\sigma (t_0)}
\Delta _{m-1}^\sigma (t_0),
\end{array}
\end{equation*}
$\pi ^m\colon J^m(\mathbb{R},M)\to \mathbb{R}$,
$\pi ^{\prime m}\colon J^m(\mathbb{R},M)\to M$
being the canonical projections and
$\Delta _k\colon J^m(\mathbb{R},M)\to \mathbb{R}$
being the mapping defined in \eqref{Delta_k}.
Because of the expression for $\Xi _\kappa ^m$,
the mapping $\Xi _\kappa $ is quasi-linear.
If $\kappa =\kappa ^\sigma $ for a Frenet curve
$\sigma $, then it is readily checked that
the equations
$\Xi _{\kappa ^\sigma }\circ j^m\sigma =0$
are equivalent to the formulas
\eqref{Delta_k^sigma}.

The fibred submanifold
$R^m=(\Xi _\kappa )^{-1}\{ 0\}
\cap (\pi _{m-1}^m)^{-1}(\mathcal{F}^{m-1}(M))$
is not formally integrable, even in the simplest
cases. For $M=\mathbb{R}^2$, $g=dx^2+dy^2$,
the submanifold $R^2$ is defined by the equations
$\dot{x}^2+\dot{y}^2=(\kappa _0)^2$,
$\dot{x}\ddot{y}-\ddot{x}\dot{y}
=(\kappa _0)^3\kappa _1$, whereas its first
prolongation $R^3$ is obtained by adding
to the equations of $R^2$ the following:
$\dot{x}\ddot{x}+\dot{y}\ddot{y}
=\kappa _0\dot{\kappa}_0$,
$\dot{x}\dddot{y}-\dddot{x}\dot{y}
=3(\kappa _0)^2\dot{\kappa}_0\kappa _1
+(\kappa _0)^3\dot{\kappa}_1$.
Hence the projection
$\pi _2^3\colon R^3\to R^2$ is not surjective.
For $M=\mathbb{R}^3$, $g=dx^2+dy^2+dz^2$,
the submanifold $R^3$\ is defined
by the following equations:
\begin{eqnarray*}
\dot{x}^2+\dot{y}^2+\dot{z}^2
&=& (\kappa _0)^2, \\
\left\vert
\begin{array}{cc}
\dot{x}^2+\dot{y}^2+\dot{z}^2
& \dot{x}\ddot{x}
+\dot{y}\ddot{y}
+\dot{z}\ddot{z} \\
\dot{x}\ddot{x}
+\dot{y}\ddot{y}
+\dot{z}\ddot{z}
& \ddot{x}^2+\ddot{y}^2
+\ddot{z}^2
\end{array}
\right\vert
&=&(\kappa _0)^6(\kappa _1)^2, \\
\left\vert
\begin{array}{ccc}
\dot{x} & \dot{y} & \dot{z} \\
\ddot{x} & \ddot{y} & \ddot{z} \\
\dddot{x} & \dddot{y} & \dddot{z}
\end{array}
\right\vert
&=& (\kappa _0)^6(\kappa _1)^2\kappa _2,
\end{eqnarray*}
and one needs to reach the second prolongation
$R^5\subset J^5(\mathbb{R},M)$ of $R^3$ in order
to obtain an integrable system. For higher
dimensions similar conclusions are obtained.
\end{remark}
\subsection{$\mathcal{F}^{m-1}(M)$ and $\mathcal{N}^{m-1}(M)$}
\label{FandN}
\begin{theorem}
\label{comparingFandN}
Let $(M,g)$ be a Riemannian manifold. Let
$\mathcal{F}_{t_0,x_0}^{m-1}(M) $ be the subset
of Frenet jets
$j_{t_0}^{m-1}\sigma \in J^{m-1}(\mathbb{R},M)$
such that, $\sigma (t_0)=x_0$. Let
$\mathcal{N}_{t_0,x_0}^{m-1}(M)$
be the subset of jets
$j_{t_0}^{m-1}\sigma \in J^{m-1}(\mathbb{R},M)$
such that,
\emph{i)} $\sigma (t_0)=x_0$,
\emph{ii)} the curve $\sigma $ is normal general
position up to the order\ $m-1$ at $t_0$, i.e.,
the tangent vectors
$U_{t_0}^{\sigma ,1},
U_{t_0}^{\sigma ,2},
\dotsc,
U_{t_0}^{\sigma ,r}$
defined in \emph{\eqref{U^sigma,k}}
are linearly independent.

If either $\dim M=m\leq 4$ or $g$
is a flat metric at a neighbourhood
of $x_0$, then
$\mathcal{F}_{t_0,x_0}^{m-1}(M)
=\mathcal{N}_{t_0,x_0}^{m-1}(M)$,
$\forall t_0\in \mathbb{R}$,
$\forall x_0\in M$.

In the general case,
$\mathcal{F}_{t_0,x_0}^{m-1}(M)
\backslash \mathcal{N}_{t_0,x_0}^{m-1}(M)$
\emph{(}resp.\
$\mathcal{N}_{t_0,x_0}^{m-1}(M)
\backslash \mathcal{F}_{t_0,x_0}^{m-1}(M)$\emph{)}
is not empty but nowhere dense in
$\mathcal{F}_{t_0,x_0}^{m-1}(M)$
\emph{(}resp.\
$\mathcal{N}_{t_0,x_0}^{m-1}(M)$\emph{)}.
\end{theorem}
\begin{proof}
If $(U,x^1,\dotsc,x^m)$ is the normal
coordinate system attached to an
orthonormal basis for $T_{x_0}M$ with respect
to the Levi-Civita connection $\nabla $ of $g$,
then for every smooth curve
$(t_0-\varepsilon ,t_0+\varepsilon )\to M$,
$\sigma (t_0)=x_0$, the following formulas hold:
\begin{equation}
\begin{array}{r}
T_{t_0}^\sigma =U_{t_0}^{\sigma ,1},
\smallskip \\
\left(
\nabla _{T^\sigma }T^\sigma
\right) _
{t_0}=U_{t_0}^{\sigma ,2},
\smallskip \\
\left(
\nabla _{T^\sigma }^2T^\sigma
\right)
_{t_0}=U_{t_0}^{\sigma ,3},
\end{array}
\label{1-2-3}
\end{equation}
and for $r\geq 3$, from the formulas \eqref{suc},
\eqref{suc0}, \eqref{suc1}, \eqref{R}, we conclude
the existence of a polynomial $P_i^r$ in the values
$(\partial ^{|I|}\Gamma _{jk}^h/\partial x^i)(x_0)$,
$I\in \mathbb{N}^m$, $|I|\leq r-2$, $\Gamma _{jk}^h$
being the Christoffel symbols of $\nabla $ with respect
to the coordinates chosen, and the components
$(d^k(x^i\circ \sigma )/dt^k)(t_0)$ (also in such
coordinates) of the tangent vectors
$U_{t_0}^{\sigma ,k}$, $1\leq k\leq r-1 $, defined
in \eqref{U^sigma,k} such that,
\begin{equation}
\label{nabla_U}
\left(
\nabla _{T^\sigma }^rT^\sigma
\right)
_{t_0}=U_{t_0}^{\sigma ,r+1}+P_i^r
\left.
\frac{\partial }{\partial x^i}
\right| _{x_0}.
\end{equation}
As the values
$(\partial ^{|I|}\Gamma _{jk}^h/\partial x^i)(x_0)$,
$1\leq |I|\leq k$, can be written as a polynomial
(e.g., see \cite{Gray1}, \cite{Kulkarni})
in the components of the curvature tensor field $R^g$
and its covariant derivatives
$\nabla R^g,\dotsc,\nabla ^{k-1}R^g$ at $x_0 $,
we conclude that the same holds for $P_i^r$.
For example,
\begin{eqnarray*}
\left(
\nabla _{T^\sigma }^3T^\sigma
\right) _{t_0}
\!\!\!\!
&=& \!\!\!\!
U_{t_0}^{\sigma ,4}
+\tfrac{1}{3}R_{x_0}^g
\! \left(
\! T_{t_0}^\sigma ,
\left(
\nabla _{T^\sigma }T^\sigma
\right)
_{t_0}
\right) \!
T_{t_0}^\sigma , \\
\left(
\nabla _{T^\sigma }^{4}T^\sigma
\right) _{t_0}
\!\!\!\!
&=& \!\!\!\!
U_{t_0}^{\sigma ,5}
+2\left(
\nabla R^g
\right) _{x_0}
\!
\left(
T_{t_0}^\sigma ,
\left(
\nabla _{T^\sigma }T^\sigma
\right) _{t_0},
T_{t_0}^\sigma ,T_{t_0}^\sigma
\right) \\
&& \!\!\!
+3R_{x_0}^g \!
\left(
T_{t_0}^\sigma ,
\left(
\nabla _{T^\sigma }T^\sigma
\right) _{t_0}
\right) \!
\left(
\nabla _{T^\sigma }T^\sigma
\right) _{t_0}
\! + \! \tfrac{7}{3}R_{x_0}^g \!
\left(
T_{t_0}^\sigma ,
\left(
\nabla _{T^\sigma }^2T^\sigma
\right) _{t_0}
\right) \!
T_{t_0}^\sigma .
\end{eqnarray*}
For $m\leq 4$ from the formulas
\eqref{1-2-3} we conclude
\begin{equation}
\label{equality}
\mathcal{F}_{t_0,x_0}^{m-1}(M)
=\mathcal{N}_{t_0,x_0}^{m-1}(M).
\end{equation}
Moreover, the equation
$(\nabla _t^rT)_{t_0}=U_{t_0}^{\sigma ,r+1}$
holds for $0\leq r\leq m-2$ if and only if,
$P_i^r=0$ for $0\leq r\leq m-2$, $1\leq i\leq m$.
In particular, this happens when $g$ is flat
at a neighbourhood of $x_0$; hence the equality
\eqref{equality} also holds in this case.
If the tangent vectors $T_{t_0}^\sigma ,
(\nabla _{T^\sigma }T^\sigma )_{t_0},
\dotsc,
(\nabla _{T^\sigma }^{m-2}T^\sigma )_{t_0}$
are linearly independent but there exists
a non-trivial linear combination, i.e.,
$0=\sum _{h=1}^{m-1}\lambda _hU_{t_0}^{\sigma ,h}$,
then from \eqref{nabla_U} we deduce
\begin{equation}
\label{linear_combinat}
\sum\nolimits_{r=0}^{m-2}\lambda _{r+1}
\left\{
\left(
\nabla _{T^\sigma }^rT^\sigma
\right) _{t_0}
-P_i^r
\left(
\partial /\partial x^i
\right) _{x_0}
\right\} =0,
\end{equation}
which implies that at least one of the vectors
$P_i^r
\left(
\partial /\partial x^i
\right) _{x_0}$,
$3\leq r\leq m-2$, does not vanish.

Letting $N_{t_0}=T_{t_0}
\times (\nabla _tT)_{t_0}\times
\cdots
\times (\nabla _t^{m-2}T)_{t_0}$, where
$\times $ stands for cross product,
we obtain a basis $(T_{t_0}^\sigma ,
(\nabla _{T^\sigma }T^\sigma )_{t_0},
\dotsc,
(\nabla _{T^\sigma }^{m-2}T^\sigma )_{t_0},
N_{t_0})$
for $T_{x_0}M$, and we can write,
\begin{equation*}
P_i^r
\left(
\partial /\partial x^i
\right) _{x_0}
=\sum\nolimits_{q=0}^{m-2}\mu _{q}^r
\left(
\nabla _{T^\sigma }^{q}T^\sigma
\right) _{t_0}
+\mu ^rN_{t_0},
\quad 0\leq r\leq m-2,
\end{equation*}
for some scalars $\mu _{q}^r,\mu ^r$, agreeing
that $P_i^r=0$ for $0\leq r\leq 2$; hence
$\mu _{q}^r=\mu ^r=0$ for $0\leq r\leq 2$. Then,
\eqref{linear_combinat} is equivalent to saying
that the homogeneous linear system
\begin{eqnarray*}
\sum\nolimits_{r=0}^{m-2}\mu ^r\lambda _{r+1}
&=&0, \\
\sum\nolimits_{r=0}^{m-2}
\left(
\mu _{q}^r-\delta _q^r
\right) \lambda _{r+1}
&=& 0,\quad 0\leq q\leq m-2,
\end{eqnarray*}
of $m$ equations in the $m-1$ unknowns
$\lambda _1,\dotsc,\lambda _{m-1}$
admits a non-trivial solution, i.e., the rank
of the $m\times (m-1)$ matrix
\begin{equation*}
\mathbf{\mu }(m)=\left(
\begin{array}{rrrccc}
0 & 0 & 0 & \mu ^3 & \ldots & \mu ^{m-2} \\
-1 & 0 & 0 & \mu _0^3 & \ldots & \mu _0^{m-2} \\
0 & -1 & 0 & \mu _1^3 & \ldots & \mu _1^{m-2} \\
0 & 0 & -1 & \mu _2^3 & \ldots & \mu _2^{m-2} \\
0 & 0 & 0 & \mu _3^3-1 & \ldots & \mu _3^{m-2} \\
\vdots & \vdots & \vdots & \vdots & \ddots & \vdots \\
0 & 0 & 0 & \mu _{m-2}^3 & \ldots & \mu _{m-2}^{m-2}-1
\end{array}
\right)
\end{equation*}
must be $\leq m-2$. This condition characterizes
$\mathcal{F}_{t_0,x_0}^{m-1}(M)
\backslash \mathcal{N}_{t_0,x_0}^{m-1}(M)$.
The proof for $\mathcal{N}_{t_0,x_0}^{m-1}(M)
\backslash \mathcal{F}_{t_0,x_0}^{m-1}(M)$
is similar.
\end{proof}
\begin{example}
\label{m=5}
For $m=5$ jets in $\mathcal{F}_{t_0,x_0}^4(M)
\backslash \mathcal{N}_{t_0,x_0}^4(M)$
are given by,
$\operatorname{rk}\mathbf{\mu }(5)=3 $.
Hence
$\mu ^3=0$, $\mu _3^3=1$; $P_i^3
\left(
\partial /\partial x^i
\right) _{x_0}
=\sum\nolimits_{q=0}^2\mu _q^3
\left(
\nabla _{T^\sigma }^{q}T^\sigma
\right) _{t_0}
+\left(
\nabla _{T^\sigma }^3T^\sigma
\right) _{t_0}$, i.e.,
\begin{equation}
\label{belongs}
R_{x_0}^g
\left(
T_{t_0}^\sigma ,
(\nabla _{T^\sigma }T^\sigma )_{t_0}
\right)
T_{t_0}^\sigma
-\left(
\nabla _{T^\sigma }^3T^\sigma
\right) _{t_0}
\in \left\langle
T_{t_0}^\sigma ,
(\nabla _{T^\sigma }T^\sigma )_{t_0},
(\nabla _{T^\sigma }^2T^\sigma )_{t_0}
\right\rangle .
\end{equation}

In addition, assume $(x^i)_{i=1}^5$ is the normal
coordinate system defined by the Frenet frame
$\left( X_i^\sigma (t_0)\right) _{i=1}^{5}$.
From the formulas \eqref{F4}, \eqref{der0},
and \eqref{der} it follows the formula
\eqref{belongs} can be reformulated by saying
that the tangent vector
\begin{multline*}
\left(
\kappa _0^\sigma
\right) ^4\kappa _1^\sigma R_{x_0}^g
\left(
X_1^\sigma (t_0),X_2^\sigma (t_0)
\right)
X_1^\sigma (t_0)-f_{14}^\sigma (t_0)
X_1^\sigma (t_0)-f_{24}^\sigma (t_0)
X_2^\sigma (t_0)\\
-f_{34}^\sigma (t_0)X_3^\sigma (t_0)
-f_{44}^\sigma (t_0)X_4^\sigma (t_0)
\end{multline*}
must belong to
$\left\langle
X_1^\sigma
\left( t_0
\right) ,
X_2^\sigma
\left(
t_0
\right) ,
X_3^\sigma
\left(
t_0
\right)
\right\rangle $, or equivalently,
\begin{equation*}
\begin{array}{l}
g\left(
R_{x_0}^g
\left(
X_1^\sigma
\left(
t_0
\right) ,
X_2^\sigma
\left(
t_0
\right)
\right)
X_1^\sigma
\left(
t_0
\right) ,
X_4^\sigma
\left(
t_0
\right)
\right)
=\kappa _2^\sigma
\left(
t_0
\right)
\kappa _3^\sigma
\left(
t_0
\right) ,
\\
g
\left(
R_{x_0}^g
\left(
X_1^\sigma
\left(
t_0
\right) ,
X_2^\sigma
\left(
t_0
\right)
\right)
X_1^\sigma
\left(
t_0
\right) ,
X_5^\sigma
\left(
t_0
\right)
\right) =0.
\end{array}
\end{equation*}
\end{example}
\begin{example}
\label{m=6}
For $m\! =\! 6$ jets in
$\mathcal{F}_{t_0,x_0}^5\!(M)
\backslash \mathcal{N}_{t_0,x_0}^5\!(M)$
are given by,
$\operatorname{rk}\mathbf{\mu }(6)\leq 4$.
The rank of $\mathbf{\mu }(6)$ is $3$
if and only if, $\mu ^3=\mu ^4=\mu _4^3
=\mu _3^4=0$, $\mu _3^3=\mu _4^4=1$,
or equivalently,
\begin{equation*}
\begin{array}{l}
P_i^3
\left(
\partial /\partial x^i
\right) _{x_0}
=\mu _0^3T_{t_0}^\sigma +\mu _1^3
\left(
\nabla _{T^\sigma }T^\sigma
\right) _{t_0}+\mu _2^3
\left(
\nabla _{T^\sigma }^2T^\sigma
\right) _{t_0}
+\left(
\nabla _{T^\sigma }^3T^\sigma
\right) _{t_0}, \\
P_i^{4}
\left(
\partial /\partial x^i
\right) _{x_0}
=\mu _0^4T_{t_0}^\sigma +\mu _1^4
\left(
\nabla _{T^\sigma }T^\sigma
\right) _{t_0}
+\mu _2^4
\left(
\nabla _{T^\sigma }^2T^\sigma
\right) _{t_0}
+\left(
\nabla _{T^\sigma }^4T^\sigma
\right) _{t_0}.
\end{array}
\end{equation*}
In other words, $P_i^3
\left(
\partial /\partial x^i
\right)
_{x_0}
-\left(
\nabla _{T^\sigma }^3T^\sigma
\right) _{t_0}$ and
$P_i^4
\left(
\partial /\partial x^i
\right) _{x_0}
-\left(
\nabla _{T^\sigma }^4T^\sigma
\right) _{t_0}$ belong to the subspace spanned
by $T_{t_0}^\sigma $,
$\left(
\nabla _{T^\sigma }T^\sigma
\right) _{t_0}$, and
$\left(
\nabla _{T^\sigma }^2T^\sigma
\right) _{t_0}$.

The rank of $\mathbf{\mu }(6)$ is $4$ if and only if,
\begin{equation*}
\begin{array}{l}
\mu _4^3\mu _3^4=(\mu _3^3-1)(\mu _4^4-1), \\
\mu _4^3\mu ^4=\mu ^3(\mu _4^4-1), \\
\mu _3^4\mu ^3=\mu ^4(\mu _3^3-1),
\end{array}
\end{equation*}
but
$\left(
\mu ^3,\mu ^4,\mu _4^3,
\mu _3^4,\mu _3^3,\mu _4^4
\right)
\neq (0,0,0,0,1,1)$.
\end{example}
\section{The equivalence problem\label{criterio_general}}
\subsection{Necessary conditions for congruence}
\begin{definition}
Two curves $\sigma \colon (a,b)\to (M,g)$,
$\bar{\sigma }\colon (a,b)\to  (\bar{M},\bar{g})$
with values in two Riemannian manifolds are said
to be \emph{congruent} if an open neighbourhood $U$
of the image of $\sigma $ in $M$ and an isometric
embedding $\phi \colon U\to \bar{M}$ exist such that,
$\bar{\sigma }=\phi \circ \sigma $. If $M$ and
$\bar{M}$ are oriented, then $\phi $ is assumed
to preserve the orientation.
\end{definition}
\begin{proposition}
\label{curvaturas}
The curvatures of a Frenet curve with values
into an oriented Riemannian manifold $(M,g)$
are invariant by congruence.
\end{proposition}
\begin{proof}
Let $\sigma \colon (a,b)\to M$,
$\bar{\sigma }\colon (a,b)\to \bar{M}$ be two
Frenet curves with values in two oriented
Riemannian manifolds $(M,g)$, $(\bar{M},\bar{g})$.
Let $\phi \colon U\to \bar{M}$ be an isometric
embedding preserving the orientation
defined on a neighbourhood $U$ of the image
of $\sigma $, such that
$\bar{\sigma }=\phi \circ \sigma $.
Since $\phi $ is an affine mapping,
from \cite[VI, Proposition 1.2]{KN} we know
$\phi \cdot (\nabla _{T^\sigma }^jT^\sigma )
=\bar{\nabla}_{T^{\bar{\sigma }}}^jT^{\bar{\sigma }}$,
for all $j\in \mathbb{N}$. As $\phi $ is an isometry,
for every $i,j=1,\dotsc,m$, the following formula holds:
\begin{eqnarray*}
\bar{g}
\left(
\bar{\nabla }_{T^{\bar{\sigma }}}^{i-1}T^{\bar{\sigma }},
\bar{\nabla}_{T^{\bar{\sigma }}}^{j-1}T^{\bar{\sigma }}
\right)
&=&\bar{g}
\left(
\phi _\ast
\left(
\nabla _{T^\sigma }^{i-1}T^\sigma
\right) ,
\phi _\ast
\left(
\nabla _{T^\sigma }^{j-1}T^\sigma
\right)
\right) \\
&=& g
\left(
\nabla _{T^\sigma }^{i-1}
T^\sigma ,\nabla _{T^\sigma }^{j-1}T^\sigma
\right) .
\end{eqnarray*}
From Proposition \ref{lem2} it follows
$\kappa _i^\sigma =\kappa _i^{\bar{\sigma}}$,
$0\leq i\leq  m-2$, $|\kappa _{m-1}^\sigma |
=|\kappa _{m-1}^{\bar{\sigma}}|$ and since $\phi $
is orientation-preserving, we have
\begin{equation*}
\mathrm{vol}_g
\left(
T^\sigma ,\dotsc,\nabla _{T^\sigma }^{m-1}T^\sigma
\right)
=\mathrm{vol}_{\bar{g}}(T^{\bar{\sigma}},
\dotsc,\bar{\nabla }_{T^{\bar{\sigma }}}^{m-1}
T^{\bar{\sigma }}).
\end{equation*}
Thus $\kappa _{m-1}^\sigma
=\kappa _{m-1}^{\bar{\sigma }}$,
and the proof is complete.
\end{proof}
\begin{proposition}
\label{frenetframe}
Let $(M,g)$, $(\bar{M},\bar{g})$ be two oriented
Riemannian manifolds with associated Levi-Civita
connections $\nabla ,\bar{\nabla}$, respectively,
and let $\sigma \colon (a,b)\to M$,
$\bar{\sigma}\colon (a,b)\to \bar{M}$, be two
Frenet curves which are congruent under
the isometric embedding $\phi $. Then,
$\phi \cdot X_i^\sigma =X_i^{\bar{\sigma }}$,
$\omega _\sigma ^i
=\phi ^\ast \omega _{\bar{\sigma }}^i$,
for $1\leq i\leq m$,
$\left(
\omega _\sigma ^1,\dotsc,\omega _\sigma ^m
\right) $,
$\left(
\omega _{\bar{\sigma }}^1,
\dotsc,\omega _{\bar{\sigma }}^m
\right) $
being the dual coframes of the Frenet frames
of $\sigma $, $\bar{\sigma }$, respectively.
Moreover,
\begin{equation*}
\phi _\ast
\left( \nabla ^jR
\left( X_{i_1}^\sigma ,
\dotsc,X_{i_{j+3}}^\sigma ,\omega _\sigma ^i
\right)
\right)
\left(
\sigma (t)
\right)
=\bar{\nabla}^j\bar{R}
\left(
X_{i_1}^{\bar{\sigma }},
\dotsc,X_{i_{j+3}}^{\bar{\sigma }},\omega _{\bar{\sigma}}^i
\right)
\left(
\bar{\sigma }(t)
\right) ,
\end{equation*}
for all $j\in \mathbb{N}$, $t\in (a,b)$,
and all systems of indices
$i,i_1,\dotsc,i_{j+3}=1,\dotsc,m$,
where $R$, $\bar{R}$ are the curvature tensors
of $(M,g)$, $(\bar{M},\bar{g})$,
respectively.
\end{proposition}
\begin{proof}
From Proposition \ref{curvaturas} and the formulas
\eqref{der0}, \eqref{der}, we obtain
$f_{ij}^\sigma =f_{ij}^{\bar{\sigma}}$ for every
$i,j=1,\dotsc,m$. Therefore
\begin{eqnarray*}
\sum _{i=1}^mf_{im}^\sigma
X_i^{\bar{\sigma }}
&=& \sum _{i=1}^mf_{im}^{\bar{\sigma }}
X_i^{\bar{\sigma }} \\
&=& \bar{\nabla}_{T^{\bar{\sigma }}}^{m-1}
T^{\bar{\sigma }} \\
&=& \phi _\ast
 \left(
\nabla _{T^\sigma }^{m-1}T^\sigma
\right) \\
&=& \phi _\ast
\Bigl(
\sum_{i=1}^mf_{im}^\sigma X_i^\sigma
\Bigr) \\
&=& \sum_{i=1}^mf_{im}^\sigma
\left(
\phi _\ast X_i^\sigma
\right) .
\end{eqnarray*}
Thus $\bar{X}_i^{\bar{\sigma }}
=\phi _\ast X_i^\sigma $
(and therefore $\omega _\sigma ^i
=\phi ^\ast \omega _{\bar{\sigma }}^i$),
$1\leq  i\leq  m$. This proves the first
part of the statement. The second one derives
from \cite[VI, Proposition 1.2]{KN}.
\end{proof}
\subsection{General criterion of congruence}
\begin{theorem}
\label{CGC}
Let $(M,g)$, $(\bar{M},\bar{g})$ be two oriented
connected Riemannian manifolds of class $C^\omega $
of the same dimension, $m=\dim M=\dim \bar{M}$,
with Levi-Civita connections $\nabla $, $\bar{\nabla}$,
and let $\sigma \colon (a,b)\to M$,
$\bar{\sigma}\colon (a,b)\to \bar{M}$ be two Frenet
curves of class $C^\omega $ with tangent fields $T$,
$\bar{T}$, respectively. If $x_0=\sigma (t_0)$,
$\bar{x}_0=\bar{\sigma }(t_0)$, $a<t_0<b$, then
$\sigma $ and $\bar{\sigma }$
are congruent on some neigbourhoods of $x_0$ and $\bar{x}_0$,
respectively, if and only if the following conditions hold:
\begin{enumerate}
\item[\emph{(i)}] For every $j\in \mathbb{N}$ and every
$0\leq i\leq m-1$,
\begin{equation}
\label{CU}
\frac{d^j\kappa _i^\sigma }{dt^j}(t_0)
=\frac{d^j\kappa _i^{\bar{\sigma }}}{dt^j}(t_0),
\end{equation}
\item[\emph{(ii)}]
For every $j\in \mathbb{N}$ and all systems of indices
$i,i_1,\dotsc,i_{j+3}=1,\dotsc,m$, the following formula
holds:
\begin{equation}
\left(
\nabla ^jR
\right)
\left(
X_{i_1}^\sigma ,
\dotsc,
X_{i_{j+3}}^\sigma ,
\omega _\sigma ^i
\right)
\left( x_0
\right)
=\left(
\bar{\nabla }^j\bar{R}
\right)
\left(
X_{i_1}^{\bar{\sigma }},
\dotsc,
X_{i_{j+3}}^{\bar{\sigma }},
\omega _{\bar{\sigma }}^i
\right)
\left(
\bar{x}_0
\right) ,
\label{DC}
\end{equation}
where
$\left(
\omega _{\sigma }^1,
\dotsc,
\omega _{\sigma }^m
\right) $,
$\left(
\omega _{\bar{\sigma }}^1,
\dotsc,
\omega _{\bar{\sigma }}^m
\right) $
are the dual coframes of the Frenet frames
$\left(
X_1^\sigma ,\dotsc,X_m^\sigma
\right) $,
$\left(
X_1^{\bar{\sigma }},\dotsc,X_m^{\bar{\sigma }}
\right) $
of $\sigma $, $\bar{\sigma }$, and $R$,
$\bar{R}$ are the curvature tensors of $(M,g)$,
$(\bar{M},\bar{g})$, respectively.
\end{enumerate}
\end{theorem}
\begin{proof}
From Proposition \ref{curvaturas} (resp.\
Propositon \ref{frenetframe}) the equations
\eqref{CU} (resp.\ \eqref{DC}) follow.
To prove the converse, let $A\colon T_{x_0}M
\to T_{\bar{x}_0}\bar{M}$ be the linear isometry
given by,
$A\left(
X_i^\sigma (t_0)
\right)
=X_i^{\bar{\sigma } }(t_0)$, $1\leq  i\leq  m$.
The condition \eqref{DC} implies that $A$ maps
the tensor $(\nabla ^jR)_{x_0}$ into the tensor
$(\bar{\nabla }^j\,\bar{R})_{\bar{x}_0}$,
for all $j\in \mathbb{N}$.
From \cite[VI, Theorem 7.2]{KN} we conclude
that the polar map $\phi \colon U\to \overline{U}$,
$\phi =\exp _{\bar{X}_0}\circ A\circ \exp _{x_0}^{-1}$,
is an affine isomorphism and from
\cite[Lemma 2.3.1]{Wolf} it follows that $\phi $
is an isometry. In order to finish the proof,
it suffices to check that
$\phi (\sigma (t))=\bar{\sigma}(t)$ for
$|t-t_0|<\varepsilon $,
and a small enough $\varepsilon >0$. The Frenet curve
$\gamma =\phi \circ \sigma \colon (a,b)\to \bar{M}$
satisfies $\gamma (t_0)=\bar{x}_0$,
$X_i^\gamma (t_0)=\phi _\ast
\left(
X_i^\sigma (t_0)
\right)
=X_i^{\bar{\sigma }}(t_0)$.
As the curvatures are of class $C^\omega $,
from the condition \eqref{CU} we deduce
$\kappa _j^{\bar{\sigma}}=\kappa _j^\sigma $,
$0\leq  j\leq  m-1$, and since the curvatures
are invariant by congruence, we know
$\kappa _j^\sigma =\kappa _j^{\gamma }$; hence
$\kappa _j^{\bar{\sigma }}=\kappa _j^{\gamma }$,
$0\leq  j\leq  m-1$. Taking the formulas \eqref{F4},
\eqref{der} and the condition \eqref{CU} into account
for $1\leq  j\leq  m$ we have
\begin{eqnarray*}
A\left(
\left(
\nabla _{T^\sigma }^{j-1}T^\sigma
\right) _{t_0}
\right)
&=& \sum _{i=1}^mf_{ij}^\sigma (t_0)A
\left(
X_i^\sigma (t_0)
\right) \\
&=& \sum _{i=1}^mf_{ij}^{\bar{\sigma}}
(t_0)X_i^{\bar{\sigma }}(t_0)\\
&=&
\left(
\bar{\nabla }_{T^{\bar{\sigma }}}^{j-1}
T^{\bar{\sigma }}
\right) _{t_0}.
\end{eqnarray*}
Therefore
\begin{eqnarray*}
\left(
\bar{\nabla }_{T^\gamma }^{j-1}T^\gamma
\right) _{t_0}
&=& \phi _\ast
\left(
\left(
\nabla _{T^\sigma }^{j-1}T^\sigma
\right) _{t_0}
\right) \\
&=& A
\left(
\left(
\nabla _{T^\sigma }^{j-1}T^\sigma
\right) _{t_0}
\right) \\
&=&
\left(
\bar{\nabla }_{T^{\bar{\sigma }}}^{j-1}
T^{\bar{\sigma }}
\right) _{t_0},
\end{eqnarray*}
for $1\leq j\leq m$. By applying
Theorem \ref{ecsfrenet} we conclude
$\bar{\sigma}=\gamma =\phi \circ \sigma $
on $(t_0-\varepsilon ,t_0+\varepsilon )$.
\end{proof}
\begin{corollary}
\label{CorolCGC}
Let $(M,g)$, $(\bar{M},\bar{g})$
be two oriented connected Riemannian manifolds
of class $C^\omega $ of the same dimension,
$m=\dim M=\dim \bar{M}$, and let
$\sigma \colon (a,b)\to M$,
$\bar{\sigma}\colon (a,b)\to \bar{M}$
be two Frenet curves, respectively.
If $x_0=\sigma (t_0)$,
$\bar{x}_0=\bar{\sigma}(t_0)$, $a<t_0<b$,
then $\sigma $ and $\bar{\sigma }$ are congruent
on some neigbourhoods $U$ and $\bar{U}$ of $x_0$
and $\bar{x}_0$, respectively if, and only if,
the following conditions hold:
\begin{enumerate}
\item[\emph{(i)}]
For every $j\in \mathbb{N}$ and every
$0\leq i\leq m-1$, it holds
$\kappa _i^\sigma (t)
=\kappa _i^{\bar{\sigma }}(t)$,
for $|t-t_0|<\varepsilon $.
\item[\emph{(ii)}]
For every $j\in \mathbb{N}$ and every system
of indices $i,i_1,\dotsc,i_{j+3}\in
\left\{
1,\dotsc,m
\right\} $,
\begin{equation*}
(\nabla ^jR)
\left(
X_{i_1}^\sigma ,
\dotsc,X_{i_{j+3}}^\sigma ,
\omega _\sigma ^i
\right)
\left( x_0
\right)
=(\bar{\nabla }^j\bar{R})
\left(
X_{i_1}^{\bar{\sigma }},
\dotsc,
X_{i_{j+3}}^{\bar{\sigma }},
\omega _{\bar{\sigma }}^i
\right)
\left(
\bar{x}_0
\right) .
\end{equation*}
\end{enumerate}
\end{corollary}
\subsection{Remarks on the criterion of congruence}
\label{remarks_CGC}
\begin{remark}
\label{remark_1.5}
The condition \eqref{DC} of Theorem \ref{CGC}
is \emph{not} equivalent to the following:
\begin{equation}
\label{unodos}
R
\left(
X_j^\sigma ,
X_k^\sigma ,
X_l^\sigma ,
\omega _\sigma ^i
\right)
\left(
\sigma (t)
\right)
=\bar{R}
\left(
X_j^{\bar{\sigma }},
X_k^{\bar{\sigma }},
X_l^{\bar{\sigma }},
\omega _{\bar{\sigma }}^i
\right)
\left(
\bar{\sigma}(t)
\right) ,
\quad
\left|
t-t_0\right| <\varepsilon .
\end{equation}
Differentiating the left-hand side
of \eqref{unodos} we have
\begin{eqnarray*}
\frac{d}{dt}R
\left(
X_j^\sigma ,
X_k^\sigma ,
X_l^\sigma ,
\omega _\sigma ^i
\right)
\left(
\sigma (t)
\right)
&=&
\left(
\nabla _{T^\sigma }R
\right)
\left(
X_j^\sigma ,
X_k^\sigma ,
X_l^\sigma ,
\omega _\sigma ^i
\right)
\left( \sigma (t)
\right) \\
&& +R
\left(
\nabla _{T^\sigma }X_j^\sigma ,
X_k^\sigma ,
X_l^\sigma ,
\omega _\sigma ^i
\right)
\left(
\sigma (t)
\right)
+\ldots \\
&& +R
\left(
X_j^\sigma ,
X_k^\sigma ,
X_l^\sigma ,
\nabla _{T^\sigma }\omega _\sigma ^i
\right)
\left(
\sigma (t)
\right) \\
&=& \kappa _0^\sigma (t)
\nabla R
\left(
X_1^\sigma ,
X_j^\sigma ,
X_k^\sigma ,
X_l^\sigma ,\omega _\sigma ^i
\right)
\left(
\sigma (t)
\right)
+\ldots
\end{eqnarray*}
As the first argument of $\nabla R$
in the formula above is $X_1^\sigma $, the function
\begin{equation*}
\nabla R
\left(
X_h^\sigma ,
X_j^\sigma ,
X_k^\sigma ,
X_l^\sigma ,
\omega _\sigma ^i
\right) ,
\quad h\neq 1,
\end{equation*}
cannot be recovered from
$R\left(
X_j^\sigma ,
X_k^\sigma ,
X_l^\sigma ,
\omega _\sigma ^i
\right)
\left(
\sigma (t)
\right) $.
Therefore the formulas \eqref{unodos}
do not imply the formulas \eqref{DC},
although \eqref{DC} do imply \eqref{unodos}
as the manifolds involved are analytic.
\end{remark}
\begin{example}
\label{example_1}
Let us consider the bidimensional torus
$\mathbb{T}\subset \mathbb{R}^3$ with implicit
equation $(x^2+y^2+z^2+3)^2=16(x^2+y^2)$.
On the radius-$2$ circumference
$C=\mathbb{T}\cap \left\{ z=1\right\} $
the Gaussian curvature of $T$ vanishes
and $C$ is a regular curve of positive constant
curvature. The curvature tensor of $\mathbb{R}^2$
vanishes in particular along any curve
$C^\prime \subset \mathbb{R}^2$ with the same
curvature as $C$; but $C$ and $C^\prime $ a
re not congruent since the Gaussian curvature
of $\mathbb{T}$ does not vanish at every
neighbourhood of a point of $C$.
\end{example}
\begin{remark}
\label{remark_2}
From Proposition \ref{referenciafrenet}
we deduce that the Frenet frame of a Frenet curve
$\sigma $ and its dual frame at a point
$\sigma (t_0)$ depend on $j_{t_0}^{m-1}\sigma $
only. Hence, for every system of indices
$j\in \mathbb{N}$,
$i_1,\dotsc,i_{j+3},i\in
\left\{
1,\dotsc,m
\right\} $,
a function
$I_{i_1\ldots i_{j+3},i}^j\colon \mathcal{F}^{m-1}(M)
\to \mathbb{R}$ can be defined on the open subset
$\mathcal{F}^{m-1}(M)\subset J^{m-1}(\mathbb{R},M)$
of the jets of order $m-1$ of Frenet curves with
values in $M$ by setting
\begin{equation}
\label{Ii1...i_j+3_i}
I_{i_1\ldots i_{j+3},i}^j(j_t^{m-1}\sigma )
=(\nabla ^jR)
\left(
X_{i_1}^\sigma ,\dotsc,X_{i_{j+3}}^\sigma ,
\omega _\sigma ^i
\right)
\left( \sigma (t)
\right) .
\end{equation}
Similarly, for every $0\leq i\leq m-1$,
a function
\begin{equation}
\label{varkappa}
\varkappa _i
\colon (\pi _{m-1}^m)^{-1}\mathcal{F}^{m-1}(M)
\subset J^m(\mathbb{R},M)
\to \mathbb{R}
\end{equation}
can be defined by setting
$\varkappa _i(j_t^m\sigma )=\kappa _i^\sigma (t)$.

From Theorem \ref{CGC} it follows that all these functions
are invariant by isometry (see the section
\ref{Differential_invariants} below). Since
$\dim \mathcal{F}^{m-1}(M)=m^2+1$, only a finite number
(not greater than $m^2+1$) of such functions can be
functionally independent generically. Hence, the infinite
number of conditions given in \eqref{DC} can be reduced
to a finite number. Nevertheless, it is not easy
to determine a bound for the index $j$, which measures
the times one has to differentiate covariantly the curvature
tensor. In Theorem \ref{surfaces} below this bound is proved
to be $2$ in the case of a surface.
\end{remark}
\begin{remark}
\label{remark_2.1}
For the sake of simplicity, here we use the Riemann curvature
tensor $R_4$ of $g$ rather than the curvature tensor $R$
(cf.\ \cite[V, Section 2]{KN}), i.e.,
\begin{equation*}
R_4(X,Y,T,Z)=g(R(T,Z)Y,X).
\end{equation*}
For $j=0$, all the functions
\begin{eqnarray*}
I_{i_1i_2i_3i_4}\colon \mathcal{F}^{m-1}(M)
&\to &\mathbb{R},
\\
I_{i_1i_2i_3i_4}(j_t^{m-1}\sigma )
&=& R_4
\left(
X_{i_1}^\sigma (t),
X_{i_2}^\sigma (t),
X_{i_3}^\sigma (t),
X_{i_4}^\sigma (t)
\right)
\end{eqnarray*}
can be written in terms of the functions
\begin{equation*}
\begin{array}{l}
I_{ij}\colon FM\to \mathbb{R}, \\
I_{ij}(X_1,\dotsc,X_{m})=R_4(X_i,X_j,X_i,X_j), \\
X_i,X_j\in T_{x}M,\;1\leq i<j\leq m,
\end{array}
\end{equation*}
where $FM$ is the bundle of linear frames
of $M$, as follows from the polarization formula,
namely
\begin{eqnarray*}
6R(X,Y,T,Z)
\!\!\!
&=&
\!\!\!
R(X,Z,X,Z)\!+\!R(T,Y,T,Y)
\! - \! R(X,T,X,T)\!-\!R(Z,Y,Z,Y) \\
&& - \! R(X,Y \! + \! Z,X,Y \! + \! Z)
\! + \! R(X,Y \! + \! T,X,Y \! + \! T) \\
&& - \! R(T,Y \! + \! Z,T,Y \! + \! Z)
\! + \! R(Z,Y\!+\!T,Z,Y\!+\!T) \\
&& + \!
R(X \! + \! T,Y \! + \! Z,X \! +\! T,Y \! + \! Z)
\! + \! R(X \! + \! Z,T,X \! + \! Z,T) \\
&& - \! R(X \! + \! T,Y,X \! + \! T,Y)
\!-\!R(X\!+\!T,Z,X\!+\!T,Z) \\
&& - \!
R(X \! + \! Z,T \! + \! Y,X \! + \! Z,T \! + \! Y)
\! + \! R(X \! + \! Z,Y,X \! + \! Z,Y).
\end{eqnarray*}
In fact, if
$\mathtt{f}_{M}\mathcal{\colon F}^{m-1}(M)\to FM$,
$s_{ij}\colon FM\to  FM$ are the maps
\begin{equation*}
\begin{array}{l}
\mathtt{f}_{M}(j_t^{m-1}\sigma )=
(X_1^\sigma (t),\dotsc,X_m^\sigma (t)),\\
s_{ij}(X_1,\dotsc,X_m)
=(X_1,
\dotsc,X_i,
\dotsc,X_i+X_j,
\dotsc,X_m),\;i<j,
\end{array}
\end{equation*}
then
\begin{align*}
6I_{i_1i_2i_3i_4}\!&
=\!\left(
I_{i_1i_4}
+I_{i_2i_3}
-I_{i_1i_3}
-I_{i_2i_4}
-I_{i_1i_4}\circ s_{i_2i_4}
+I_{i_1i_3}\circ s_{i_2i_3}
-I_{i_3i_4}\circ s_{i_2i_4}
\right. \\
& \quad
+I_{i_3i_4}\circ s_{i_2i_3}
+I_{i_3i_4}\circ s_{i_2i_4}\circ s_{i_1i_3}
+I_{i_3i_4}\circ s_{i_1i_4}
-I_{i_2i_3}\circ s_{i_1i_3} \\
& \quad
\left.
-I_{i_3i_4}\circ s_{i_1i_3}
-I_{i_3i_4}\circ s_{i_2i_3}\circ s_{i_1i_4}
+I_{i_2i_4}\circ s_{i_1i_4}
\right)
\circ \mathtt{f}_{M}.
\end{align*}
\end{remark}
\begin{remark}
\label{remark_2.2}Theorem \ref{CGC} is the most general
result we can expect without imposing any additional
condition on $(M,g)$ and $(\bar{M},\bar{g})$ except
for the fact of being analytic. This is principally
due to the fact that \cite[VI, Theorem 7.2]{KN} cannot
be generalized to non-analytic manifolds, as shown
in the next example.
\end{remark}
\begin{example}
\label{example_2}
Let $g$, $\bar{g}$ be the two Riemannian metrics
on $M=\bar{M}=\mathbb{R}^m$, $m\geq 2$, defined by,
\begin{eqnarray*}
g_{ij}(x) &=& \delta _{ij}+\exp (-|x|^{-2}), \\
\bar{g}_{ij}(x) &=& \delta _{ij},
\end{eqnarray*}
respectively; hence $(M,g)$ is not analytic at the origin.
If $R$ is the curvature tensor of $(M,g)$ and $\nabla $
is its associated Levi-Civita connection, then
$(\nabla ^nR)(0)=0$ for all $n\in \mathbb{N}$.
The identity map $Id\colon T_0M\to  T_{\bar{0}}\bar{M}$
is an isometry, since $g_{ij}(0)=\bar{g}_{ij}(\bar{0})
=\delta _{ij}$. Moreover, $(\nabla ^jR)(0)
=(\bar{\nabla }^j\,\bar{R})(\bar{0})=0$, where
$\bar{R}$ (resp.\  $\bar{\nabla}$) is the curvature
tensor (resp.\  the Levi-Civita connection) of $\bar{g}$.
If there exists an affine isomorphism
$\phi \colon U\to  \bar{U}=\bar{M}$, defined on normal
neighbourhoods of $0$, such that $\phi _{\ast ,0}=Id$,
then taking \cite[Lemma 2.3.1]{Wolf} into account,
$\phi $ must necessarily be an isometry. Hence $\phi $
maps the tensor $\nabla ^jR$ into the tensor
$\bar{\nabla}^j\,\bar{R}=0$, for all $j\in \mathbb{N}$.
Consequently, $\nabla ^jR$ must vanish in a normal
neighbourhood of $0$, but this is not true. In fact,
as $g_{ij}=\delta _{ij}+h(|x|)$, with
$h(s)=\exp (-s^{-2})$, we have
\begin{equation*}
g^{ij}=\delta _{ij}-\frac{h(|x|)}{1+mh(|x|)}.
\end{equation*}
Following the notation in \cite{KN},
the Christoffel symbols are,
\begin{equation*}
\Gamma _{ij}^k=\frac{h^\prime (|x|)}{|x|}
\left(
x^i+x^j-x^k
+\frac{h(|x|)}{1+mh(|x|)}
\left(
\sum _{a=1}^mx^a-m(x^i+x^j)
\right)
\right) .
\end{equation*}
If $x_t=\left(
t,\dotsc,t\right)
\in \mathbb{R}^m$, $t\neq 0$, then
\begin{equation*}
\Gamma _{ij}^k
\left(
x_t
\right)
=\frac{h^\prime (|x_t|)t}
{|x_t|
\left(
1+mh(|x_t|)
\right)
} \neq 0,
\end{equation*}
\begin{eqnarray*}
\Gamma _{ii}^a
\left( x_t
\right)
\Gamma _{aj}^j
\left(
x_t
\right)
&=& \left(
\frac{h^\prime (|x_t|)}
{|x_t|
\left(
1+mh(|x_t|)
\right) }
\right) ^2t^2 \\
&=&\Gamma _{ji}^a
\left(
x_t
\right)
\Gamma _{ai}^j
\left(
x_t
\right) .
\end{eqnarray*}
Consequently, $\Gamma _{ii}^a
\left( x_t
\right)
\Gamma _{aj}^j
\left(
x_t
\right)
-\Gamma _{ji}^a
\left(
x_t
\right)
\Gamma _{ai}^j
\left(
x_t
\right)
=0$, and hence
\begin{eqnarray*}
R_{iji}^j
\left(
x_t
\right)
&=&\frac{\partial \Gamma _{ii}^j}{\partial
x^j}
\left(
x_t
\right)
-\frac{\partial \Gamma _{ji}^j}{\partial x^i}
\left(
x_t
\right) \\
&=&\frac{-2h^\prime (|x_t|)}{|x_t|} \\
&\neq & 0,
\end{eqnarray*}
Thus, $R_{iji}^j$ does not vanish at $x_t$ for small enough
$t\neq 0$.
\end{example}
\section{Differential invariants}\label{Differential_invariants}
\subsection{Basic definitions}\label{definitions}
Let $\mathfrak{I}(M,g)$ be the group of isometries
of a complete Riemannian connected manifold $(M,g)$
endowed with its structure of Lie transformation
group (cf.\ \cite[VI, Theorem 3.4]{KN}) and let
$\mathfrak{i}(M,g)$ be its Lie algebra, which is
anti-isomorphic to the algebra of Killing vector fields.

Every diffeomorphism $\phi \colon M\to M$ induces
a transformation $\phi ^{(r)}$ on $J^r
\left(
\mathbb{R},M
\right) $
given by $\phi ^{(r)}
\left(
j_t^r\sigma
\right)
=j_t^r
\left( \phi \circ \sigma
\right) $,
and a natural action (on the left) of the group
$\mathfrak{I}(M,g)$ on $J^r
\left(
\mathbb{R},M\right) $ can be defined by
$\phi \cdot j_t^r\sigma =\phi ^{(r)}
\left(
j_t^r\sigma
\right) $.
Each $X\in \mathfrak{i}(M,g)$ induces a flow
$\phi _t$ and its jet prolongation
$\phi _t^{(r)}$ determines a flow on $J^r
\left(
\mathbb{R},M
\right) $,
the infinitesimal generator of which is
the vector field denoted by $X^{(r)}\in \mathfrak{X}
\left(
J^r
\left(
\mathbb{R},M
\right)
\right) $.
The tangent spaces to the orbits of the action
of $\mathfrak{I}(M,g)$ on $J^r
\left(
\mathbb{R},M
\right) $
coincide with the fibres of the distribution
$\mathfrak{D}^r\subset \mathfrak{X}
\left( J^r\left( \mathbb{R},M\right) \right) $
spanned by the vector fields $X^{(r)}$;
more precisely,
we have
\begin{equation*}
T_{j_t^r\sigma }
\left(
\mathfrak{I}(M,g)\cdot j_t^r\sigma
\right)
=\mathfrak{D}_{j_t^r\sigma }^r
=\left\{
X_{j_t^r\sigma }^{(r)}:X\in \mathfrak{i}(M,g)
\right\} .
\end{equation*}
\begin{definition}
A smooth function
$I\colon J^r(\mathbb{R},M)\to \mathbb{R}$
is said to be an \emph{invariant} of order
$r$ (cf.\ \cite[7, 4.1]{AVL}, \cite{Kumpera})
if, $I\circ \phi ^{(r)}=I$,
$\forall \phi \in \mathfrak{I}(M,g)$.
A first integral
$f\colon J^r(\mathbb{R},M)\to \mathbb{R}$
of the distribution $\mathfrak{D}^r$
is called a \emph{differential invariant }
of order $r$; i.e., $X^{(r)}(f)=0$,
$\forall X\in \mathfrak{i}(M,g)$.
\end{definition}
\begin{remark}
A differential invariant is an invariant
with respect to the connected component
of the identity $\mathfrak{I}^0(M,g)$
in $\mathfrak{I}(M,g)$.
\end{remark}
\begin{lemma}
[cf.\ \protect\cite{N}]
\label{lemma_rank}
The distribution $\mathfrak{D}^r$ is involutive
and its rank is locally constant on a dense open subset
$\mathcal{U}^r\subseteq J^r(\mathbb{R},M)$. If $N_r$
denotes the maximal number of functionally independent
differential invariants of order $r\geq 0$, then
\begin{eqnarray}
N_r &=& \dim J^r
\left(
\mathbb{R},M
\right)
-\operatorname{rk}
\left.
\mathfrak{D}^r
\right\vert _{\mathcal{U}^r}
\label{N_r} \\
&=& m(r+1)
+1-\operatorname{rk}
\left.
\mathfrak{D}^r
\right\vert _{\mathcal{U}^r}.
\notag
\end{eqnarray}
\end{lemma}
\begin{proof}
$\mathfrak{D}^r$ is involutive
as $[X_1^{(r)},X_2^{(r)}]=[X_1,X_2]^{(r)}$,
$\forall X_1,X_2\in \mathfrak{i}(M,g)$.
Let $\mathcal{U}^r$ be the subset defined
as follows: A point
$\xi =j_t^r\sigma \in J^r(\mathbb{R},M)$
belongs to $\mathcal{U}^r$ if and only if
$\xi $ admits an open neigbourhood
$N_\xi $ such that
$\dim \mathfrak{D}_{\xi ^\prime }^r
=\dim \mathfrak{D}_\xi ^r$ for every
$\xi ^\prime \in N_\xi $. As
$N_\xi \subseteq \mathcal{U}^r$, it follows
that $\mathcal{U}^r$ is an open subset,
which is non-empty as the dimension of the fibres
of $\mathfrak{D}^r$ is uniformly bounded and hence,
$\mathcal{U}^r$ contains the points $\xi $
for which $\dim \mathfrak{D}_\xi ^r
=\max _{\xi ^\prime \in J^r(\mathbb{R},M)}
\dim \mathfrak{D}_{\xi ^\prime }^r=d$. In fact,
if this equation holds, then there exists
an open neighbourhood $N_\xi $ of $\xi $ such that
the dimension of the fibres of $\mathfrak{D}^r$
over the points $\xi ^\prime \in N_\xi $ is
at least $d$, as if $(X_i^{(r)})_\xi $,
$1\leq i\leq d$, is a basis for $\mathfrak{D}_\xi ^r$,
then the vector fields $(X_i^{(r)})$ are
linearly independent at each point of an open neigbourhood
and hence, they are also a basis, $d$ being the maximal
value of the dimension of the fibres of $\mathfrak{D}^r$.
From the very definition of $\mathcal{U}^r$ we thus
conclude that $N_{\xi }\subseteq \mathcal{U}^r$.
The same argument proves that the rank of $\mathfrak{D}^r$
is locally constant over $\mathcal{U}^r $. Next,
we prove that $\mathcal{U}^r$ is dense. If
$O\subset J^r(\mathbb{R},M)$ is a non-empty open subset,
then there exists $\xi \in O$ such that
$\dim \mathfrak{D}_\xi ^r=\max_{\xi ^\prime \in O}
\dim \mathfrak{D}_{\xi ^\prime }^r$ and we can conclude
as above. The last part of the statement follows directly
from the Frobenius theorem.
\end{proof}
If $f$ is a differential invariant of order $r$,
then $D_t(f)$ is a differential invariant of order $r+1$.
This fact follows from the formula $X^{(r+1)}\circ D_t
=D_t\circ X^{(r)}$ for every $X\in \mathfrak{i}(M,g)$,
which, in its turn, follows from the formula
\begin{equation}
\label{F^r}
X^{(r)}=\sum _{j=0}^r(D_t)^j(f^i)
\frac{\partial }{\partial x_j^i},
\end{equation}
for every $X\in \mathfrak{X}(M)$ with local expression
\begin{equation}
\label{LocalExprX}
X=f^i\frac{\partial }{\partial x^i},
\quad f^i\in C^\infty (M).
\end{equation}
If $\pi _l^k\colon J^k(\mathbb{R},M)
\to J^l(\mathbb{R},M)$ is the canonical projection
for $k>l$, then
\begin{equation*}
(\pi _{r-1}^r)_\ast X^{(r)}=X^{(r-1)},
\quad \forall X\in \mathfrak{X}(M),
\end{equation*}
and the following exact sequence defines
the subdistribution $\mathfrak{D}^{r,r-1}$:
\begin{equation}
\label{ExactSeq}
0\to \mathfrak{D}_{j_t^r\sigma }^{r,r-1}
\to  \mathfrak{D}_{j_t^r\sigma }^r
\overset{(\pi _{r-1}^r)_\ast }{\longrightarrow }
\mathfrak{D}_{j_t^{r-1}\sigma }^{r-1}
\to  0,
\qquad
\forall j_t^r\sigma \in J^r(\mathbb{R},M).
\end{equation}
\subsection{Stability}
\begin{theorem}
\label{stability}
Let $(M,g)$ be a complete Riemannian connected manifold
and let $\sigma \colon (a,b)\to M$ be a smooth curve
such that $j_{t_0}^{m-1}\sigma \in \mathcal{N}^{m-1}(M)$,
$a\leq t_0\leq b$, with the same notations as
in \emph{section \ref{FandN}}. If $X\in \mathfrak{i}(M,g)$
is a Killing vector field such that
$X_{j_{t_0}^{m-1}\sigma }^{(m-1)}=0$, $m=\dim M$,
then $X=0$.
\end{theorem}
\begin{proof}
Let $x_0=\sigma (t_0)$ and let $U\subset T_{x_0}(M)$
be an open neighbourhood of the origin on which
the exponential mapping $\exp \colon T_{x_0}(M)\to  M$
is a diffeomorphim onto its image. Let $(X_j)_{j=1}^m$
be a $g$-orthonormal basis for $T_{x_0}M$ with dual
basis $(w^i)_{i=1}^m$, $w^i\in T_{x_0}^\ast (M)$
(i.e., $w^i(X_j)=\delta _j^i$) and let
$x^i=w^i\circ (\exp |_{U})^{-1}$, $1\leq i\leq m$,
be the corresponding normal coordinate system.

If $\phi \colon M\to M$ is an affine transformation
of the Levi-Civita connection of $g$ (in particular,
if $\phi $ is an isometry of $g $) leaving the point
$x_0$ invariant, then (cf.\ \cite[VI, Proposition 1.1]{KN}),
$\phi \circ \exp =\exp \circ \phi _\ast $,
$\phi _\ast \colon T_{x_0}(M)\to  T_{x_0}(M)$
being the Jacobian mapping at $x_0$. Hence
$x^i\circ \phi
=(w^i\circ \phi _\ast )\circ (\exp |_{U})^{-1}$.
If $\phi _\ast (X_j)=a_j^iX_i$, then
$w^i\circ \phi _\ast =a_h^iw^h$ and hence,
$x^i\circ \phi =a_h^ix^h$. In particular,
if $\phi _\tau $ is the flow of a Killing vector
field $X$ locally given as in \eqref{LocalExprX},
then
\begin{equation*}
\begin{array}{lll}
x^i\circ \phi _\tau
=a_h^i(\tau )x^h,
& (a_h^i(\tau ))_{h,i=1}^m\in O(m),
& \forall \tau \in \mathbb{R},
\smallskip \\
f^i=b_h^ix^h,
& b_h^i=\dfrac{da_h^i}{d\tau }(0),
&
B=(b_h^i)_{h,i=1}^m\in \mathfrak{so}(m).
\end{array}
\end{equation*}
According to \eqref{F^r}, the assumption
on $X$ in the statement is equivalent
to saying
\begin{equation*}
\frac{d^k(f^i\circ \sigma )}{dt^k}(t_0)
=0,
\quad
1\leq i\leq m,\; 0\leq k\leq m-1,
\end{equation*}
or equivalently,
$b_h^i(d^k(x^h\circ \sigma )/dt^k)(t_0)=0$,
i.e., $B(U_{t_0}^{\sigma ,k})=0$,
$1\leq k\leq m-1$, the tangent
vectors $U_{t_0}^{\sigma ,k}$ being defined
in the formula \eqref{U^sigma,k}. As $B$
is skew-symmetric, we also have
$B(U_{t_0}^{\sigma ,1}\times \cdots
\times U_{t_0}^{\sigma ,m-1})=0$.
Therefore, $B=0$.
\end{proof}
\begin{corollary}
\label{corollary_stability}
On a complete Riemannian connected manifold $(M,g)$
the order of asymptotic stability \emph{(cf.\
\cite{Arnold}, \cite{Kumpera})} of the algebra
of differential invariants, is $\leq m$. Accordingly,
$N_r=(r+1)m+1-\dim \mathfrak{i}(M,g)$,
$\forall r\geq m-1$.
\end{corollary}
\begin{proof}
This follows from the previous theorem and the exact
sequence \eqref{ExactSeq}, taking account of the fact
that $X_{j_t^{m-1}\sigma }^{(m-1)}=0$ if and only if
$X_{j_t^m\sigma }^{(m)}
\in \mathfrak{D}_{j_t^m\sigma }^{m,m-1}$.
\end{proof}
\begin{corollary}
On a complete Riemannian connected manifold $(M,g)$
the distribution $\mathfrak{D}^{m-1}$ takes
its maximal rank on $\mathcal{N}^{m-1}(M)$.
\end{corollary}
\begin{proof}
If $j_t^{m-1}\sigma \in \mathcal{N}^{m-1}(M)$,
then from Theorem \ref{stability} it follows
that the linear map $\mathfrak{i}(M,g)
\to \mathfrak{D}_{j_t^{m-1}\sigma }^{m-1}$,
$X\mapsto X_{j_t^{m-1}\sigma }^{(m-1)}$,
is an isomorphism.
\end{proof}
\subsection{Completeness}
Let a Lie group $G$ act on a manifold $N$.
If the quotient manifold $q\colon N\to N/G$
exists, then the image of the mapping
$q^\ast \colon C^\infty (N/G)\to C^\infty (N)$,
$f\mapsto f\circ q$, is the subalgebra
of $G$-invariant functions, namely
$C^\infty (N)^G=q^\ast C^\infty (N/G)$.
Below we are concerned with the case
$G=\mathcal{I}^0(M,g)$ acting
on $N=J^r(\mathbb{R},M)$ as defined
at the beginning of the section \ref{definitions}.
\begin{definition}
Let $O^r\subseteq J^r(\mathbb{R},M)$ be an invariant
open subset under the natural action of the group
$\mathcal{I}^0(M,g)$. A system of invariant functions
$I_i\colon O^r\to  \mathbb{R}$, $1\leq i\leq \nu $,
is said to be \emph{complete} if the equations
$I_i(j_{t_0}^r\sigma )=I_i(j_{t_0}^r\sigma ^\prime )$,
$1\leq i\leq \nu $, $j_{t_0}^r\sigma $,
$j_{t_0}^r\sigma ^\prime \in O^r$, imply that $\sigma $
and $\sigma ^\prime $ are congruent on a neighbourhood
of $t_0$.
\end{definition}
\begin{proposition}
\label{proposition_diagram}
Let $\nabla $ be a linear connection on $M$.
The mapping $\Phi _{\nabla }^r$ defined
in the formula \emph{(\ref{Phi^r})} makes the diagram
\begin{equation*}
\begin{array}{ccc}
J^r(\mathbb{R},M)
& \overset{\Phi _\nabla ^r}{\longrightarrow }
& \mathbb{R}\times \oplus ^rTM
\smallskip \\
{\scriptstyle\phi ^{(r)}} \!
\downarrow
&  &
\quad
\downarrow
{\scriptstyle(1_{\mathbb{R}},\oplus ^r\phi _\ast )}\\
J^r(\mathbb{R},M)
& \overset{\Phi _{\phi \cdot \nabla }^r}{\longrightarrow }
& \mathbb{R}\times \oplus ^rTM
\end{array}
\end{equation*}
commutative for every $\phi \in \mathrm{Diff}(M)$.
\end{proposition}
\begin{proof}
The proof is a consequence of the formula \eqref{suc}
and Lemma \ref{ecsfrenet}, taking the definition
of $\phi \cdot \nabla $ into account.
\end{proof}
\begin{theorem}
\label{completeness}
Let $(M,g)$ be a complete oriented
connected Riemannian manifold of class $C^\omega $. If
\begin{equation}
\label{complete_system}
I_i\colon (\pi _{m-1}^r)^{-1}\mathcal{F}^{m-1}(M)
\to \mathbb{R},\quad r\geq m,\;1\leq i\leq \nu ,
\end{equation}
is a complete system of invariants, then there exists
a dense open subset $O^r$ in
$(\pi _{m-1}^r)^{-1}\mathcal{F}^{m-1}(M)$ such that
$I_i|_{O^r}$, $1\leq i\leq \nu $, generate the ring
of differential invariants under the group
$\mathcal{I}^0(M,g)$ on an open neighbourhood
$N^r\subseteq O^r$ of every point $j_t^r\sigma \in O^r$,
i.e.,
\begin{equation*}
C^\infty
\left(
N^r
\right) ^{\mathcal{I}^0(M,g)}
=\left(
\left.
I_1
\right\vert _{N^r},\dotsc,
\left. I_\nu
\right\vert _{N^r}
\right)
^\ast C^\infty (\mathbb{R}^\nu ).
\end{equation*}
Conversely, if a system of functions as
in \emph{(\ref{complete_system})} locally generates
the ring of invariants over a dense subset
$\tilde{O}^r
\subseteq (\pi _{m-1}^r)^{-1}\mathcal{F}^{m-1}(M)$,
then it is complete.
\end{theorem}
\begin{proof}
According to Theorem \ref{stability}, we can confine
ourselves to prove the statement for $r=m$.
First of all, we prove that the quotient manifold
\begin{equation*}
q^m\colon (\pi _{m-1}^m)^{-1}\mathcal{F}^{m-1}(M)
\to (\pi _{m-1}^m)^{-1}\mathcal{F}^{m-1}(M)/
\mathcal{I}^0(M,g)
=Q^m
\end{equation*}
exists. To this end, by applying \cite[Theorem 9.16]{Lee},
we only need to prove that the following two conditions hold:
\begin{enumerate}
\item
The isotropy subgroup $\mathcal{I}^0(M,g)_{j_t^m\sigma }$
reduces to the identity map of $M$ for every $j_t^m\sigma $
in $(\pi _{m-1}^m)^{-1}\mathcal{F}^{m-1}(M)$.
\item
$\mathcal{I}^0(M,g)$ acts properly on
$(\pi _{m-1}^m)^{-1}\mathcal{F}^{m-1}(M)$.
\end{enumerate}
The image of $(\pi _{m-1}^m)^{-1}(\mathcal{F}^{m-1})$
by the diffeomorphism $\Phi _\nabla ^m$ is equal
to the subset $U^m\subset \mathbb{R}\times \oplus ^mTM$
of elements $(t,X_1,\dotsc,X_m)$ such that
$(X_1,\dotsc,X_{m-1})$ are linearly independent tangent
vectors. From Proposition \ref{proposition_diagram}
we deduce that an isometry $\phi $ belongs to the isotropy
subgroup $\mathcal{I}^0(M,g)_{j_t^m\sigma }$ of a point
$j_t^m\sigma $ in $(\pi _{m-1}^r)^{-1}\mathcal{F}^{m-1}(M)$,
if and only
${(1_{\mathbb{R}},\oplus ^m\phi _\ast (\sigma (t)))}$
belongs to the isotropy subgroup of the point
$\Phi _\nabla ^m(j_t^m\sigma )=(t,X_1,\dotsc,X_m)\in U^m$.
Hence $\phi =Id_M$ and consequently, $\mathcal{I}^0(M,g)$
acts freely on $(\pi _{m-1}^m)^{-1}(\mathcal{F}^{m-1})$,
thus proving the first item above.

Moreover, if $g_1$ is the Sasakian metric induced by $g$
on $TM$ (e.g., see \cite[1.K]{Besse}, \cite[Section 7]{GK},
\cite[IV, Section 1]{YI}), then $\mathcal{I}^0(M,g)$
acts by isometries of the metric on $J^m(\mathbb{R},M)$
given by
\begin{equation*}
\mathbf{g}^m=\left(
\Phi _\nabla ^m
\right) ^\ast
\left(
dt^2+\sum_{i=1}^m(\mathrm{pr}_i)^\ast g_1
\right) ,
\end{equation*}
where
$\mathrm{pr}_i\colon \mathbb{R}\times \oplus ^mTM\to TM$
is the projection $\mathrm{pr}_i(t,X_1,\dotsc,X_m)=X_i$,
and the image of the mapping
$\mathcal{I}^0(M,g)
\to \mathcal{I}^0(J^m(\mathbb{R},M),\mathbf{g}^m)$,
$\phi \mapsto \phi ^{(m)}$, is closed as it is defined
by the following closed conditions:
\[
\varphi ^\ast (t)=t,
\;
(j^m\sigma )^\ast (\varphi ^\ast \omega )=0,
\quad
\varphi \in \mathcal{I}^0(J^m(\mathbb{R},M),\mathbf{g}^m),
\]
for every $\sigma \in C^\infty (\mathbb{R},M)$
and every contact $1$-form $\omega $
on $J^m(\mathbb{R},M)$, by virtue of
\cite[Theorem 3.1]{Yamaguchi}.
From \cite[5.2.4. Proposition]{PalaisTerng}
we conclude the second item above.

The invariant functions
$I_i\colon (\pi _{m-1}^m)^{-1}
\mathcal{F}^{m-1}(M)\to \mathbb{R}$,
$1\leq i\leq \nu $, induce smooth functions
on the quotient manifold,
$\bar{I}_i\colon Q^m\to \mathbb{R}$.
As $I_1,\dotsc,I_{\nu }$ is a complete system
of invariants, the mapping
$\Upsilon \colon Q^m\to \mathbb{R}^{\nu }$
whose components are $\bar{I}_1,\dotsc,\bar{I}_\nu $,
is injective.

The same argument as in the proof
of Lemma \ref{lemma_rank} states the following
property: If $\phi \colon N\to  N^\prime $ is
an smooth mapping, then the subset of the points
$x\in N$ for which there exists an open
neighbourhood $U(x)\subseteq N$ such that
$\phi |_{U(x)}$ is a mapping of constant rank,
is a dense open subset in $N$. Hence an injective
smooth map $\phi \colon N\to  N^\prime $ is
an immersion on a dense open subset in $N$
(cf.\ \cite[Theorem 7.15-(b)]{Lee}).
By applying this result to $\Upsilon $,
we conclude the existence of a dense open
subset $\bar{O}^m\subseteq Q^m$ such that
$\Upsilon |_{\bar{O}^m}$ is an injective
immersion. Hence for every
$q^m(j_t^m\sigma )\in \bar{O}^m$ there exists
a system of coordinates on $Q^m$ defined
on an open neighbourhood of $q^m(j_t^m\sigma )$
constituted by some functions
$\bar{I}_{i_1},\dotsc,\bar{I}_{i_k}$,
$k=\dim Q^m$. As $(q^m)^\ast C^\infty (Q^m)$
can be identified to the ring of differential
invariants, we can take
$O^m=(q^m)^{-1}(\bar{O}^m)$.

Conversely, if
$j_{t_0}^m\sigma ,j_{t_0}^m\sigma ^\prime
\in \tilde{O}^m$ are such that
$q^m(j_{t_0}^m\sigma ^\prime )
\neq q^m(j_{t_0}^m\sigma )$, then there exists
$\rho \in C^\infty (q^m\tilde{O}^m)$
satisfying $\rho (q^m(j_{t_0}^m\sigma ))=0$,
$\rho (q^m(j_{t_0}^m\sigma ^\prime ))=1$.
As $\rho \circ q^m$ is an invariant function
on $\tilde{O}^m$ by virtue of the hypothesis
there exists $f\in C^\infty (\mathbb{R}^\nu )$
such that $\rho \circ (q^m)=f\circ
\left(
\left.
I_1
\right\vert _{\tilde{O}^m},
\dotsc,
\left.
I_\nu
\right\vert _{\tilde{O}^m}
\right) $. Hence an index $i$ must exists
for which $I_i(j_{t_0}^m\sigma )
\neq I_i(j_{t_0}^m\sigma ^\prime )$,
thus proving that $I_1,\cdots,I_\nu $
is a complete system of invariants.
\end{proof}
\begin{remark}
If the injective immersion
$\Upsilon \colon \bar{O}^m\to \mathbb{R}^\nu $
is a closed map, then one has
\begin{equation*}
C^\infty
\left(
O^m
\right) ^{\mathcal{I}^0(M,g)}
=\left(
I_1,\dotsc,I_\nu
\right) ^\ast
C^\infty (\mathbb{R}^\nu ).
\end{equation*}
\end{remark}
\subsection{Generating complete systems of invariants}
\begin{theorem}
\label{Generating_Invariants}
For every $r\in \mathbb{N}$, let $k_r$ be the maximal
number of generically functionally independent $r$-th
order invariants not belonging to the closed---in
the $C^\infty $ topology---subalgebra generated by
the invariants of order $<r$, and their derivatives
with respect to the operator $D_t$. Then
\begin{eqnarray}
k_r &=& N_r-1-\sum _{i=0}^{r-1}(r+1-i)k_i,
\label{formula1} \\
m &=& \sum _{i=0}^mk_i.
\label{formula2}
\end{eqnarray}
Hence for every complete Riemannian manifold
of dimension $m$, there exist $m$ generically
independent invariants generating a complete
system of invariant functions by adding
their derivatives with respect to $D_t$
for every order $r\leq m$. Moreover,
$k_r=0$, $\forall r>m$.
\end{theorem}
\begin{proof}
Let $t,I_i^0\in C^\infty (M)$,
$1\leq i\leq k_0\leq m$, $N_0=1+k_0$, be
a maximal system of invariant functions
of order zero. (If $(M,g)$ is not homogeneous,
then there exist zero-order differential
invariants independent of $t$; because
of this the proof must start on this order.)
Therefore, the rank of the Jacobian matrix
$\mathcal{J}^0(I_1^0,\dotsc,I_{k_0}^0)
=\left(
\partial I_i^0/\partial x^j
\right) _{1\leq j\leq m}^{1\leq i\leq k_0}$
must be maximal; namely,
$\operatorname{rk}\mathcal{J}^0
(I_1^0,\dotsc,I_{k_0}^0)=k_0$.
Moreover, one has
\begin{equation*}
\frac{\partial (D_tf)}{\partial x_{r+1}^i}
=\frac{\partial f}{\partial x_r^i},
\quad
\forall f\in C^\infty (J^r(\mathbb{R},M)),
\end{equation*}
as $\left[
\partial /\partial x_{r+1}^i,D_t
\right]
=\partial /\partial x_r^i$,
$\forall r\in \mathbb{N}$.
Hence the Jacobian matrix of the functions
$I_1^0,
\dotsc,
I_{k_0}^0,
D_tI_1^0,
\dotsc,
D_tI_{k_0}^0$ on $J^1(\mathbb{R},M)$ is of the form
\begin{equation*}
\mathcal{J}^1(I_1^0,
\dotsc,I_{k_0}^0,D_tI_1^0,
\dotsc,D_tI_{k_0}^0)
=\left(
\begin{array}{cc}
\left(
\frac{\partial I_i^0}{\partial x^j}
\right) &
0\smallskip \\
\star &
\left(
\frac{\partial I_i^0}{\partial x^j}
\right)
\end{array}
\right) ,
\end{equation*}
and $\operatorname{rk}\mathcal{J}^1
(I_1^0,
\dotsc,
I_{k_0}^0,
D_tI_1^0,
\dotsc,
D_tI_{k_0}^0)=2k_0\leq N_1-1$.
We can thus complete the previous
system with $k_1=N^1-1-2k_0$ new functionally
independent invariants
$I_1^1,\dotsc,I_{k_1}^1$, in such a way
that the Jacobian matrix of the full
system is as follows:
\begin{equation*}
\mathcal{J}^1(I_1^0,
\dotsc,I_{k_0}^0,D_tI_1^0,
\dotsc,D_tI_{k_0}^0,I_1^1,
\dotsc,I_{k_1}^1)
=\left(
\begin{array}{cc}
\left(
\frac{\partial I_i^0}{\partial x^j}
\right)
& 0\smallskip \\
\star
&
\left(
\frac{\partial I_i^0}{\partial x^j}
\right)
\smallskip
\\
\star
& \left(
\frac{\partial I_i^1}{\partial x_1^j}
\right)
\end{array}
\right) ,
\end{equation*}
with $\operatorname{rk}\mathcal{J}^1
(I_1^0,
\dotsc,I_{k_0}^0,D_tI_1^0,
\dotsc,D_tI_{k_0}^0,I_1^1,
\dotsc,I_{k_1}^1)
=2k_0+k_1=N_1-1$. Let us consider
the second-order invariants
\begin{eqnarray*}
&& I_1^0,
\dotsc,I_{k_0}^0, \\
&& D_tI_1^0,
\dotsc,D_tI_{k_0}^0,I_1^1,
\dotsc,I_{k_1}^1,
\\
&& D_t^2I_1^0,
\dotsc,D_t^2I_{k_0}^0,D_tI_1^1,
\dotsc,D_tI_{k_1}^1,
\end{eqnarray*}
the Jacobian matrix of which is
\begin{equation*}
\left(
\begin{array}{ccc}
\left(
\frac{\partial I_i^0}{\partial x^j}
\right)
& 0 & 0\
\smallskip \\
\star &
\left(
\frac{\partial I_i^0}{\partial x^j}
\right)
& 0
\smallskip
\\
\star &
\left(
\frac{\partial I_i^1}{\partial x_1^j}
\right) & 0
\smallskip \\
\star & \star &
\left(
\frac{\partial I_i^0}{\partial x^j}
\right)
\smallskip \\
\star & \star &
\left( \frac{\partial I_i^1}{\partial x_1^j}
\right)
\end{array}
\right)
\end{equation*}
and its rank is equal to $3k_0+2k_1$.
Hence we need to choose $k_2$ new second-order
invariants $I_1^2,\dotsc,I_{k_2}^2$, with
$k_2=N_2-1-(3k_0+2k_1)$, such that the matrix
\begin{equation*}
\left(
\begin{array}{ccc}
\left(
\frac{\partial I_i^0}{\partial x^j}
\right)
& 0 & 0
\smallskip \\
\star &
\left(
\frac{\partial I_i^0}{\partial x^j}
\right)
& 0
\smallskip
\\
\star &
\left(
\frac{\partial I_i^1}{\partial x_1^j}
\right)
& 0
\smallskip \\
\star & \star &
\left(
\frac{\partial I_i^0}{\partial x^j}
\right)
\smallskip \\
\star & \star &
\left(
\frac{\partial I_i^1}{\partial x_1^j}
\right)
\smallskip \\
\star & \star &
\left(
\frac{\partial I_i^2}{\partial x_2^j}
\right)
\smallskip
\end{array}
\right)
\end{equation*}
is of maximal rank, i.e., $3k_0+2k_1+k_2=N_2-1$.
Proceeding step by step in the same way, we conclude
that the formula \eqref{formula1} in the statement
holds true. Moreover, from this formula we obtain
$N_{r}-N_{r-1}=\sum_{i=0}^rk_i$. According
to Corollary \ref{corollary_stability}, we have
$N_r=(r+1)m+1-\dim \mathfrak{i}(M,g)$,
$\forall r\geq m-1$, and then
$\sum _{i=0}^mk_i=N_m-N_{m-1}=m$, thus proving
the formula \eqref{formula2} in the statement
and finishing the proof.
\end{proof}
\begin{remark}
\label{remark_Green}
Following the same notations as in \cite{Green},
let $H_{r+1}=\mathcal{I}^0(M,g)_{j_t^r\sigma }$
be the isotropy subgroup of a point
$j_t^r\sigma $ belonging to the open subset
$\mathcal{U}^r$ where the distribution
$\mathfrak{D}^r$ is of constant rank
(see Lemma \ref{lemma_rank}). In \cite{Green}
the following formula is mentioned:
$k_r=\dim H_{r-1}+\dim H_{r+1}-2\dim H_r$
in the homogeneous case, i.e., $M=G/H$.
Nevertheless, this formula holds
on an arbitrary Riemannian manifold
and it is an easy consequence of the formulas
\eqref{N_r} and \eqref{formula1}.
In fact, one has,
$\dim H_{r+1}=\dim \mathcal{I}^0(M,g)
-\operatorname{rk}
\left.
\mathfrak{D}^r
\right\vert _{\mathcal{U}^r}$.
Hence
\begin{equation*}
\dim H_{r-1}+\dim H_{r+1}-2\dim H_r
=N_{r-2}+N_r-2N_{r-1}=k_r.
\end{equation*}
\end{remark}
\section{Congruence on symmetric manifolds}
\begin{theorem}
\label{CCS}
Let $(M,g)$, $(\bar{M},\bar{g})$ be two
locally symmetric Riemannian manifolds
of the same dimension,
$m=\dim M=\dim \bar{M}$, and let
$\sigma \colon (a,b)\to M$,
$\bar{\sigma}\colon (a,b)\to \bar{M}$ be
two Frenet curves. If $x_0=\sigma (t_0)$,
$\bar{x}_0=\bar{\sigma}(t_0)$, $a<t_0<b$,
then $\sigma $ and $\bar{\sigma}$ are
congruent on some neigbourhoods $U$
and $\bar{U}$ of $x_0$ and $\bar{x}_0$,
respectively if, and only if,
the following conditions hold:
\begin{enumerate}
\item[\emph{1.}]
For every $j\in \mathbb{N}$
and every $0\leq i\leq m-1$,
\begin{equation}
\label{Symmetric0}
\kappa _i^\sigma (t)
=\kappa _i^{\bar{\sigma}}(t) , \quad
\left\vert t-t_0
\right\vert <\varepsilon .
\end{equation}
\item[\emph{2.}]
For every $i,j,k,l=1,\dotsc,m$,
\begin{equation}
\label{Symmetric}
R(X_i^\sigma ,
X_j^\sigma ,
X_k^\sigma ,
\omega _\sigma ^l)(x_0)
=\bar{R}(X_i^{\bar{\sigma }},
X_j^{\bar{\sigma }},
X_k^{\bar{\sigma }},
\omega _{\bar{\sigma }}^l)(\bar{x}_0)),
\end{equation}
$\left(
X_1^\sigma ,\dotsc,X_m^\sigma
\right) $,
$\left(
X_1^{\bar{\sigma }},\dotsc,X_m^{\bar{\sigma }}
\right) $
being the Frenet frames of $\sigma $,
$\bar{\sigma }$, with dual coframes
$\left(
\omega _\sigma ^1,\dotsc,\omega _\sigma ^m
\right) $,
$\left(
\omega _{\bar{\sigma }}^1,
\dotsc,
\omega _{\bar{\sigma }}^m
\right) $, and $R$, $\bar{R}$
the curvature tensors of $(M,g)$,
$(\bar{M},\bar{g})$, respectively.
\end{enumerate}
\end{theorem}
\begin{proof}
From \cite[VI, Theorem 7.7]{KN} we know
every locally symmetric Riemannian manifold
is analytic, and its associated Levi-Civita
connection $\nabla $ also is, so we can apply
Corollary \ref{CorolCGC}. In this case,
the conditions \eqref{DC} are simply
\eqref{Symmetric}.
\end{proof}
\begin{theorem}
\label{locSIM}
Let $(M,g)$ be an arbitrary Riemannian manifold
verifying the following property: Two Frenet curves
$\sigma ,\bar{\sigma}\colon (a,b)\to M$,
$\sigma (t_0)=\bar{\sigma }(t_0)=x_0$,
are congruent on some neighbourhood of $x_0$
(preserving the orientation if $\dim M$ is even
and reversing the orientation if $\dim M$ is odd)
if and only if the conditions
\emph{\eqref{Symmetric0}} and
\emph{\eqref{Symmetric}}
of \emph{Theorem \ref{CCS}} hold. Then,
$(M,g)$ is locally symmetric.
\end{theorem}
\begin{proof}
Let us fix an orientation on a neighbourhood
of $x_0\in M$, and let $(v_i)_{i=1}^m$ be
a positive orthonormal basis of $T_{x_0}M$.
Let $\kappa _j\in C^\infty (t_0-\delta ,t_0+\delta )$,
$0\leq  j\leq  m-1$, $\delta >0$, be functions
such that $\kappa _j>0$ for $0\leq j\leq m-2$.
From Theorem \ref{ecsfrenet} we know there exist
two Frenet curves $\sigma ,\bar{\sigma }\colon
(t_0-\varepsilon ,t_0+\varepsilon )\to M$,
$0<\varepsilon <\delta $, such that
$\sigma (t_0)=\bar{\sigma }(t_0)=x_0$ and
$X_i^\sigma (t_0)=-X_i^{\bar{\sigma }}(t_0)=v_i$
(hence $\omega _{\sigma }^i(t_0)
=-\omega _{\bar{\sigma }}^i(t_0)$)
for $1\leq  i\leq  m$, with the same curvatures
$\kappa _j$, $0\leq  j\leq  m-1$.
If $\dim M$ is odd, we considered the opposite
orientation to construct $\bar{\sigma }$.
According to this choice of orientations constructing
the Frenet curves $\sigma $ and $\bar{\sigma }$,
for every $i,j,k,l=1,\dotsc,m$, we have
\begin{eqnarray*}
R\left(
X_i^\sigma ,
X_j^\sigma ,
X_k^\sigma ,
\omega _\sigma ^l
\right) (x_0)
&=&
R\left(
-X_i^\sigma ,
-X_j^\sigma ,
-X_k^\sigma ,
-\omega _\sigma ^l
\right) (x_0) \\
&=&
R\left(
X_i^{\bar{\sigma }},
X_j^{\bar{\sigma }},
X_k^{\bar{\sigma }},
\omega _{\bar{\sigma}}^l
\right) (x_0).
\end{eqnarray*}
From the hypothesis, an open neighbourhood
$U$ of the image of $\sigma $ and an isometric
embedding $\phi \colon U\to  M$ (leaving $x_0$ fixed)
exist such that $\bar{\sigma }=\phi \circ \sigma $.
Moreover $\phi _\ast (X_i^\sigma (t_0))
=X_i^{\bar{\sigma}}(t_0)=-X_i^\sigma (t_0) $,
$1\leq  i\leq  m$. Thus $\phi _\ast =-Id_{T_{x_0}M}$.
Since $x_0\in M$ is arbitrary, we can conclude.
\end{proof}
\begin{remark}
\label{remark_2.3}
Let $(M,g)$ be a Riemannian symmetric space.
Let $G=\mathfrak{I}^0(M,g)$ be the connected component
of the identity in the group of isometries and let $H$
be the isotropy subgroup of the point $x_0$. As usual,
we set $\mathfrak{g}=\mathfrak{h}\oplus \mathfrak{m}$
and $\mathfrak{m}$ is identified to $T_{x_0}M$.
According to \cite[XI, Theorem3.2]{KN}, the following
formula holds: $R(X,Y)Z=-[[X,Y],Z]$ for every
$X,Y,Z\in \mathfrak{m}$. Therefore, if a basis
$(Y_1,\dotsc,Y_\mu )$ for $\mathfrak{h}$ is fixed, then
\begin{equation*}
\begin{array}{ll}
\lbrack X_i^\sigma ,X_j^\sigma ]
=c_{ij}^\alpha Y_\alpha ,
& c_{ij}^\alpha \in \mathbb{R}, \\
\lbrack Y_\alpha ,X_k^\sigma ]
=d_{\alpha k}^hX_h^\sigma ,
& d_{\alpha k}^h\in \mathbb{R}.
\end{array}
\end{equation*}
Hence, $g([[X_i^\sigma ,X_j^\sigma ],
X_k^\sigma ],X_l^\sigma )
=c_{ij}^\alpha d_{\alpha k}^l$. Consequently,
the condition \eqref{Symmetric} can be rewritten
as $c_{ij}^\alpha d_{\alpha k}^l
=\bar{c}_{ij}^\alpha \bar{d}_{\alpha k}^l$,
with the obvious notations for the curve
$\bar{\sigma }$.
\end{remark}
\begin{remark}
\label{remark_2.4}
According to Theorem \ref{CCS}, on a locally symmetric
Riemannian manifold\ $(M,g)$ the $\frac{1}{2}m(m+1)$
functions $I_{ij}\circ \mathtt{f}_{M}$, $1\leq i<j\leq m$,
and $\varkappa _k$, $0\leq k\leq m-1$, defined
in the formulas \eqref{Ii1...i_j+3_i} and \eqref{varkappa}
of Remark \ref{remark_2}, respectively, constitute
a complete system of differential invariants in the sense
that two curves with values in $(M,g)$ are congruent
if and only if the functions $I_{ij}\circ \mathtt{f}_{M}$,
$\varkappa _k$ take the same values on both curves.
\end{remark}
\begin{example}
For $M=\mathbb{C}P^n$, from the formula for the curvature
tensor in \cite[XI, p.\ 277]{KN}, the Riemann curvature reads
\begin{eqnarray*}
R_4\left(
X_i,X_j,X_k,X_l
\right)
\!\!\! &=& \!\!\!
\tfrac{c}{4}
\left\{
\delta _{jl}\delta _{ik}
-\delta _{jk}\delta _{il}
+g\left(
X_j,JX_{l}
\right)
g
\left(
JX_k,X_i
\right)
\right. \\
&& \left.
-g\left(
X_j,JX_k
\right)
g
\left(
JX_{l},X_i
\right)
+2g
\left(
X_k,JX_{l}
\right)
g
\left(
JX_j,X_i
\right)
\right\} ,
\end{eqnarray*}
$J$ being the canonical complex structure.
In particular,
\begin{equation*}
R_4
\left( X_i,X_j,X_i,X_j
\right)
=\tfrac{c}{4}
\left(
1-\delta _{ij}+3\omega
\left(
X_i,X_j
\right) ^2
\right) ,
\end{equation*}
where $\omega $ is the canonical K\"{a}hler $2$-form
in $\mathbb{C}P^n$. Therefore, if we define the funcions
$\varpi _{ij}\colon (\pi _{m-1}^m)^{-1}\mathcal{F}^{m-1}(M)
\subset J^m(\mathbb{R},M)\to \mathbb{R}$, $1\leq i<j\leq m$,
as
\begin{equation*}
\varpi _{ij}(j_{t_0}^m\sigma )
=\omega (X_i^\sigma (t_0),X_j^\sigma (t_0)),
\end{equation*}
the family $\{\varpi _{ij}\}_{i<j}$ together with
$\{\varkappa _i\}_{i=0}^{m-1}$ constitute
a complete system of differential invariants.
\end{example}
\section{Congruence on constant curvature manifolds}
\label{ccc}
\begin{theorem}
\label{CTE}
Two Frenet curves
$\sigma ,\bar{\sigma}\colon (a,b)\to (M,g)$
taking values in an oriented Riemannian manifold
of constant curvature are congruent on some
neigbourhoods $U$ and $\bar{U}$ of $x_0=\sigma (t_0)$
and $\bar{x}_0=\bar{\sigma}(t_0)$,
$a<t_0<b$, respectively if and only,
$\kappa _i^\sigma (t) =\kappa _i^{\bar{\sigma}}(t)$
for $0\leq i\leq m-1$ and small enough $|t-t_0|$.
Conversely, if on an oriented Riemannian manifold
$(M,g)$ two arbitrary Frenet curves
$\sigma ,\bar{\sigma }\colon (a,b)\to (M,g)$
are congruent on some neighbourhhods of
$x_0=\sigma (t_0)$, $\bar{x}_0=\bar{\sigma }(t_0)$,
$a<t_0<b$, if and only if
$\kappa _i^\sigma (t) =\kappa _i^{\bar{\sigma }}(t)$
for $0\leq i\leq m-1$ and small enough $|t-t_0|$,
then $(M,g)$ is a manifold of constant curvature.
\end{theorem}
\begin{proof}
On a manifold of constant curvature $k$ one has
(\cite[V. Corollary 2.3]{KN}):
$R(X,Y)Z=k(g(Y,Z)X-g(X,Z)Y)$; hence
\begin{equation*}
R\left(
X_i^\sigma ,
X_j^\sigma ,
X_k^\sigma ,
\omega _\sigma ^l
\right)
=k\left(
\delta _{jk}\delta _i^l-\delta _{ik}\delta _j^l
\right) .
\end{equation*}
It follows that on a manifold of constant curvature
the equation \eqref{Symmetric} holds identically.
The first part of the statement thus follows
from Theorem \ref{CCS}.

Let $(v_1,\dotsc,v_m)$, $(w_1,\dotsc,w_m)$ be two
positively-oriented orthonormal bases in $T_{x_0}M$.

From Theorem \ref{ecsfrenet} there exist two Frenet
curves $\sigma ,\bar{\sigma}
\colon (t_0-\varepsilon ,t_0+\varepsilon )\to M$,
$\varepsilon >0$, such that,

\smallskip

\noindent i) $\sigma (t_0)=\bar{\sigma }(t_0)=x_0$,

\noindent ii) $X_i^\sigma (t_0)=v_i$,
$X_i^{\bar{\sigma }}(t_0)=w_i$, $1\leq i\leq m$,

\noindent iii) $\kappa _i^\sigma (t)
=\kappa _i^{\bar{\sigma }}(t)$, $0\leq i\leq m-1$,
$|t-t_0|<\varepsilon $.

\smallskip

By virtue of the hypothesis, there exists an isometry
$\phi $ defined on a neighbourhood of $x_0$, fixing $x_0$,
such that $\phi \circ \sigma =\bar{\sigma }$. Hence
$\phi _\ast (v_i)=w_i$, $1\leq i\leq m$, and accordingly
$M$ is an isotropic manifold (e.g., see \cite{Thirring})
and therefore of constant curvature.
\end{proof}
Two $r$-jets of curves
$j_{t_0}^r\sigma ,j_{t_0}^r\bar{\sigma }
\in J_{t_0}^r\left( \mathbb{R},M\right) $
are said to be $r$\emph{-congruent}, which is denoted
by $j_{t_0}^r\sigma \sim _r\,j_{t_0}^r\bar{\sigma }$,
if there exists an isometry $\phi $ defined
on a neighbourhood of $\sigma (t_0)$ such that
$\phi ^{(r)}(j_{t_0}^r\sigma )=j_{t_0}^r\bar{\sigma }$.
It is obvious that ``to be $r$-congruent''
is an equivalence relation on
$J_{t_0}^r\left( \mathbb{R},M\right) $.
\begin{theorem}
For every $1\leq r\leq m$, let $x_{hi}=x_{ih}$,
$h,i=1,\dotsc,r$, be the standard coordinate system
in $S^2(\mathbb{R}^r)$. Given an oriented Riemannian
manifold $(M,g)$ of constant curvature, the mapping
$f_M^r\colon J_{t_0}^r(\mathbb{R},M)
\to S^2(\mathbb{R}^r)$ with components
$(x_{hi}\circ f_M^r)(j_{t_0}^r\sigma )
=g(\nabla _{T^\sigma }^{h-1}T^\sigma ,
\nabla _{T^\sigma }^{i-1}T^\sigma )(t_0)$,
$h,i=1,\dotsc,r$, induces a homeomorphism between
$J_{t_0}^r(\mathbb{R},M)/\sim _r$
and the submanifold with corners
$Q^r\subset S^2(\mathbb{R}^r)$
of the positive semidefinite symmetric matrices.
Hence, if $M$ is simply connected and complete, then
$J_{t_0}^r(\mathbb{R},M)/\mathfrak{I}(M,g)\cong Q^r$.
\end{theorem}
\begin{proof}
A symmetric matrix $A\in S^2(\mathbb{R}^r)$ belongs
to $Q^r$ if and only if all its principal minors
are nonnegative, thus proving that $Q^r$ is a submanifold
with corners. Certainly, the following properties hold:
i) $f_M^r(J_{t_0}^r(\mathbb{R},M))\subset Q^r$,
and ii) $f_M^r(j_{t_0}^r\sigma )
=f_M^r(j_{t_0}^r\bar{\sigma })$ if $j_{t_0}^r\sigma $
and $j_{t_0}^r\bar{\sigma }$ are $r$-congruent.

Firstly, we prove that
$f_M^r\colon J_{t_0}^r(\mathbb{R},M)\to Q^r$,
$r\leq m$, is surjective. Given
$A=(a_{hi})_{h,i=1}^r\in Q^r$, let
$A^\prime =(a_{hi}^\prime )_{h,i=1}^m\in Q^m$
be the matrix obtained by letting
$a_{hi}^\prime =a_{hi}$ for $i\leq r$,
and $a_{hi}^\prime =\delta _{hi}$ for
$r+1\leq i\leq m$. If there exists
$j_{t_0}^m\sigma $ such that
$f_M^m(j_{t_0}^m\sigma )=A^\prime $,
then $f_{M}^r(j_{t_0}^r\sigma )=A$. Therefore,
we can assume $r=m$. In this case, as
$A\in Q^m$, there exists a $m\times m$ matrix
$B=(b_{hi})_{h,i=1}^m$ such that $A=BB^t$.
Let $(v_1,\dotsc,v_m)$ be a positively
oriented orthonormal basis in $T_{x_0}M$.
It suffices to prove the existence of a curve
$\sigma \colon (t_0-\varepsilon ,t_0+\varepsilon )\to M$,
such that, $\sigma (t_0)=x_0$ and
$(\nabla _{T^\sigma }^{h-1}T^\sigma )_{t_0}
=\sum_{i=1}^mb_{hi}v_i$, $1\leq  h\leq  m$.
This fact is a consequence of the existence
and uniqueness theorem for ordinary differential
systems, taking the formulas \eqref{suc} into account.

If $f_M^r(j_{t_0}^r\sigma )=f_M^r(j_{t_0}^r\tau )$,
then the Gram matrices of the systems
$(\nabla _{T^\sigma }^{h-1}T^\sigma )_{t_0}$,
$(\nabla _{T^\tau }^{h-1}T^\tau )_{t_0}$,
$1\leq h\leq  r$, are equal and consequently,
there exists a linear isometry
$L\colon (T_{x_0}M,g_{x_0})\to (T_{x_0}M,g_{x_0})$
such that,
$L(\nabla _{T^\sigma }^{h-1}T^\sigma )_{t_0}
=(\nabla _{T^\tau }^{h-1}T^\tau )_{t_0}$
for $1\leq  h\leq  r$. As $(M,g)$ is of constant
curvature, there exists a local isometric embedding
$\phi $ of $(M,g)$ such that, $\phi (x_0)=x_0$,
$\phi _\ast (x_0)=L$, and one thus deduces
$\phi ^{(r)}(j_{t_0}^r\sigma )=j_{t_0}^r\tau $,
i.e., $j_{t_0}^r\sigma \sim _r\,j_{t_0}^r\tau $.
Hence, $f_M^r$ induces a continuous bijection
$\bar{f}_M^r\colon J_{t_0}^r(\mathbb{R},M)/\sim _r\to Q^r$,
which is a homeomorphism as $f_M^r$ is a proper map.
\end{proof}
\begin{remark}
The points in the interior of $Q^r$ correspond
with $r$-jets of curves such that
$(T_{t_0},
(\nabla _tT)_{t_0},
\dotsc,
(\nabla _t^{r-1}T _{t_0})$ are linearly independent.
In particular, $Q^{m-1}\backslash \partial Q^{m-1}$
correspond with the $(m-1)$-jets of Frenet curves.
\end{remark}
\begin{corollary}
The ring of differential invariants of order $r\leq m$
over a manifold of constant curvature is isomorphic
to $C^\infty (Q^r)$.
\end{corollary}
\section{Some examples}
\subsection{Euclidean space}
\begin{theorem}
\label{Euclidean_invariants}
If $g=(dx^1)^2+\ldots +(dx^m)^2$ is the Euclidean
metric on $\mathbb{R}^m$ and $j_{t_0}^{m-1}\sigma
\in J^{m-1}(\mathbb{R},\mathbb{R}^m)$ is a Frenet
jet, then
\begin{equation*}
\dim \mathfrak{D}_{j_{t_0}^r\sigma }^r
=\left\{
\begin{array}{rl}
m+\tfrac{1}{2}(2m-r-1)r,
& 1\leq r\leq m-1,
\smallskip \\
m+\tbinom{m}{2},
& r\geq m.
\end{array}
\right.
\end{equation*}
Hence
\begin{equation*}
N_r
=\left\{
\begin{array}{rl}
1+\tbinom{r+1}{2},
& 1\leq r\leq m-1,
\smallskip \\
1+\frac{1}{2}m(2r-m+1),
& r\geq m.
\end{array}
\right.
\end{equation*}
Moreover, the functions $t$ and
$D_t^h\varkappa _i
\colon (\pi _{m-1}^m)^{-1}
\mathcal{F}^{m-1}(M)
\subset J^m(\mathbb{R},M)
\to \mathbb{R}$, $h+i\leq m-1$,
where
$\varkappa _0,\dotsc,\varkappa _{m-1}$
are defined in the formula
\emph{\eqref{varkappa}}, are a basis
for the algebra of differential invariants
of order $\leq m$.
\end{theorem}
\begin{proof}
The Lie algebra $\mathfrak{i}(\mathbb{R}^m,g)$
has the basis constituted by the $m$ translations
$T_h=\partial /\partial x^h$, $1\leq h\leq m$,
and the $\frac{1}{2}m(m-1)$\ rotations
$R_{ij}=x^j\partial /\partial x^i
-x^i\partial /\partial x^j$,
$1\leq i\leq j\leq m$. As $T_1,\dotsc,T_{m}$
span $\mathfrak{D}^0$ over
$C^\infty (\mathbb{R}^m)$, one has
$\dim \mathfrak{D}_{(t,x)}^0=m$,
$\forall (t,x)\in J^0(\mathbb{R},M)
=\mathbb{R}\times M$.
Let $(x^i)_{i=1}^m$ be the only Euclidean
coordinates centred at $x_0=\sigma (t_0)$ such
that
$X_i^\sigma (t_0)=(\partial /\partial x^i)_{x_0}$,
$1\leq i\leq m$, where $(X_i^\sigma )_{i=1}^m$
denotes the Frenet frame for $\sigma $.
Let $\mathrm{prol}_{j_{t_0}^r\sigma }^r
\colon \mathfrak{i}(\mathbb{R}^m,g)
\to \mathfrak{D}_{j_{t_0}^r\sigma }^r$
be the mapping
$\mathrm{prol}_{j_{t_0}^r\sigma }^r(X)
=X_{j_{t_0}^r\sigma }^{(r)}$. According to
the formula \eqref{F^r} for the jet prolongation
of a vector field, a translation $T$ belongs to
$\ker \mathrm{prol}_{j_{t_0}^r\sigma }^r$
if and only if $T=0$, and the rotations
$R=a_j^ix^j\partial /\partial x^i$,
$A=(a_j^i)_{i,j=1}^m\in \mathfrak{so}(m)$,
belonging to
$\ker \mathrm{prol}_{j_{t_0}^r\sigma }^r$
are characterized by
\begin{equation*}
\sum _{k=0}^r\sum _{i,j=1}^m
a_j^ix_k^j(j_{t_0}^r\sigma )
\left.
\frac{\partial }{\partial x_k^i}
\right\vert _{j_{t_0}^r\sigma }=0,
\end{equation*}
or equivalently, $A\cdot U_{t_0}^{\sigma ,k}=0$,
$1\leq k\leq r$, with the same notations
as in the formula \eqref{U^sigma,k}.
As the metric is flat, one has
$U^{\sigma ,k}
=\nabla _{T^\sigma }^{k-1}T^\sigma $
for every $k\geq 1$, and we conclude
\begin{eqnarray*}
\left\langle
U_{t_0}^{\sigma ,1},\dotsc,U_{t_0}^{\sigma ,r}
\right\rangle
&=&
\left\langle
X_1^\sigma (t_0),
\dotsc,
X_r^\sigma (t_0)
\right\rangle \\
&=&
\left\langle
\partial /\partial x^1|_{x_0},
\dotsc,
\partial /\partial x^r|_{x_0}
\right\rangle .
\end{eqnarray*}
Hence $R\in \ker \mathrm{prol}_{j_{t_0}^r\sigma }^r$
if and only the kernel of the matrix $A$ contains
the subspace $\mathbb{R}^r\subset \mathbb{R}^m$
of $m$-uples whose last $m-r$ components vanish;
it thus follows that
$\ker \mathrm{prol}_{j_{t_0}^r\sigma }^{(r)}$
is generated by $R_{ij}$ for $r+1\leq i<j\leq m$,
when $r<m-1$; for $r\geq m-1$, we have
$\ker \mathrm{prol}_{j_{t_0}^r\sigma }^r=\{ 0\} $.
Hence
\begin{eqnarray*}
\dim \mathfrak{D}_{j_{t_0}^r\sigma }^r
&=& \tbinom{m}{2}-\tbinom{m-r}{2} \\
&=& \tfrac{1}{2}(2m-r-1)r.
\end{eqnarray*}
Moreover, if
$\sum _{0\leq h+i\leq m-1}
\lambda _h^id(D_t^h\varkappa _i)_{j_{t_0}^m\sigma }=0$,
then by applying this equation to
$\partial /\partial x_m^j|_{j_{t_0}^m\sigma }$
and taking the formula
$[\partial /\partial x_r^j,D_t]
=\partial /\partial x_{r-1}^j$ into account,
one obtains
\begin{eqnarray*}
0 &=& \sum\nolimits_{i=0}^{m-1}
\lambda _{m-1-i}^i
\frac{\partial }{\partial x_{m}^j}
\left(
D_t^{m-1-i}\varkappa _i
\right)
(j_{t_0}^m\sigma ) \\
&=& \lambda _0^{m-1}
\frac{\partial \varkappa _{m-1}}{\partial x_{m}^j}
(j_{t_0}^m\sigma )
+\lambda _1^{m-2}
\frac{\partial \varkappa _{m-2}}{\partial x_{m-1}^j}
(j_{t_0}^m\sigma )
+\ldots +\lambda _{m-1}^0
\frac{\partial \varkappa _0}{\partial x_1^j}
(j_{t_0}^m\sigma ) \\
&=& \sum_{i=1}^mk_j^i
\lambda _{m-i}^{i-1},\quad 1\leq j\leq m,
\end{eqnarray*}
where
\begin{equation*}
k_j^i=\dfrac{\partial \varkappa _{i-1}}
{\partial x_i^j}(j_{t_0}^m\sigma ),
\quad i,j=1,\dotsc,m.
\end{equation*}
By proceeding by recurrence on $h+i$, it suffices
to prove $\lambda _{m-i}^{i-1}=0$, $1\leq i\leq m$.

Let $\pi ^r\colon J^r(\mathbb{R},M)\to  \mathbb{R}$,
$\pi ^{\prime r}\colon J^r(\mathbb{R},M)\to M$
be the projections $\pi ^r(j_t^r\sigma )=t$,
$\pi ^{\prime r}(j_t^r\sigma )=\sigma (t)$,
and let $\mathcal{U}^k$ be the local section
of the vector bundle $(\pi ^{\prime m})^\ast TM$
defined by,
\begin{equation*}
\mathcal{U}^k=x_k^j\frac{\partial }{\partial x^j},
\quad 1\leq k\leq m.
\end{equation*}
The metric $g$ induces a positive-definite
scalar product on $(\pi ^{\prime m})^\ast TM$
also denoted by $g$. From the definition
of the open subset of Frenet jets
$\mathcal{F}^{m-1}(M)$ in $J^{m-1}(\mathbb{R},M)$
it follows that the system
$(\mathcal{U}^1,\dotsc,\mathcal{U}^{m-1})$
is linearly independent on the open subset
$\mathcal{F}^{m-1}(M)\times _{M}TM$ in $(\pi
^{\prime m-1})^\ast TM$. By applying the Gram-Schmidt
process on such open subset to the sections
$(\mathcal{U}^1,\dotsc,\mathcal{U}^m)$, one obtains
the following systems of local sections
$(\mathcal{X}_1,\dotsc,\mathcal{X}_m)$,
$(\mathcal{Y}_1,\dotsc,\mathcal{Y}_m)$
of $(\pi ^{\prime m})^\ast TM$:
\begin{eqnarray*}
\mathcal{Y}_i
&=& \mathcal{U}^i-\sum\nolimits_{h=1}^{i-1}g
\left(
\mathcal{U}^i,\mathcal{X}_h
\right) \mathcal{X}_h,
\quad 1\leq i\leq m, \\
\mathcal{X}_i
&=& \frac{\mathcal{Y}_i}{
\left\vert \mathcal{Y}_i
\right\vert },
\quad 1\leq i\leq m-1, \\
\mathcal{X}_{m}
&=& \mathcal{X}_1
\times \cdots \times \mathcal{X}_{m-1},
\end{eqnarray*}
and according to \cite[Theorem 4.2]{Gluck1}
one has $\varkappa _i=\dfrac{
\left\vert \mathcal{Y}_{i+1}
\right\vert }{
\left\vert \mathcal{Y}_1
\right\vert
\cdot
\left\vert
\mathcal{Y}_i
\right\vert }$,
$1\leq i\leq m-1$. From the very definitions,
one deduces
$\left.
\mathcal{U}^k
\right\vert _{j^m\sigma }=U^{\sigma ,k}$
(see \eqref{U^sigma,k}) and
$\left.
\mathcal{X}_i
\right\vert _{j^m\sigma }
=X_i^\sigma $ (see Proposition \ref{referenciafrenet}).
Hence for $2\leq i\leq m$, we obtain
\begin{eqnarray*}
\dfrac{\partial \varkappa _{i-1}}{\partial x_i^j}
&=& \frac{1}{
\left\vert
\mathcal{Y}_1
\right\vert
\cdot
\left\vert \mathcal{Y}_{i-1}
\right\vert }
\frac{\partial
\left\vert
\mathcal{Y}_i
\right\vert }
{\partial x_i^j},
\\
\left\vert
\mathcal{Y}_i
\right\vert ^2
&=& g
\left(
\mathcal{U}^i,\mathcal{U}^i
\right)
-\sum\nolimits _{h=1}^{i-1}g
\left(
\mathcal{U}^i,\mathcal{X}_h
\right) ^2,
\end{eqnarray*}
and taking derivatives with respect
to $\partial x_i^j$ on the second
formula, we have
\begin{eqnarray*}
\left\vert
\mathcal{Y}_i
\right\vert
\frac{\partial
\left\vert \mathcal{Y}_i
\right\vert }
{\partial x_i^j}
&=& x_i^j-\sum\nolimits_{h=1}^{i-1}g
\left(
\mathcal{U}^i,\mathcal{X}_h
\right) \mathcal{X}_h(x^j) \\
&=& \left(
\mathcal{U}^i-\sum\nolimits_{h=1}^{i-1}g
\left(
\mathcal{U}^i,\mathcal{X}_h
\right)
\mathcal{X}_h
\right) (x^j) \\
&=&\mathcal{Y}_i(x^j).
\end{eqnarray*}
Therefore
\begin{equation*}
k_j^i=\frac{\mathcal{Y}_i(x^j)}{
\left\vert \mathcal{Y}
_1\right\vert
\cdot
\left\vert
\mathcal{Y}_{i-1}
\right\vert
\cdot
\left\vert
\mathcal{Y}_i
\right\vert }
(j_{t_0}^m\sigma ),\quad 2\leq i\leq m.
\end{equation*}
Hence
\begin{multline*}
\det (k_j^i)_{i,j=1}^m=\det
\left(
\frac{\mathcal{Y}_i(x^j)}{
\left\vert
\mathcal{Y}_1
\right\vert
\cdot
\left\vert
\mathcal{Y}_{i-1}
\right\vert
\cdot
\left\vert
\mathcal{Y}_i
\right\vert }
(j_{t_0}^m\sigma )
\right) _{i,j=1}^m \\
=\tfrac{1}{
\left(
\left\vert
\mathcal{Y}_1
\right\vert ^{m+2}
\left\vert
\mathcal{Y}_{m}
\right\vert
\prod\nolimits_{a=2}^{m-1}
\left\vert
\mathcal{Y}_a
\right\vert ^2
\right) (
j_{t_0}^m\sigma )}
\cdot
\left\vert
\begin{array}{cccc}
\mathcal{Y}_1(x^1)
& \mathcal{Y}_1(x^2)
& \ldots
& \mathcal{Y}_1(x^m) \\
\mathcal{Y}_2(x^1)
& \mathcal{Y}_2(x^2)
& \ldots
& \mathcal{Y}_2(x^2) \\
\vdots
& \vdots
& \ddots
& \vdots \\
\mathcal{Y}_m(x^1)
& \mathcal{Y}_m(x^2)
& \ldots
& \mathcal{Y}_m(x^m)
\end{array}
\right\vert
(j_{t_0}^m\sigma ),
\end{multline*}
and accordingly,
\begin{equation*}
\det (k_j^i)_{i,j=1}^m
=\dfrac{
\det _{(X_1^\sigma (t_0),\dotsc,X_m^\sigma (t_0))}
\left(
\mathcal{Y}_1,\dotsc,\mathcal{Y}_m
\right)
}{
\left\vert
\mathcal{Y}_1
\right\vert ^{m+2}
\left\vert
\mathcal{Y}_m
\right\vert
\prod\nolimits_{a=2}^{m-1}
\left\vert
\mathcal{Y}_a
\right\vert ^2}
(j_{t_0}^m\sigma ).
\end{equation*}
As the vectors $\mathcal{Y}_1,\dotsc,\mathcal{Y}_m$
 are pairwise orthogonal its determinant is
 $\prod\nolimits_{a=1}^m
 \left\vert
 \mathcal{Y}_a
 \right\vert $. Hence
\begin{eqnarray*}
\det (k_j^i)_{i,j=1}^m
&=& \dfrac{1}{
\left\vert
\mathcal{Y}_1
\right\vert ^m
\prod\nolimits_{a=1}^{m-1}
\left\vert
\mathcal{Y}_a
\right\vert }
(j_{t_0}^m\sigma ) \\
&=& \frac{1}{\kappa _0^\sigma (t_0)^{\frac{m(m+1)}{2}}
\kappa _1^\sigma (t_0)^{m-1}\kappa _2^\sigma (t_0)^{m-2}
\cdots \kappa _{m-2}^\sigma (t_0)},
\end{eqnarray*}
which is positive. Hence $\lambda ^{i-1}_{m-i}=0$.
\end{proof}
\subsection{Few isometries}\label{FewIsom}
Below we consider Riemannian manifolds $(M,g)$ such that,
$\dim \mathfrak{i}(M,g)\! \leq \! \dim M$.
\begin{proposition}
\label{Few_isomtries}
If $(M,g)$ is a connected complete Riemannian manifold
of class $C^{\omega }$ and $\mathfrak{i}(M,g)$ admits
a basis $(X_1,\dotsc,X_l)$ such that the tangent vectors
$((X_1)_{x_0},\dotsc,(X_l)_{x_0})$ at a point $x_0\in M$
are linearly independent, then $N_r=mr+1+m-l$,
$\forall r\geq 0$. The order of asymptotic stability
of such manifolds is $r=1$; namely, for every $r\geq 2$
a basis of differential invariants of order $r$ can be
obtained by applying the operator $D_t$ to a basis
of first-order invariants successively.
\end{proposition}
\begin{proof}
According to the hypothesis of the statement,
the set of points $x\in M$ for which $((X_1)_x,\dotsc,(X_l)_x)$
are linearly independent, is a dense open subset $O\subseteq M$,
and for every $j_t^r\sigma \in J^r(\mathbb{R},O)$
the homomorphism $(\pi _0^r)_\ast
\colon \mathfrak{D}_{j_t^r\sigma }^r
\to \mathfrak{D}_{\sigma (t)}^0$
is an isomorphism. The formula for $N_r$ now follows from
Lemma \ref{lemma_rank}.

If $(t,I_i)_{i=l+1}^m$ is a basis of differential invariants
of order $0$, as $N_1=$ $2m+1-l$, then there exist $l$ new
invariants of (strict) order $1$, say $(I_1,\dotsc,I_l)$,
such that $(t,I_{l+1},\dotsc,I_m,D_t(I_{l+1}),\dotsc,
D_t(I_m),I_1,\dotsc,I_l)$ is a basis of differential
invariants of order $1$. Actually, if
\begin{equation*}
0=\sum\nolimits_{i=1}^{m-l}\lambda ^id
\left( D_t(I_{l+i})
\right)
=\sum\nolimits_{i=1}^{m-l}\lambda ^i
\left\{
x_1^k
\frac{\partial ^2I_{l+i}}
{\partial x^j\partial x^k}dx^j
+\frac{\partial I_{l+i}}
{\partial x^j}dx_1^j
\right\} ,
\end{equation*}
then $\lambda ^i$ must vanish, as
$\operatorname{rk}(\partial I_{l+i}/\partial x^j)=m-l$.
We can repeat this process indefinitely as in Theorem
\ref{Generating_Invariants}, taking account
of the fact that the number of differential invariants
of strict order $r$ in a basis of invariants is
$N_{r+1}-N_r=m$.
\end{proof}
\begin{remark}
\label{remark_1}
If $\dim \mathfrak{i}(M,g)=1$ or $2$,
then $\mathfrak{i}(M,g)$ always admits a basis
of linearly independent vector fields at some point
$x_0\in M$.
In fact, if $X$, $\rho X$, $\rho \in C^\infty (M)$,
are Killing vector fields for $g$, then
$0=L_{\rho X}g=d\rho \cdot c(X\otimes g)$,
where the dot denotes symmetric product and
$c\colon TM\otimes S^2T^\ast M\to  T^\ast M$
is the contraction operator. As the symmetric algebra
of every cotangent space is an integral domain,
it follows $c(X\otimes g)=0$, but in this case
$\rho X$ would be a Killing vector field for every
$\rho $, which contradicts the fact that
$\mathfrak{i}(M,g)$ is a finite-dimensional Lie algebra
(e.g., see \cite[VI, Theorem 3.3]{KN}).
\end{remark}
\begin{proposition}
[{cf.\ \protect\cite[(1.3)]{Griffiths},
\protect\cite[I.11,Theorem 1]{Jensen}}]
\label{left_invariant_metric}
Let $(G,g)$ be a connected Lie group with Lie
algebra $\mathfrak{g}$ endowed with a left-invariant
Riemannian metric. Two smooth curves
$\sigma _1,\sigma _2\colon \mathbb{R}\to G$ are
congruent with respect to $G$ if and only if
\begin{equation}
\label{gg}
\omega (T^{\sigma _1}(t))
=\omega (T^{\sigma _2}(t)),
\quad \forall t\in \mathbb{R},
\end{equation}
where $\omega \in \Omega ^1(G,\mathfrak{g})$
denotes the Maurer-Cartan form.
\end{proposition}
\begin{proof}
If there exists an element $\gamma \in G$ such that
$\sigma _1=L_\gamma \circ \sigma _2$, then
\begin{eqnarray*}
\omega (T^{\sigma _1}(t))
&=& \omega ((L_\gamma )_\ast T^{\sigma _2}(t)) \\
&=& \left(
(L_{\gamma })^\ast \omega
\right)
(T^{\sigma _2}(t)) \\
&=& \omega (T^{\sigma _2}(t)).
\end{eqnarray*}
Conversely, assume the equation \eqref{gg} holds
and let $\gamma \in G$ be the only element such that
$\sigma _1(0)=\gamma \cdot \sigma _2(0)$. On the product
manifold $\mathbb{R}\times G$, let $Z$ be the vector field
defined as follows:
$Z_{(t,\alpha )}
=(\partial /\partial t,\omega (T^{\sigma _i}(t))_\alpha )$,
$\forall (t,\alpha )\in \mathbb{R}\times G$, $i=1,2$,
where $\omega (T^{\sigma _i}(t))_\alpha $ denotes
the evaluation of the left invariant vector field
$\omega (T^{\sigma _i}(t))\in \mathfrak{g}$ at
$\alpha \in G$ for $i=1,2$; i.e.,
\begin{equation}
\label{jj}
\omega (T^{\sigma _i}(t))_\alpha
=\left.
\frac{d}{du}
\right\vert _{u=0}
L_{\alpha }\circ \exp
\left(
u\omega (T^{\sigma _i}(t))
\right) .
\end{equation}
Hence $\tilde{\sigma }_i(t)=(t,\sigma _i(t))$
is an integral curve of $Z$, $i=1,2$. Moreover, the curve
$\tilde{\sigma}(t)=(t,(L_\gamma \circ \sigma _2)(t))$
is also an integral curve of $Z$. In fact, from \eqref{jj}
we obtain
\begin{eqnarray*}
T^{\tilde{\sigma}}(t)
&=&
\left(
\frac{\partial }{\partial t},(L_\gamma )_\ast
T^{\sigma _2}(t)
\right) \\
&=&
\left(
\frac{\partial }{\partial t},
\left. \frac{d}{du}\right\vert
_{u=0}L_{\gamma \cdot \gamma _2(t)}\circ \exp
\left(
u\omega ((L_\gamma )_\ast T^{\sigma _2}(t)
\right)
\right) \\
&=&
\left(
\frac{\partial }{\partial t},
\left.
\frac{d}{du}
\right\vert _{u=0}
L_{\gamma \cdot \gamma _2(t)}\circ \exp
\left(
u\omega (T^{\sigma _2}(t)
\right)
\right) \\
&=&\left(
\frac{\partial }{\partial t},
\left.
\frac{d}{du}
\right\vert
_{u=0}L_{\gamma \cdot \gamma _2(t)}\circ \exp
\left(
u\omega (T^{\sigma _1}(t)
\right)
\right) \\
&=&
\left(
\frac{\partial }{\partial t},
\omega (T^{\sigma _1}(t))_{\gamma \cdot \gamma _2(t)}
\right) .
\end{eqnarray*}
As $\tilde{\sigma }_1(0)
=(0,\sigma _1(0))
=(0,\gamma \cdot \sigma _2(0))
=\tilde{\sigma}(0)$, from the uniqueness of a solution
to a system of ordinary differential equations we conclude
$\tilde{\sigma }_1(t)=\tilde{\sigma }(t)$, $\forall t$,
i.e., $\sigma _1=L_\gamma \circ \sigma _2$.
\end{proof}
\begin{corollary}
\label{corollary_left_invariant}
Let $(G,g)$ be a connected Lie group with Lie algebra
$\mathfrak{g}$ endowed with a left-invariant Riemannian
metric such that $\dim \mathfrak{I}(G,g)=\dim G$.
Two smooth curves $\sigma _1,\sigma _2\colon \mathbb{R}\to G$
are congruent with respect to $\mathfrak{I}^0(G,g)$
if and only if \emph{\eqref{gg}} holds. Moreover,
a basis of invariants of order $\leq r$ is
$\{ t,(D_t)^k(I)\} _{k=0}^{r-1}$, where $I=(I_1,\dotsc,I_m)
\colon J^1(\mathbb{R},M)\to \mathfrak{g}$ are the functions
$I(j_t^1\sigma )=\omega (T^\sigma (t))$.
\end{corollary}
\begin{proposition}
\label{proposition_independent_vectors}
If $(M,g)$ is a complete Riemannian
connected manifold of class $C^\omega $ such that
$\mathfrak{i}(M,g)$ admits a basis
$(X_1,\dotsc,X_m)$, $m=\dim M$, with linearly
independent tangent vectors
$((X_1)_{x_0},\dotsc,(X_m)_{x_0})$ at $x_0\in M$,
then $(M,g)$ is isometric to $(\mathfrak{I}^0(M,g),\bar{g})/H$,
where $H$ is a finite subgroup in $G=\mathfrak{I}^0(M,g)$,
and $\bar{g}$ is a left-invariant Riemannian metric.
\end{proposition}
\begin{proof}
The mapping $\mu \colon G\to  M$, $\mu (\phi )=\phi (x_0)$,
is $G$-equivariant, i.e.,
$\gamma \circ \mu =\mu \circ L_\gamma $,
where $\gamma \in G$ and $L_\gamma $ denotes the left
translation by $\gamma $. Hence $\mu $ is of constant rank
and $\operatorname{im}\mu $ is an open subset in $M$ by virtue
of the assumption, which we claim is the whole manifold $M$.
In fact, if $x\in \partial (\operatorname{im}\mu )$, then
there exists a sequence $x_n=\gamma _n(x_0)$ such that
$\lim x_n=x$; as the action of $G$ is proper (see
\cite[5.2.4]{PalaisTerng}), we conclude that there is
a convergent subsequence $\gamma _{n_k}\to \gamma $.
Hence $x=\gamma (x_0)$ and the image of $\mu $ is a closed
subset. Therefore the isotropy subgroup $H$ of the point $x_0$
is a finite subgroup (\cite[I, Corollary4.8]{KN}) in $G$
and $G/H\cong M$, where $H$ acts on the right on $G$.
The metric $\bar{g}=\mu ^\ast (g)$ is invariant under left
translations of elements in $G$; in fact,
\begin{equation*}
L_\gamma ^\ast (\bar{g})
=\left(
L_\gamma ^\ast \circ \mu ^\ast
\right)
(g)=(\gamma \circ \mu )^\ast (g)
=\mu ^\ast (\gamma ^\ast (g))
=\mu ^\ast (g)=\bar{g}.
\end{equation*}
\end{proof}
\begin{proposition}
\label{proposition_omega_tilde_H}
Let $\omega \colon TG\to  \mathfrak{g}$ be the Maurer-Cartan
form of a connected Lie group $G$ and let
$\omega _h\colon TG\to \mathfrak{g}/H$ be map given
by $\omega _h(X)=\omega (X)\operatorname{mod}H$,
$\forall X\in TG$, where $H$ is a finite subgroup acting
on the right on $\mathfrak{g}$ by setting,
$X\cdot h=\operatorname{Ad}_{h^{-1}}X$, $\forall h\in H$,
$\forall X\in \mathfrak{g}$, and $\operatorname{Ad}$
denoting the adjoint representation of $G$. As
$\omega _h$ is $H$-invariant, it induces a ``$1$-form''
$\tilde{\omega }_h$ on $M\cong G/H$\ with values
in the quotient $\mathfrak{g}/H$ (which is not a vector space).
With the same hypotheses
of \emph{Proposition \ref{proposition_independent_vectors}}
about $(M,g)$, two curves
$\sigma _1,\sigma _2\colon (a,b)\to M$
are congruent under $G$ if and only if,
$\tilde{\omega }_h(T^{\sigma _1}(t))
=\tilde{\omega }_h(T^{\sigma _2}(t))$, $\forall t\in (a,b)$.
\end{proposition}
\begin{proof}
The curve $\sigma _i$ can be lifted to
$\gamma _i\colon (a,b)\to G$, i.e.,
$\sigma _i(t)=\gamma _i(t)\operatorname{mod}H$,
for $i=1,2$, as $G\to  M\cong G/H$ is a covering.
The condition $\gamma (\sigma _1(t))=\sigma _2(t)$,
$\gamma \in G$, is equivalent to
$\gamma \cdot \gamma _1(t)=\gamma _2(t)\cdot h$
for some $h\in H$. From Proposition \ref{left_invariant_metric}
this last condition holds if and only if,
$\omega (T^{\gamma _1}(t))=\omega ((R_h)_\ast T^{\gamma _2}(t))$,
$\forall t$. As $\omega $ is left invariant, the previous equation
is equivalent to the following:
\begin{eqnarray*}
\omega (T^{\gamma _1}(t))
&=&
\omega ((L_{h^{-1}})_\ast (R_h)_\ast T^{\gamma _2}(t)) \\
&=&
\operatorname{Ad}\nolimits_{h^{-1}}(\omega (T^{\gamma _2}(t))),
\end{eqnarray*}
or equivalently, $\omega _h(T^{\gamma _1}(t))
=\omega _h(T^{\gamma _2}(t))$, thus concluding the proof.
\end{proof}
\begin{remark}
\label{remark_3}
As $H$ is finite, the adjoint representation of $H$
on $\mathfrak{g}$ is completely reducible and
a classical result (see \cite{Luna}) assures that the algebra
of smooth invariants $C^\infty (\mathfrak{g)}^H$ admits
a finite basis, i.e.,
$C^\infty (\mathfrak{g)}^H=C^\infty [I_1,\dotsc,I_k]$,
$I_1,\dotsc,I_k$ being $H$-invariant polynomial functions
on $\mathfrak{g}$. Two curves $\sigma _1$,
$\sigma _2$ are congruent if and only if,
$I_i(\tilde{\omega }_h(T^{\sigma _1}(t)))
=I_i(\tilde{\omega }_h(T^{\sigma _2}(t)))$
for $1\leq i\leq k$; i.e., $\tilde{\omega }_h$
can be replaced by the ordinary scalar invariants
$I_1\circ \tilde{\omega }_h,
\dotsc,
I_k\circ \tilde{\omega }_h$.
\end{remark}
For $\dim \mathfrak{I}^0(M,g)<\dim M$,
as a consequence of the results in this section,
we can state the following:
\begin{proposition}
\label{Geometric Meaning}
Assume $(M,g)$ is a complete Riemannian connected
manifold of class $C^\omega $ that satisfies
the following two conditions:
\emph{i)} $\dim \mathfrak{I}^0(M,g)=l<m$,
and \emph{ii)} there exists a basis
$(X_1,\dotsc,X_{l})$ for $\mathfrak{i}(M,g)$
such that the tangent vectors $((X_1)_{x_0},
\dotsc,(X_{l})_{x_0})$ are linearly
independent at $x_0\in M$. We set
$G=\mathfrak{I}^0(M,g)$.

Let $S\subset M$ be a slice at $x_0$, let $U$
be a $G$-invariant neighbourhood of $x_0$ and
let $r\colon U\to  G\cdot x_0$ be
a $G$-equivariant retraction of the inclusion
$G\cdot x_0\subset U$ such that,
\emph{1)} $r^{-1}(x_0)=S$
\emph{(cf.\ \cite[I, 5.1.11]{PalaisTerng})},
\emph{2)} the assumption \emph{ii)} above holds
for every $x\in S$. Then, $G_x=G_{x_0}$
for every $x\in S$, and we can define
$\hat{\omega }_h\colon TU\to \mathfrak{g}/H$,
with $H=G_{x_0}$, by setting,
$\hat{\omega }_h(X)=\tilde{\omega }_h(\pi _xX)$,
where
$\tilde{\omega }_h\colon T(G/H)\to \mathfrak{g}/H$
is defined in
\emph{Proposition \ref{proposition_omega_tilde_H}},
$\pi _x\colon T_{x}U\to T_x(G\cdot x)$ denotes
the orthogonal $g_{x}$-projection and
the canonical isomorphism $G\cdot x\cong G/H$
is used.

Let $q\colon U\to U/G\cong S/H$ be the quotient mapping.
If $(\bar{x}^1,\dotsc,\bar{x}^{m-l})$ is a coordinate
system on $S/H$, then $I_i=\bar{x}^i\circ q$ is a system
of $m-l$ differential invariants of order zero on $U$.
If $\sigma _1,\sigma _2\colon (a,b)\to  M$ are two
curves such that, $\sigma _1(t_0)=\sigma _2(t_0)=x_0$,
and $\phi \circ \sigma _1=\sigma _2$, for some $\phi \in G$,
then

\smallskip

\emph{a)} $\hat{\omega }_h(T^{\sigma _1}(t_0))
=\hat{\omega }_h(T^{\sigma _2}(t_0))$.

\smallskip

\emph{b)} The invariants $I_1,\dotsc,I_{m-l},I$, where
$I\colon J^1(R,U)\to \mathfrak{g}/H$, defined
by $I(j_t^1\sigma )=\hat{\omega }_h(T^\sigma (t))$,
are functionally independent.

\smallskip

\emph{c)} The invariants $t,I_i,D_tI_i$,
$1\leq i\leq m-l$, and $I$ are a basis of first-order
differential invariants and they generate the ring
of invariant functions of arbitrary order $r\geq 2$.
\end{proposition}
\subsection{Surfaces}
\label{surf}
\begin{theorem}
\label{surfaces}
Let $(M,g)$ be a Riemannian bidimensional manifold
of class $C^\omega $, let
$\mathcal{O}^1\subset J^1(\mathbb{R},M)$ be the open
subset of immersive jets, which coincides with Frenet
curves in this case (see \emph{Remark \ref{remark_2}}),
let $(X_1^\sigma ,X_2^\sigma )$ be the Frenet frame
of an immersion $\sigma $, let $K^g$ be the Gaussian
curvature of $g$, and let $\nabla ^g$ be the Levi-Civita
connection of $g$. The functions
$I_i^g\colon \mathcal{O}^1\to  \mathbb{R}$,
$1\leq i\leq 4$, defined by
$I_1^g
\left( j_{t_0}^1\sigma
\right)
=g\left(
T_{t_0}^\sigma ,T_{t_0}^\sigma
\right) $,
$I_2^g
\left(
j_{t_0}^1\sigma
\right)
=\left(
dK^g
\right)
(X_1^\sigma )$,
$I_3^g
\left(
j_{t_0}^1\sigma
\right)
=\left(
dK^g\right)
\left(
X_2^\sigma
\right) $, and $I_4^g
\left(
j_{t_0}^1\sigma
\right)
=\left(
\nabla ^gdK^g
\right)
(X_1^\sigma ,X_1^\sigma )$ are invariant.
If $\mathcal{M}\to  M$ denotes the bundle
of Riemannian metrics on $M$ of class
$C^\omega $, then there exists a dense open
subset $U^5\subset J^5(\mathcal{M})$
such that for every metric $g$ for which
$j^5g$ takes values in $U^5$ the invariants
$(I_1^g,\dotsc,I_4^g) $ are functionally
independent and two Frenet curves $\sigma $,
$\bar{\sigma} $ with values in $(M,g)$
are congruent on some neighbourhoods
of $x_0=\sigma (t_0)$ and $
\bar{x}_0=\bar{\sigma}(t_0)$ if and only
if, the following conditions hold:
\begin{equation*}
\kappa _0^\sigma (t)
=\kappa _0^{\bar{\sigma}}(t),\quad \kappa
_1^\sigma (t)=\kappa _1^{\bar{\sigma}}(t),
\end{equation*}
for small enough $|t-t_0|$, and
\begin{eqnarray*}
\left(
dK^g
\right)
(X_1^\sigma )(x_0)
&=& \left(
dK^g
\right)
(X_1^{\bar{\sigma}})(\bar{x}_0), \\
\left(
dK^g
\right) (
X_2^\sigma )(x_0)
&=& \left(
dK^g
\right)
(X_2^{\bar{\sigma}})(\bar{x}_0), \\
\left(
\nabla ^gdK^g
\right)
(X_1^\sigma ,X_1^\sigma )
&=& \left(
\nabla ^gdK^g
\right)
(X_1^{\bar{\sigma}},X_1^\sigma ).
\end{eqnarray*}
\end{theorem}
\begin{proof}
Let
$\tilde{I}_i\colon \mathcal{O}^1\to \mathbb{R}$,
$0\leq i\leq 4$, be the invariant functions
defined as follows:
\begin{eqnarray*}
\tilde{I}_0(j_{t_0}^1\sigma )
&=& t_0, \\
\tilde{I}_1^g(j_{t_0}^1\sigma )
&=& \kappa _0^\sigma (t_0), \\
\tilde{I}_2^g(j_{t_0}^1\sigma )
&=& \nabla ^gR_4^g(X_1^\sigma ,
X_1^\sigma ,
X_2^\sigma ,
X_1^\sigma ,
X_2^\sigma )(\sigma (t_0)), \\
\tilde{I}_3^g(j_{t_0}^1\sigma )
&=&\nabla ^gR_4^g(X_2^\sigma ,
X_1^\sigma ,
X_2^\sigma ,
X_1^\sigma ,
X_2^\sigma )(\sigma (t_0)), \\
\tilde{I}_4^g(j_{t_0}^1\sigma )
&=&
\left(
\nabla ^g
\right) ^2
\!R_4^g(X_1^\sigma ,
\! X_1^\sigma ,
\! X_1^\sigma ,
\! X_2^\sigma ,
\! X_1^\sigma ,
\! X_2^\sigma )
(\sigma (t_0)),
\end{eqnarray*}
$R_4^g$ being the Riemann-Christoffel
tensor of $g$. In local coordinates
$x=x^1$, $y=x^2$, $\dot{x}=x_1^1$,
$\dot{y}=x_1^2$, we have
\begin{eqnarray}
\tilde{I}_1^g
&=& (I_1^g)^{\frac{1}{2}}
=\left(
g_{11}\dot{x}^2
+2g_{12}\dot{x}\dot{y}
+g_{22}\dot{y}^2
\right) ^{\frac{1}{2}},
\label{I1} \\
\tilde{I}_2^g
&=& I_2^g
=\left(
I_1^g
\right) ^{-1}
(\dot{x}K_x^g+\dot{y}K_y^g),
 \label{I2} \\
\tilde{I}_3^g
&=& I_3^g
=\left(
I_1^g
\right) ^{-1}
\det (g_{ij})^{-\frac{1}{2}}
\left(
\left(
g_{11}\dot{x}+g_{12}\dot{y}
\right)
K_y^g
-\left(
g_{12}\dot{x}+g_{22}\dot{y}
\right)
K_x^g
\right) ,
\label{I3} \\
\tilde{I}_4^g
&=&I_4^g
=\left(
I_1^g
\right) ^{-2}
\left(
\dot{x}^2K_{xx}^g
+2\dot{x}\dot{y}K_{xy}^g
+\dot{y}^{2}K_{yy}^g
\right) ,
\label{I4}
\end{eqnarray}
where $g=\sum _{i,j=1}^2g_{ij}dx^i\otimes dx^j$,
$g_{12}=g_{21}$, which proves that
$(I_1^g,\dotsc,I_4^g)$ are invariant functions
according to Theorem \ref{CGC}, and also that
$(\tilde{I}_0,\tilde{I}_1^g,\dotsc,\tilde{I}_4^g)$
are functionally independent if and only if
$(I_1^g,\bar{I}_2^g
=\tilde{I}_1^gI_2^g,\bar{I}_3^g
=\tilde{I}_1^gI_3^g,\bar{I}_4^g
=(\tilde{I}_1^g)^2I_4^g)$ are. Moreover,
as a simple---although rather long---computation
shows one has
\begin{equation*}
\frac{\partial
\left(
I_1^g,\bar{I}_2^g,\bar{I}_3^g,\bar{I}_4^g
\right)
}{\partial
\left( x,y,\dot{x},\dot{y}
\right) }
\left(
j_{t_0}^1\sigma
\right)
=g\left(
T_{t_0}^\sigma ,T_{t_0}^\sigma
\right)
\det (g_{ij}(x_0))^{-\frac{3}{2}}\Phi
\left(
j_{t_0}^1\sigma ,j_{x_0}^5g
\right) ,
\end{equation*}
where $\Phi $ is a polynomial function
of the form
\begin{equation*}
\Phi
\left(
j_{t_0}^1\sigma ,j_{x_0}^{5}g
\right)
=\Phi _{11}(j_{x_0}^5g)
\dot{x}(j_{t_0}^1\sigma )^2
+\Phi _{12}(j_{x_0}^{5}g)
\dot{x}(j_{t_0}^1\sigma )
\dot{y}(j_{t_0}^1\sigma )
+\Phi _{22}(j_{x_0}^{5}g)
\dot{y}(j_{t_0}^1\sigma )^2,
\end{equation*}
and $U^5$ is given by
$(\Phi _{11})^2+(\Phi _{12})^2+(\Phi _{22})^2>0$.

As $\dim J^1(\mathbb{R},M)=5$, we can conclude
by simply applying Corollary \ref{CorolCGC}
because the rest of invariants depends functionally
on $(\tilde{I}_0,\tilde{I}_1^g,\dotsc,\tilde{I}_4^g)$.
\end{proof}
\begin{remark}
From the previous theorem it follows that the order
of asymptotic stability\ of generic Riemannian surface
is $\leq 2$, which agrees with Theorem \ref{stability},
but the order of asymptotic stability is $\leq 1$,
indeed. In fact, we can solve the equations
\eqref{I1}--\eqref{I4} for $x$, $y$, $\dot{x}$,
$\dot{y}$ and taking derivatives with respect to
$D_t$ we conclude that $\kappa _1^\sigma $ can be written
as a function of the invariants $t,I_1^g,\dotsc,I_4^g$
and their total derivatives.
\end{remark}
\subsection{$3$-dimensional manifolds}\label{three_dim}
In this section, $(M,g)$ denotes a $3$-dimensional
simply connected complete Riemannian manifold
of non-constant curvature. Because of the dimensional
gap \cite[II, Theorem 3.2]{Kobayashi}
(also see \cite{Ku}), $\dim \mathfrak{I}(M,g)\leq 4$.
If $\dim \mathfrak{I}(M,g)\leq 2 $, then not
much more can be said about the structure of invariants
of $(M,g)$ in addition of what has been said in general
in the section \ref{FewIsom}; on the other hand,
if $\dim \mathfrak{I}(M,g)=3$ or $4$ and
$\mathfrak{I}(M,g)$ acts transitively on $M$,
then the structure of invariants can fully be determined,
as explained below.
\begin{theorem}
\label{4_isometries}
Let $(M,g)$ be a $3$-dimensional simply connected
complete Riemannian manifold such that,
i) $\dim \mathfrak{I}(M,g)=4$ and ii) $\mathfrak{I}(M,g)$
acts transitively on $M$. There exists a global unitary
vector field $Z$ on $M$ such that, $\phi \cdot Z=Z$
for every $\phi \in \mathfrak{I}^0(M,g)$ and
the function $I_1\colon J^1(\mathbb{R},M)\to \mathbb{R}$
given by
$I_1(j_t^1\sigma )=g(Z_{\sigma (t)},T^\sigma (t))$,
is a first-order differential  invariant.
A complete system of differential invariants is given
as follows: $t$, $\varkappa _0$, $I_1$, and $\varkappa _1$.
\end{theorem}
\begin{proof}
As is known (e.g., see \cite{Cartan}, \cite{Daniel},
\cite{Engel}), any simply connected complete homogeneous
$3$-manifold with $4$-dimensional isometry group,
fibres as an $1$-dimensional principal fibre bundle
over a homogeneous $2$-manifold of constant sectional
curvature $\kappa $. The orthogonal distribution
to the fibres is the kernel of a connection form
the curvature of which has constant norm $\tau $.
Such a principal bundle is usually denoted by
$(M_{\kappa ,\tau },g_{\kappa ,\tau })$. In a certain
fibred coordinate system $(x,y,z)$ the metric reads
as follows:
\begin{eqnarray}
g_{\kappa ,\tau }
&=& \frac{1}{
\left( 1+\frac{\kappa }{4}
\left(
x^2+y^2
\right)
\right) ^2}
\left(
dx^2+dy^2
\right)
\label{mmme} \\
&& +\left(
\frac{\tau }{1+\frac{\kappa }{4}
\left( x^2+y^2
\right) }
\left(
ydx-xdy
\right)
+dz
\right) ^2.
\notag
\end{eqnarray}
The lines parallel to the $z$ axis are the fibers
of $M_{\kappa ,\tau }$. Moreover, the tangent vector
field $Z=\partial /\partial z$ is unitary
and more exactly is the unitary infinitesimal generator
defined by the right action of the structure group
on $M_{\kappa ,\tau }$ as a principal bundle.

We claim that every
$\phi \in \mathfrak{I}^0(M_{\kappa ,\tau },
g_{\kappa ,\tau })$ maps fibres to fibres.
Working on the expression \eqref{mmme}
of the metric $g_{\kappa ,\tau }$, one can check
that the following vector fields are a basis
for $\mathfrak{i}(M_{\kappa ,\tau },g_{\kappa ,\tau })$:
\begin{eqnarray*}
X_1 &=& -2\kappa xy
\frac{\partial }{\partial x}
+(\kappa x^2-\kappa y^2-4)
\frac{\partial }{\partial y}
+4\tau x
\frac{\partial }{\partial z}, \\
X_2 &=& 2\kappa xy
\frac{\partial }{\partial y}
+(\kappa x^2-\kappa y^2+4)
\frac{\partial }{\partial x}
+4\tau y
\frac{\partial }{\partial z}, \\
X_3 &=& -y
\frac{\partial }{\partial x}
+x\frac{\partial }{\partial y}, \\
Z &=& \frac{\partial }{\partial z},
\end{eqnarray*}
As $[Z,X]=0$,
$\forall X\in \mathfrak{i}
(M_{\kappa ,\tau },g_{\kappa ,\tau }) $,
it follows $\phi \cdot Z=Z$,
$\forall \phi \in \mathfrak{I}^0
(M_{\kappa ,\tau },g_{\kappa ,\tau })$.
If
$\sigma _i\colon (a,b)\to  M_{\kappa ,\tau }$,
$i=1,2$, are two curves such that,
$\sigma _1=\phi \circ \sigma _2$ for certain
$\phi \in \mathfrak{I}^0(M_{\kappa ,\tau },
g_{\kappa ,\tau })$, then
$g_{\kappa ,\tau }(Z_{\sigma _1(t)},T^{\sigma _1}(t))
=g_{\kappa ,\tau }(Z_{\sigma _2(t)},T^{\sigma _2}(t))$,
$\forall t\in (a,b)$. Hence the function
$I_1$ in the statement is an invariant.

Moreover, from the formula for $N_r$
in Lemma \ref{lemma_rank} (resp.\ Corollary
\ref{corollary_stability}) in this case one has
$N_r=3r+4-\operatorname{rk}
\left.
\mathfrak{D}^r
\right\vert _{\mathcal{U}^r}$
(resp.\  $N_r=3r$, $\forall r\geq 4)$.
As $\mathfrak{I}(M_{\kappa ,\tau },
g_{\kappa ,\tau })$ acts transitively
on $M_{\kappa ,\tau }$ one has,
$\operatorname{rk}\mathfrak{D}_{(t,x)}^0=3$,
for all
$(t,x)\in \mathbb{R}\times M_{\kappa ,\tau }$;
hence $N_0=1$, and the problem is reduced to compute
$\operatorname{rk}\mathfrak{D}^r$ for $r=1,2,3$.
From the formula \eqref{F^r},
one obtains,
\begin{eqnarray*}
X_1^{(1)} &=& -2\kappa xy
\frac{\partial }{\partial x}
+(\kappa x^2-\kappa y^2-4)
\frac{\partial }{\partial y}
+4\tau x
\frac{\partial }{\partial z} \\
&& -2\kappa
\left(
x_1y+xy_1
\right)
\frac{\partial }{\partial x_1}
+2\kappa (xx_1-yy_1)
\frac{\partial }{\partial y_1}
+4\tau x_1
\frac{\partial }{\partial z_1}, \\
X_2^{(1)} &=&2 \kappa xy
\frac{\partial }{\partial y}
+(\kappa x^2-\kappa y^2+4)
\frac{\partial }{\partial x}
+4\tau y
\frac{\partial }{\partial z} \\
&&+2\kappa
\left(
x_1y+xy_1
\right)
\frac{\partial }{\partial y_1}
+2\kappa (xx_1-yy_1)
\frac{\partial }{\partial x_1}
+4\tau y_1
\frac{\partial }{\partial z_1}, \\
X_3^{(1)} &=& -y\frac{\partial }{\partial x}
+x\frac{\partial }{\partial y}
-y_1\frac{\partial }{\partial x_1}
+x_1\frac{\partial }{\partial y_1},
\\
Z^{(1)} &=&\frac{\partial }{\partial z}.
\end{eqnarray*}
As a computation shows, one has
$\operatorname{rk}
\left. \mathfrak{D}^1
\right\vert _{\mathcal{U}^1}=4$ with
$\mathcal{U}^1=\{j_t^1\sigma
:T^\sigma (t)\neq 0\}$.
As $\mathfrak{D}^r$ is spanned by four
vector fields, we conclude that
$\operatorname{rk}\mathfrak{D}^r\leq 4$,
and since $\operatorname{rk}\mathfrak{D}^r
=\operatorname{rk}\mathfrak{D}^{r-1}
+\operatorname{rk}\mathfrak{D}^{r,r-1}$
by virtue of \eqref{ExactSeq}, we finally
obtain $\operatorname{rk}
\left.
\mathfrak{D}^r
\right\vert _{\mathcal{U}^r}=4$,
$\forall r\geq 1$. Hence $N_r=3r$,
$\forall r\geq 1$.

For $r=1$, the invariants are $t$,
$\tilde{\varkappa }_0$, and $I_1$,
as they are functionally independent
because of their expressions below,
\begin{eqnarray*}
\tilde{\varkappa }_0
&=& \varkappa _0
-\left(
I_1
\right) ^2
=\frac{(x_1)^2+(y_1)^2}{
\left( 1+\frac{\kappa }{4}
\left( x^2+y^2
\right)
\right) ^2}, \\
I_1 &=& \frac{\tau }{1+\frac{\kappa }{4}
\left(
x^2+y^2
\right) }
\left(
yx_1-xy_1
\right)
+z_1.
\end{eqnarray*}
Moreover, as a calculation shows,
the curvature function
$\varkappa _1\colon \mathcal{F}^2(M)
\to \mathbb{R}$ is given by,
\begin{eqnarray*}
\left(
\varkappa _0
\right) ^6
\left(
\varkappa _1
\right) ^2
&=& (\varkappa _0)^2
\left(
\left( g_{\kappa ,\tau }
\right) _{11}
(x_2)^2
+\left(
g_{\kappa ,\tau }
\right) _{22}
(y_2)^2
+\left(
g_{\kappa ,\tau }
\right) _{33}
(z_2)^2
\right) \\
&& +2(\varkappa _0)^2
\left(
\left(
g_{\kappa ,\tau }
\right) _{12}
x_2y_2
+\left(
g_{\kappa ,\tau }
\right) _{13}
x_2z_2
+\left(
g_{\kappa ,\tau }
\right) _{23}
y_2z_2
\right) \\
&& +Ax_2+By_2+Cz_2
+\text{terms of order }\leq 1,
\end{eqnarray*}
$A,B,C$ being homogeneous polynomials
of degree $6$ in the following functions:

1) $\varkappa _0,D_t\varkappa _0,x_1,y_1,z_1$,

2) the local coefficients
$(g_{\kappa ,\tau })_{1\leq i\leq j\leq 3}$ of
$g_{\kappa ,\tau }$, and

3) the Christoffel symbols $\Gamma _{ij}^h$
of the Levi-Civita connection of $g_{\kappa ,\tau }$.

Following the proof of Theorem \ref{Generating_Invariants}
and by using the formula \eqref{formula1}, in the present
case one has, $k_0=N_0-1=0$, $k_1=N_1-1=2$,
$k_2=N_2-1-2k_1=1$. Hence there are two functionally
independent invariants of order $1$, namely
$\tilde{\varkappa}_0$, $I_1$, and one invariant
of order $2$, which is $\varkappa _1$. We claim
the invariants $t$, $\tilde{\varkappa}_0$, $I_1$,
$D_t\tilde{\varkappa }_0 $, $D_tI_1$, $\varkappa _1$
are generically functionally independent. In fact,
\begin{eqnarray*}
\left.
d\tilde{\varkappa }_0
\right\vert _{\ker (\pi _0^1)_\ast }
&=& \frac{2
\left( x_1\delta x_1+y_1\delta y_1
\right)
}
{\left(
1+\frac{\kappa }{4}
\left(
x^2+y^2
\right)
\right) ^2}, \\
\left.
dI_1
\right\vert _{\ker (\pi _0^1)_\ast }
&=& \frac{\tau
\left(
y\delta x_1-x\delta y_1
\right) }
{1+\frac{\kappa }{4}
\left(
x^2+y^2
\right) }
+dz_1,
\end{eqnarray*}
\begin{eqnarray*}
\left.
d(D_t\tilde{\varkappa}_0)
\right\vert _{\ker (\pi _1^2)_\ast }
&=& \frac{2
\left(
x_1\delta x_2+y_1\delta y_2
\right) }
{\left(
1+\frac{\kappa }{4}
\left(
x^2+y^2
\right)
\right) ^2}, \\
\left.
\left(
\varkappa _0
\right) ^4\varkappa _1d\varkappa _1
\right\vert _{\ker (\pi _1^2)_\ast }
&=& \left\{
\left(
g_{\kappa ,\tau }
\right) _{11}x_2
+\left(
g_{\kappa ,\tau }
\right) _{12}y_2
+\left(
g_{\kappa ,\tau }
\right) _{13}z_2
\right\}
\delta x_2 \\
&& +\left\{
\left( g_{\kappa ,\tau }
\right) _{12}x_2
+\left(
g_{\kappa ,\tau }
\right) _{22}y_2
+\left(
g_{\kappa ,\tau }
\right) _{23}z_2
\right\}
\delta y_2 \\
&& +\left\{
\left(
g_{\kappa ,\tau }
\right) _{13}x_2
+\left(
g_{\kappa ,\tau }
\right) _{23}y_2
+\left(
g_{\kappa ,\tau }
\right) _{33}z_2
\right\}
\delta z_2,
\end{eqnarray*}
where
$\delta x_r=dx_r|_{\ker (\pi _{r-1}^r)_\ast }$,
$\delta y_r=dy_r|_{\ker (\pi _{r-1}^r)_\ast }$,
$\delta z_{r}=dz_r|_{\ker (\pi _{r-1}^r)_\ast }$
for $r=1,2$.

As $3=\sum _{i=0}^3k_i$, by virtue of \eqref{formula2},
it follows $k_i=0$, for $i>2$; i.e., the subalgebra
generated by the invariants of order $<3$, and their
derivatives with respect to the operator $D_t$,
exhausts the algebra of invariants for every order
$i>2$, and the given system is thus complete.
\end{proof}
\begin{theorem}
\label{3_isometries}
Let $(M,g)$ be a $3$-dimensional simply connected
complete Riemannian manifold of class $C^\omega $
such that $\mathfrak{I}(M,g)$ is
a $3$-dimensional group acting transitively on $M$.
Then $(M,g)$ is isometric
to $(\mathfrak{I}^0(M,g),\bar{g})$, $\bar{g}$ being
a left-invariant Riemannian metric. A complete
system of invariants are: The function $t$ and
the three first-order invariants given by
$I_i(j_t^1\sigma )=\omega _i(T^\sigma (t))$,
$1\leq i\leq 3$, $(\omega _1,\omega _2,\omega _3)$
being a basis for the Maurer-Cartan forms.

Moreover, with the same notations and results
as in \emph{\cite{HaLee}}, the metrics $\bar{g}$
are the following:
\begin{enumerate}
\item[\emph{(i)}]
For the group $\widetilde{PSL}(2,\mathbb{R})$
(resp.\  $SU(2) $),
$\bar{g}=\lambda (\omega _1)^2+\mu (\omega _2)^2
+\nu (\omega _3)^2$,
$\lambda ,\mu ,\nu \in \mathbb{R}^+$,
with $\mu >\nu $ (resp.\  $\lambda >\mu >\nu $),
where $(\omega _1,\omega _2,\omega _3)$
is the dual basis of a basis $(X_1,X_2,X_3)$
of the algebra of left-invariant vector fields such that,
\begin{equation*}
\!\!\!
\begin{array}{llll}
\lbrack X_1,X_2]=X_3,
& [X_1,X_3]=-X_2,
& [X_2,X_3]=-X_1,
& G=\widetilde{PSL}(2,\mathbb{R}), \\
\lbrack X_1,X_2]=X_3, & [X_1,X_3]
=-X_2,
& [X_2,X_3]=X_1,
&
G=SU(2).
\end{array}
\end{equation*}
\item[\emph{(ii)}]
For the unimodular solvable group
$\mathbb{R}^2\rtimes _{\varphi }\mathbb{R}$,
with coordinates $(x,y,z)$, and
$\varphi (z)=\operatorname{diag}
(\exp (z),\exp (-z))$, $\omega _1=\exp (-z)dx$,
$\omega _2=\exp (z)dy$, $\omega _3=dz$,
\begin{equation*}
\bar{g}=(\omega _1)^2+(\omega _2)^2
+\nu (\omega _3)^2\text{ or }
\bar{g}=(\omega _1)^2
+\omega _1\otimes \omega _2
+\omega _2\otimes \omega _1
+\mu (\omega _2)^2
+\nu (\omega _3)^2,
\end{equation*}
$\mu >1$, $\nu >0$.
\item[\emph{(iii)}]
For the group $\tilde{E}_0(2)
=\mathbb{C}\rtimes \mathbb{R}$,
$\bar{g}=(\omega _1)^2
+\mu (\omega _2)^2
+\nu (\omega _3)^2$,
$\mu ,\nu \in \mathbb{R}^{+}$,
with $0<\mu <1$,
$\omega _1=\cos (2\pi z)dx+\sin (2\pi z)dy$,
$\omega _2=\sin (2\pi z)dx-\cos (2\pi z)dy$,
$\omega _3=dz$, $x+iy\in \mathbb{C}$,
$z\in \mathbb{R}$.
\item[\emph{(iv)}]
For the non-unimodular solvable group
$G_{c}=\mathbb{R}^2
\rtimes _{\varphi _c}\mathbb{R}$,
 with $c\neq 0$, and
\begin{equation*}
\varphi _c(z)=\exp
\left(
\begin{array}{cc}
0 & -cz \\
z & 2z
\end{array}
\right)
=\left(
\begin{array}{cc}
a_1^1(z) & a_2^1(z) \\
a_1^2(z) & a_2^2(z)
\end{array}
\right) ,\quad z\in \mathbb{R},
\end{equation*}
$\omega _1=a_1^1(-z)dx+a_2^1(-z)dy$,
$\omega _2=a_1^2(-z)dx+a_2^2(-z)dy$,
$\omega _3=dz$, we have
\begin{enumerate}
\item[\emph{(iv.a)}]
If $c<0$ or $c\geq 1$, then
$\bar{g}=(\omega _1)^2
+\mu (\omega _2)^2
+\nu (\omega _3)^2$,
$0<\mu \leq |c|$, $\nu >0$;
if $c=1$, the metrics
$\bar{g}=(\omega _1)^2
+\lambda (\omega _1\otimes \omega _2
+\omega _2\otimes \omega _1)
+\nu (\omega _3)^2 $,
$0<\lambda <1$, $\nu >0$,
must also be included.
\item[\emph{(iv.b)}]
If $0<c<1$, then
\begin{eqnarray*}
\bar{g} &=& \tfrac{1}{2}
\frac{(\mu +1)c-2}{c^2(c-1)}(\omega _1)^2
+\tfrac{1}{2}\frac{\mu -1}{c(c-1)}
\left(
\omega _1\otimes \omega _2
+\omega _2\otimes \omega _1
\right) \\
&&+\tfrac{1}{2}
\frac{\mu -1}{c-1}
(\omega _2)^2+\nu (\omega _3)^2,
\end{eqnarray*}
with $0\leq \mu <1$, $\nu >0$.
\end{enumerate}
\end{enumerate}
\end{theorem}
\begin{proof}
According to Proposition
\ref{proposition_independent_vectors}
the Riemannian manifold $(M,g)$ is isometric
to $(G=\mathfrak{I}^0(M,g),\bar{g})$,
$\bar{g}$ being a left-invariant Riemannian
metric and according to Proposition
\ref{proposition_omega_tilde_H}, $t$
and $I_1,I_2,I_3$ constitute a complete system
of invariants.

(i) Assume $G=\widetilde{PSL}(2,\mathbb{R})$
or $SU(2)$, and let $\mathfrak{X}^L(G)$
(resp.\ $\mathfrak{X}^R(G)$) be the Lie algebra
of left (resp.\ right) invariant vector fields.
As $G$ is simple and $\bar{g}$ is left invariant,
we have (see \cite[(5.1)Theorem]{Gordon},
\cite[Theorem 5]{O}),
\begin{equation*}
\mathfrak{X}^R(G)
\subseteq \mathfrak{i}(G,\bar{g})
\subseteq \mathfrak{X}^L(G)
+\mathfrak{X}^R(G).
\end{equation*}
As a calculation shows, $\mathfrak{X}^L(G)
\cap \mathfrak{X}^R(G)=\{ 0\} $.
Hence $\dim \mathfrak{i}(G,\bar{g})\geq 4$
if and only if there exists a left-invariant
Killing vector field $A\neq 0$. In this case
$(R_{\exp (tA)}
\circ L_{\exp (-tA)})^\ast \bar{g}
=\bar{g}$, and taking derivatives,
\begin{equation}
\label{adA}
\bar{g}
\left(
X,[A,Y]
\right)
+\bar{g}
\left(
[A,X],Y
\right)
=0,
\quad \forall X,Y\in \mathfrak{X}^L(G).
\end{equation}
We follow \cite{Milnor1}.
If $A=\sum _{i=1}^3a_iX_i$, $a_i\in \mathbb{R}$,
then the equation \eqref{adA} is equivalent
to the system $\sum\nolimits_{i=1}^3
a_i(c_{ij}^h\lambda _h+c_{ih}^j\lambda _j)=0 $,
$1\leq h\leq j\leq 3$,
$\bar{g}=\sum _{i=1}^3\lambda _i(\omega _i)^2$,
$c_{jk}^i$ being the structure constants
in the basis $(X_1,X_2,X_3)$.
For $\widetilde{PSL}(2,\mathbb{R})$
the previous system becomes,
$a_1(\lambda _2-\lambda _3)
=a_2(\lambda _3+\lambda _1)
=a_3(\lambda _1+\lambda _2)
=0$; hence $a_2=a_3=0$, and the value
of $a_1$ is necessarily $0$ if and only if
$\lambda _2\neq \lambda _3$. For $SU(2)$
the system is,
$\left(
\lambda _3-\lambda _2
\right)
a_1=\left(
\lambda _3-\lambda _1
\right)
a_2=\left(
\lambda
_2-\lambda _1
\right)
a_3=0$; hence $a_1=a_2=a_3=0$ for
$\lambda _1
\neq \lambda _2
\neq \lambda _3
\neq \lambda _1$.

\medskip

(ii) If $G=\mathbb{R}^2\rtimes _\varphi \mathbb{R}$,
then it is readily seen that
$\mathfrak{X}^R(G)
=\left\langle
Y_1,Y_2,Y_3
\right\rangle $,
with $Y_1=\partial /\partial x$,
$Y_2=\partial /\partial y$,
$Y_3=x\partial /\partial x
-y\partial /\partial y
+\partial /\partial z$. We know
(see \cite[Corollar 4.5]{HaLee})
that $(G,\bar{g})$ is not of constant
curvature.
If $\dim \mathfrak{i}(G,\bar{g})=4$,
then according to Theorem \ref{4_isometries}
there exists a unitary vector field
$Z=f\partial /\partial x
+g\partial /\partial y
+h\partial /\partial z$,
$f,g,h\in C^\infty (x,y,z)$,
in the center of $\mathfrak{i}(G,\bar{g})$.
From $[Y_1,Z]=[Y_2,Z]=0$, it follows
that $f$, $g$, and $h$ depend on $z$ only,
and by imposing $[Y_3,Z]=0$,
one obtains $f(z)=a\exp (z)$,
$g(z)=b\exp (z)$, $h(z)=c$, with
$a,b,c\in \mathbb{R}$, and the condition
$L_Z\bar{g}=0$ (for both classes of metrics
in the statement) is easily proved to imply
$a=b=c=0$; hence $\mathfrak{i}(G,\bar{g})$
coincides with $\mathfrak{X}^R(G)$.

\medskip

(iii) Similarly, in this case we have
$\mathfrak{X}^r(\tilde{E}_0(2))
=\left\langle
Y_1,Y_2,Y_3
\right\rangle $, with
$Y_1=\partial /\partial x$,
$Y_2=\partial /\partial y$,
$Y_3=2\pi y\partial /\partial x
-2\pi x\partial /\partial y
-\partial /\partial z$.
As in the previous case, the assumption
$\dim \mathfrak{i}(\tilde{E}_0(2),\bar{g})=4$
leads one to the following equations
for the vector field $Z$:
$f(z)=a\cos (2\pi z)+b\sin (2\pi z)$,
$g(z)=a\sin (2\pi z)-b\cos (2\pi z)$, $h(z)=c$,
with $a,b,c\in \mathbb{R}$. Taking account
of the fact that $\mu \neq 1$, the condition
$L_Z\bar{g}=0$ then implies $a=b=c=0$.

\medskip

(iv.a) In this case, $\mathfrak{X}^r(G_c)
=\left\langle
Y_1,Y_2,Y_3
\right\rangle $, with
$Y_1=\partial /\partial x$,
$Y_2=\partial /\partial y$,
$Y_3=-cy\partial /\partial x
+(x+2y)\partial /\partial y
+\partial /\partial z$. The assumption
$\dim \mathfrak{i}(G_c,\bar{g})=4$ leads
one to the following equations
for the vector field $Z$:
$f(z)=\alpha a_1^1(z)+\beta a_2^1(z)$,
$g(z)=\alpha a_1^2(z)+\beta a_2^2(z)$,
$h(z)=\gamma $, with
$\alpha ,\beta ,\gamma \in \mathbb{R}$.
By imposing $L_Z\bar{g}=0$ for any one
of the metrics $\bar{g}=(\omega _1)^2
+\mu (\omega _2)^2+\nu (\omega _3)^2$,
$0<\mu \leq |c|$, $\nu >0$,
or $\bar{g}=(\omega _1)^2
+\lambda (\omega _1\otimes \omega _2
+\omega _2\otimes \omega _1)
+\nu (\omega _3)^2$, $0<\lambda <1$,
$\nu >0$, when $c=1$, one concludes $Z=0$.

\medskip

(iv.b) This is case follows by proceeding
as in the previous one.
\end{proof}
\begin{remark}
The other metrics in \cite{HaLee} admit a group
of isometries of dimension either $4$ or $6$.
\end{remark}
The results in the sections \ref{ccc}, \ref{FewIsom},
and \ref{three_dim} allow one to solve the equivalence
problem completely in dimensions $2$ and $3$. In fact,
if $\dim M=2$, then $\dim \mathfrak{i}(M,g)\leq 3$;
if $\dim \mathfrak{i}(M,g)\leq 2$, then the results
in section \ref{FewIsom} apply, and if
$\dim \mathfrak{i}(M,g)=3$, then $(M,g)$ is
of constant curvature and the classification follows
from section \ref{ccc}. Similarly, if
$\dim M=3$, then the cases $\dim \mathfrak{i}(M,g)\leq 3$
also follow from section \ref{FewIsom}, the cases
$\dim \mathfrak{i}(M,g)=3$ or $4$, with $(M,g)$ Riemannian
homogeneous, follow from section \ref{three_dim}, and
the manifold is of constant curvature in the case
$\dim \mathfrak{i}(M,g)=6$. The non-homogeneous case
$\dim M=3$ and $\dim \mathfrak{i}(M,g)=4$ cannot
occurr as the isotropy subgroups of the action of
$\mathfrak{I}(M,g)$ should be $2$-dimensional tori,
which is not possible. Moreover, the non-homgeneous 
case $\dim M=3$ and $\dim \mathfrak{i}(M,g)=3$
is further covered by the following
\begin{proposition}
If $(M,g)$ is a $3$-dimensional Riemannian connected
manifold such that $\mathfrak{i}(M,g)
=\langle X_1,X_2,X_3\rangle $ is
$3$-dimensional but $(X_1)_x,(X_2)_x,(X_3)_x$ are linearly
dependent for every $x\in M$, then the orbits of the isometry
group $G=\mathfrak{I}^0(M,g)$ on $M$ are generically surfaces
of constant signed curvature. Let $U$ be a coordinate
neighbourhood of a generic point $x_0\in M$ with coordinates
$(x^1,x^2,x^3)$, $x^i\in V^i\subset \mathbb{R}$, such that
the foliation of the orbits of $G$ intersected with $U$
are given by $x^1=\mathrm{constant}$. Then $x^1$
is a differential invariant of order zero on $U$.
If $\pi _x\colon T_xU\to T_{x}(G\cdot x)$ denotes
the orthogonal $g_x$-projection, $x\in U$,then the function
$\bar{\varkappa }_0\colon J^1(\mathbb{R},U)\to \mathbb{R}$
defined by
$\bar{\varkappa }_0(j_t^1\sigma )
=\varkappa _0
\left(
\pi _x(\sigma ^\prime (t))
\right) $
is a first-order differential invariant on $U$.
The coordinates $(x^1,x^2,x^3)$ induce a decomposition
 $J^2(\mathbb{R},U)
 =J^2(\mathbb{R},V^1)\times J^2(\mathbb{R},V^2\times V^3)$.
Let $\pi _2$ be the projection onto the second factor.

The function $\bar{\varkappa }_1\colon J^2(\mathbb{R},U)
\to \mathbb{R}$ defined by
$\bar{\varkappa }_1(j_t^2\sigma )
=\varkappa _1(\pi _2(j_t^2\sigma ))$
is a second-order differential invariant on $U$,
where $\varkappa _1 $ is the curvature function
of the surfaces of the foliation.

Finally, the functions $x^1$, $\bar{\varkappa }_0$,
and $\bar{\varkappa }_1$ are a complete system
of differential invariants on $U$.
\end{proposition}
The proof of the proposition above is similar to that
of Proposition \ref{Geometric Meaning} and it will
therefore be omitted.

We should also remark on the fact that the situation
of the previous proposition can happen even in simple
cases. For example, the isometry group of the metric
\begin{eqnarray*}
g & = & (1+x^2)dx^2+(1+y^2)dy^2+(1+z^2)dz^2
+xy\left(
dx\otimes dy+dy\otimes dx
\right) \\
& &
+xz\left(
dx\otimes dz+dz\otimes dx
\right)
+yz\left(
dy\otimes dz+dz\otimes dy
\right)
\end{eqnarray*}
on $\mathbb{R}^3$ is $O(3)$.

\end{document}